\documentclass[a4paper,11pt]{amsart}
\usepackage[utf8]{inputenc}
\usepackage{amsmath}
\usepackage{amssymb}
\usepackage{amsthm}
\usepackage{xcolor}

\usepackage{enumerate}
\usepackage{amssymb, amsmath,graphicx}
\usepackage{todonotes}
\usepackage{verbatim}
\usepackage{float}

\theoremstyle{plain}
\newtheorem{thm}{Theorem}[section]
\newtheorem{theorem}[thm]{Theorem}
\newtheorem{proposition}[thm]{Proposition}
\newtheorem{definition}[thm]{Definition}
\newtheorem{lemma}[thm]{Lemma}
\newtheorem{corollary}[thm]{Corollary}

\newtheorem{challenge}[thm]{Challenge}

\newtheorem{remark}[thm]{Remark}
\newtheorem{example}[thm]{Example}

\newtheorem*{example*}{Example}
\newtheorem*{remark*}{Remark}

\usepackage{enumerate}
\usepackage{amssymb, amsmath}
\usepackage{mathrsfs}
\usepackage{mathbbol}
\usepackage{marginnote}
\usepackage{xcolor}
\usepackage{graphicx}
\usepackage{psfrag}

\usepackage[T1]{fontenc}
\usepackage{graphicx, graphics, wrapfig}
\usepackage{bm}
\usepackage{sidecap}
\usepackage{caption,subcaption}
\def\N{{\mathbb N}}   
 \def\Z{{\mathbb Z}} 
\def\R{{\mathbb R}}

\def\one{{\mathbb 1}}

\def\d{{\sf d}}
\def\X{{\sf X}}

\def\mm{{\sf m}}

\def\c{\mathfrak c}

\def\Dist{{\mathsf D}}
\def\T{{\sf T}}
\def\M{{\sf M}}
\def\g{{\sf g}}
\def\Co{{\mathsf C}}%
\def\Mea{{{\mathcal M}_{0}}}
\def\HK{{\sf H\!\!K}}

\def\Leb{\mathcal{L}}
\def\C{\mathcal{C}}
\def\Lip{\mathrm{Lip}}

\newcommand{\pr}{{\rm{pr}}}

\newcommand{\Ric}{\mathsf{Ric}}
\newcommand{\Sec}{\mathsf{Sec}}
\newcommand{\supp}{\mathrm{supp}}
\newcommand{\Wass}{\mathscr{P}} 

\newcommand{\W}{{\sf{W}}}

\DeclareMathOperator{\diam}{diam}

\newcommand{{\push}}{*}
\DeclareMathOperator{\Cpl}{Cpl}
\DeclareMathOperator{\Id}{Id}

\newcommand{\comments}[1]{}

\newcommand{\bookrefs}[1]{}

\newcommand{\source}[1]{}

\begin{document}
\title[Curvature for Wasserstein and Hellinger-Kantorovich Geometries]{Sectional Curvature\\ for\\ Kantorovich-Wasserstein\\ and\\ Hellinger-Kantorovich Geometries}
\author{Karl-Theodor Sturm}
 \address{University of Bonn\\
 Institute for Applied Mathematics\\
 Endenicher Allee 60\\
 53115 Bonn\\
 Germany}
  \email{sturm@uni-bonn.de}

\maketitle

\begin{abstract} We derive an explicit formula for the sectional curvature of the space $\Mea(\M)$ of finite measures on a Riemannian manifold $\M$. The space $\Mea(\M)$ is equipped with the Hellinger-Kantorovich metric $\HK$. Even in the case $\M=\R^n$, the curvature is comprised of two parts: the `lifted part' is negative, and  the `twisted part'  is positive. It will be analyzed in detail for the multidimensional torus.

Our general approach to sectional curvature in geodesic spaces also leads to new insights into the curvature of the space $\Wass_2(\M)$ of probability measures on $\M$ equipped with the    Kantorovich-Wasserstein metric $\W_2$.
\end{abstract}

	\tableofcontents

\section{Introduction and summary of main results}

We will analyze in detail curvature properties of the space $\Mea=\Mea(\M)$ of finite measures on a (Euclidean or Riemannian) space $\M$. Equipped with the Hellinger-Kantorovich metric $\HK$, the space of measures is a geodesic metric space  of fundamental importance which has attracted a lot of attention in recent years, both from theoretical perspectives  in unbalanced optimal transport problems and from applications, including
computational optimal transport, statistical optimal transport, applications to data science and machine learning,
modeling cellular dynamics, birth-death dynamics for sampling, and
interaction-force transport gradient flows.

Most importantly,  the Hellinger-Kantorovich metric  admits an intuitive interpretation as a relaxation of the Kantorovich-Wasserstein metric $\W_2$ with optional creation and annihilation of mass, the cost of the latter being measured in terms of the Kakutani-Hellinger metric.
For a measure $\mu\in\Mea$, the `analytic tangent space' is given by the Sobolev space  $\T_\mu\Mea=H^1(\mu)$ and the `extended tangent space' is
$\T^{\sf ext}_\mu\Mea=H^1(\mu)\otimes 
\Mea\big({\M\setminus B_{\pi/2}(\supp[\mu])}\big)$.

\medskip

\paragraph{\bf a} We will show that for every $\mu\in\Mea(\M)$ and every pair of tangent vectors $\varphi,\psi\in\mathcal C^\infty_c(\M)\subset H^{1}(\mu)$, the sectional curvature is given as \begin{align*}
 \Sec_\mu(\varphi,\psi)&:=\lim_{t\to0}
\frac3{\Lambda t^4}\Bigg[t^2 |\varphi-\psi|_{H^1(\mu)}^2-
 \HK^2\Big(\exp_\mu^\Mea(t\varphi),\exp_\mu^\Mea(t\psi)\Big)
\Bigg]
\end{align*}
with
$\Lambda:=|\varphi|_{H^1(\mu)}^2\cdot|\psi|_{H^1(\mu)}^2-\langle \varphi,\psi\rangle_{H^1(\mu)}^2$, and 
decomposes
into a lifted and a twisted part
$$
 \Sec_\mu(\varphi,\psi)= \Sec_\mu^\uparrow(\varphi,\psi)+ \Sec_\mu^\nabla(\varphi,\psi).$$
\paragraph{\bf b} The lifted part vanishes if $\dim\M=1$ and otherwise it is given in terms of the sectional curvature of the underlying Riemannian manifold $\M$: 
$$\Sec_\mu^\uparrow\big(\varphi,\psi \big)=\frac1\Lambda \int_{\M}\Big(\Sec^\M_{x}\big(\nabla\varphi,\nabla\psi\big)-1\Big)\cdot \Big[|\nabla\varphi|_x^2\, |\nabla\psi|_x^2-\langle \nabla\varphi,\nabla\psi\rangle_x^2
\Big]\,d\mu(x).$$
\paragraph{\bf c} The twisted part is always nonnegative:
$$\Sec_\mu^\nabla\big(\varphi,\psi \big)\ge0.$$
In the Euclidean case with absolutely continuous $\mu$, it is explicitly given as
\begin{align*}
\Sec_\mu^\nabla(\varphi,\psi)
 &=\frac3\Lambda\, \inf\bigg\{
\int_\M\Big[
|F-\nabla \eta|^2+4\eta^2
\Big]\,d\mu: \ \eta\in
H^1(\R^n,\mu)
\bigg\}\\
&=\frac{12}{\Lambda}\int\Big|\Big(-\Delta_\mu(4-\Delta_\rho)\Big)^{-1/2}\mathrm{div}_{\mu}  F\Big|^2\,d\mu
\end{align*}
with 
$F:=\frac12\big(\nabla^2\varphi\nabla\psi-\nabla^2\psi\nabla\varphi\big)$. Even for 1-dimensional spaces, 
$\Sec_\mu^\nabla$ does not vanish.

\medskip

\paragraph{\bf d} In particular for $\M=\R^n$, there will be a competition between the two parts: $\Sec_\mu^\nabla(\varphi,\psi)\ge0$ and $\Sec_\mu^\uparrow(\varphi,\psi)\le0$.
So far, it was only known that $(\Mea(\R^n),\HK)$ for $n\ge2$ does not have nonnegative curvature in the sense of Alexandrov. We show that it also does not have nonpositive curvature in the sense of Alexandrov.

\medskip

\paragraph{\bf e} More detailed insights are obtained  for the multi-dimensional torus $\M=\mathbb T^n$ and for its normalized volume measure $\mu$. Let $(e_k)_{k\in \Z^n}$ denote the Fourier basis of $L^2(\mathbb T^n,\mu)$. Then 
\begin{align*}\Sec_\mu(e_k,e_\ell)&=\frac{1}{2^n}\,\sum_{\sigma\in\{-1,1\}} S(k^\sigma,\ell)
\end{align*}
with $k^\sigma=(\sigma_pk_p)_{p=1,\ldots,n}\in\Z^n$ and
\begin{align*}S(k,\ell)&=\frac{1}{(4+|k|^2)\,(4+|\ell|^2)}
\bigg(-|k|^2\, |\ell|^2+\langle k^2,\ell^2\rangle
+\frac34\,
\bigg[| k-\ell|^2
-
\frac{\big(|k|^2-|\ell|^2
\big)^2}{4+|k+\ell|^2}
\bigg]
\, 
 \langle k,\ell\rangle^2 \bigg)\end{align*}
 for any pair $k,\ell\in\Z^n$ which is `generic'  in the sense that $k_p\not=\pm\ell_p, k_p\ell_p\not=0$ for all $p=1,\ldots,n$.
 Moreover, $$-\Big(\frac32\Big)^{n}\le \Sec_\mu\big(e_{k},e_{\ell} \big)\le \frac32\, 4^n\, \Big(|k|^2+|\ell|^2\Big)$$
for arbitrary $k,\ell$ with $k\not=\ell$. 
In particular, we prove that
\begin{itemize}
\item there are plenty of cases  with negative sectional curvature:
if $n\ge 2$, $k,\ell\not=0$, and $k_p\ell_p=0$ ($\forall p=1,\ldots,n$), then
$$\Sec_\mu(e_k,e_\ell)=-\frac{|k|^2\,|\ell|^2}{(4+|k|^2)\,(4+|\ell|^2)}\ \in \ \bigg(-1,-\frac1{25}\bigg];$$

\item in average, the sectional curvature will be positive:
for all sufficiently large $C,L$ and  all $k\in\Z^2$ with $|k_1|,|k_2|\ge C$,
$$\sum_{\ell\in \Z^2, |\ell|\le L}\Sec_\mu(e_k,e_\ell)\ge \frac1{1000} \,C^2\,L^2;$$

\item
for all $k\in\Z^2$ with sufficiently large $|k_1|,|k_2|$, the Ricci curvature diverges
$$\mathsf{Ric}_\mu(e_{k}, e_{k}):=\|e_k\|_{H^1(\mu)}^2\cdot\lim_{L\to\infty}\sum_{\ell\in\in \Z^2, |\ell|\le L}\Sec_\mu(e_{k},e_{\ell})=+\infty;$$

\item  for every $k\in\Z^n$ and every $s>2+\frac n2$, the re-normalized Ricci curvature
 $\mathsf{Ric}^{(s)}_\mu(e_{k},e_{k})$
 exists and is finite. 
\end{itemize}

\medskip

\paragraph{\bf f} Our approach to sectional curvature on the space of finite measures is  based on the observation that 
in the case of a smooth Riemannian manifold $(\M,\g)$,  the sectional curvature $\Sec^\M_z(v,w)$  is given in purely metric terms as
\begin{align*}\Sec^\M_z(v,w)&=3\lim_{t\to0}\frac{t^2|v-w|^2-\d^2\big(\exp_z(tv),\exp_z(tw)\big)}{t^4\big(|v|^2|w|^2-\langle v,w\rangle^2\big)}\end{align*}
for every $z\in\M$ and every linearly independent pair of tangent vectors $v,w\in \T_z\M$.
Inspired by this observation, therefore for each geodesic metric space $(\X,\d)$, each $z\in\X$ and each `independent' pair of geodesics $\alpha,\beta$ emanating from $z$ we define
$$\Sec^\X_z(\alpha,\beta):=3\,\lim_{t\to0}
\frac{t^2\d_z^2(\alpha,\beta)-\d^2\big(\alpha_t,\beta_t\big)}{t^4\big(|\dot\alpha|^2|\dot\beta|^2-\langle \alpha,\beta\rangle_z^2\big)}
$$ where
$\d_z(\alpha,\beta):=\limsup_{t\to0}\frac1t \d(\alpha_t,\beta_t)$ denotes the directional distance of the geodesics,
$|\dot\alpha|, |\dot\beta|$ their metric speed,
and
$\langle\alpha,\beta\rangle_z:=\frac12\big(
|\dot\alpha|^2+|\dot\beta|^2-\d_z^2(\alpha,\beta)\big)$.
One easily verifies that for every complete geodesic space $(\X,\d)$,  
\begin{itemize}
\item
$\Sec(\X,\d)\ge\kappa\text{ in Alexandrov sense}\quad \Longrightarrow \quad\forall z\in\X, \alpha,\beta\in  \mathring\T^{\sf dir}_z\X: \ \Sec_z(\alpha,\beta)\ge \kappa,
$
\item
$\Sec(\X,\d)\le\lambda\text{ in Alexandrov sense}\quad \Longrightarrow \quad\forall z\in\X, \alpha,\beta\in  \mathring\T^{\sf dir}_z\X: \ \Sec_z(\alpha,\beta)\le \lambda.
$
\end{itemize}

\medskip

\paragraph{\bf g}

This concept of sectional curvature for geodesic metric spaces will be analyzed in detail for two spaces: the Hellinger-Kantorovich space $(\Mea(\M),\HK)$ of finite measures on $\M$
and the Kantorovich-Wasserstein space $(\Wass_2(\M),\W_2)$ of probability measures on $\M$. In both cases, for convenience $\M$ is assumed to be a smooth Riemannian manifold.

In the latter case, which is significantly easier than the former, we can indeed confirm and extend fundamental insights  of Otto \cite{Ot01} and  Lott \cite{Lott07}. Based on a formal 
 differential-geometric  calculus  on the Kantorovich-Wasserstein space $(\Wass_2(\R^n),\W_2)$ --- which nowadays is called Otto calculus --- Otto \cite{Ot01} deduced the formula 
 $$\Sec^\nabla_\mu(\nabla\varphi,\nabla\psi)=
\frac3\Lambda\cdot\inf\bigg\{\int_{\M} \big|\nabla^2\psi\nabla\varphi-\nabla\eta\big|^2d\mu: \ \eta\in \C_c^\infty(\M)
\bigg\}$$
 with $\Lambda:=|\nabla\varphi|^2_{L^2(\mu)}|\nabla\psi|^2_{L^2(\mu)}-\langle\nabla\varphi,\nabla\psi\rangle_{L^2(\mu)}^2
$ for the sectional curvature of $\Wass_2$ (which in the Euclidean case coincides with 
 its twisted part) and provided an interpretation of it in terms of O'Neill's formula for
 Riemannian submersions.

Lott \cite{Lott07} extended these formal calculations to analyze the
 sectional curvature of $\Wass_2(\M)$ for general Riemannian manifolfds $\M$. 
He deduced the formulas for the twisted part as before for $\M=\R^n$ and in addition now
for the lifted part 
$$\Sec_\mu^\uparrow(\nabla\varphi,\nabla\psi)= \frac1\Lambda\int_\M \Sec^\M_x(\nabla\varphi,\nabla\psi)\cdot\big[|\nabla\varphi|_x^2\, |\nabla\psi|_x^2-\langle \nabla\varphi,\nabla\psi\rangle_x^2
\big]\,d\mu(x).$$
Extensions of these concepts and results 
beyond the setting of absolutely continuous
measures $\mu$ where e.g.~discussed in \cite{Giglimemo} and \cite{lott2017tangent}.

 \medskip

\paragraph{\bf h} The PDE approach of the Otto calculus, however, is limited to sufficiently regular measures $\mu$ and to directions $\nabla\varphi,\nabla\psi$ in the analytic tangent space (which correspond to Monge transport problems) whereas our OT approach also applies to general measures $\mu$ and to general directions $P,Q$ in the geometric tangent space (which correspond to Kantorovich transport problems).
This, for instance, allows us to prove that
\begin{equation*}{\Sec}^{\Wass_2}_{\delta_z}(P,Q)\le 0\end{equation*}
for all $z\in\M$ and all linearly independent $P,Q\in  \mathring\T^{\sf geo}_{\delta_z}\Wass_2$
provided that $\Sec^\M_x\le0$  for all $x\in\M$. 
In particular, ${\Sec}^{\Wass_2}_{\delta_z}=0$ for all $z\in\M$ in the Euclidean case $\M=\R^n$.
\medskip

\paragraph{\bf i}
In the spirit of the Otto calculus, 
formal calculations on the space $(\Mea(\M),\HK)$ have been proposed in \cite{gallouet2018camassa} towards identifying the sectional curvature and establishing a Riemannian submersion property as expressed in
O'Neill's formula. 
Our OT based approach provides rigorous proofs and plenty of novel deep insights.

\medskip

{\it Acknowledgement.}
	The author gratefully acknowledges funding by the Deutsche Forschungsgemeinschaft through the project `Random Riemannian Geometry' within the SPP 2265 \emph{Random Geometric Systems}, 
	through the Hausdorff Center for Mathematics (project ID 390685813), and through  project A05  within the CRC 1720 (project ID 539309657).
	
	Data sharing is not applicable to this article as no datasets were generated or analyzed during the current study.

\section{Sectional curvature in Riemannian and metric geometry}

\begin{lemma}\label{asym-lem}
Assume that $\M={\mathbf M}^{\kappa,n}$ is the $n$-dimensional model space of constant curvature $\kappa\in\R$. Consider two geodesics $(\alpha_t)_{0\le t\le a}$ and $(\beta_t)_{0\le t\le b}$ emanating from the same point $\alpha_0=\beta_0$. Then for $t\searrow0$,
\begin{equation}\d^2(\alpha_t,\beta_t)= 
t^2\big|\dot\alpha_0-\dot\beta_0|^2-\frac\kappa3t^4\Big[
|\dot\alpha_0|^2\,|\dot\beta_0|^2-|\langle\dot\alpha_0,\dot\beta_0\rangle|^2
\Big]+ \mathit{O}(t^6).\end{equation}
More generally, for every $C>0$,
\begin{equation}\d^2(\alpha_s,\beta_t)= 
\big|s\dot\alpha_0-t\dot\beta_0|^2-\frac\kappa3s^2t^2\Big[
|\dot\alpha_0|^2\,|\dot\beta_0|^2-|\langle\dot\alpha_0,\dot\beta_0\rangle|^2
\Big]+ \mathit{O}(s^3t^3)\end{equation}
as $s,t\to0$ with $\frac st+\frac ts\le C$.
\end{lemma}
\begin{proof}
Let $\phi$ denote the angle between $\dot\alpha_0$ and $\dot\beta_0$ and put $v:=|\dot\alpha_0|$, $w:=|\dot\beta_0|$,
$\delta_t:=\d(\alpha_t,\beta_t)$. Then in case $\kappa>0$ by the law of cosines in spherical  geometry,
\begin{align*}
\cos(\sqrt\kappa\delta_t)=\cos(\sqrt\kappa tv)\cos(\sqrt\kappa tw)+\sin(\sqrt\kappa tv)\sin(\sqrt\kappa tw)\cos(\phi).
\end{align*}
Thus
\begin{align*}
-\frac\kappa2\delta_t^2+\frac{\kappa^2}{24}\delta_t^4=&
-\frac\kappa2t^2v^2-\frac\kappa2t^2w^2+\frac{\kappa^2}{4}t^4v^2w^2+\frac{\kappa^2}{24}t^4v^4+\frac{\kappa^2}{24}t^4w^4
\\
&+\Big[\kappa t^2vw-\frac{\kappa^2}6t^4vw^3-\frac{\kappa^2}6t^4v^3 w\Big]
\cos(\phi)+ \mathit{O}(t^6).
\end{align*}
In particular,
$\delta_t^2=t^2(v^2+w^2-2vw\cos\phi)+ \mathit{O}(t^4)$, and thus
\begin{align*}
\delta_t^2=&\frac{\kappa}{12}\Big[t^4\big(v^2+w^2-2vw\cos\phi\big)^2\Big]
\\&
+t^2v^2+t^2w^2-\frac{\kappa}{2}t^4v^2w^2-\frac{\kappa}{12}t^4v^4-\frac{\kappa}{12}t^4w^4
\\
&-\Big[2 t^2vw-\frac{\kappa}3t^4vw^3-\frac{\kappa}3t^4v^3 w\Big]
\cos(\phi)+ \mathit{O}(t^6)\\
&=t^2\Big[v^2+w^2-2vw\cos\phi
\Big]
-\frac\kappa3 t^4v^2w^2(1-\cos^2\phi)+ \mathit{O}(t^6).
\end{align*}
Analogous calculations prove the claim in the cases $\kappa=0$ and $\kappa<0$.

To prove the second claim, let a sequence $(s_i,t_i)_{i\in\N}$ be given  with $s_i/t_i+t_i/s_i\le C$ and put
$\alpha^{(i)}_t:=\alpha_{t s_i/t_i}$.
Then the claim follows from the first claim applied to $\alpha^{(i)}$ and $\beta$ in the place of $\alpha$ and $\beta$ and with error term uniform in $i$,
\begin{align*}
\d^2(\alpha_{t s_i/t_i},\beta_{t})&=
t^2\big|\frac{s_i}{t_i}\dot\alpha_0-\dot\beta_0|^2-\frac\kappa3t^4\Big[
|\frac{s_i}{t_i}\dot\alpha_0|^2\,|\dot\beta_0|^2-|\langle\frac{s_i}{t_i}\dot\alpha_0,\dot\beta_0\rangle|^2
\Big]+ \mathit{O}(t^6),
\end{align*}
and thus
\begin{align*}
\d^2(\alpha_{s_i},\beta_{t_i})
=
\big|{s_i}\dot\alpha_0-t_i\dot\beta_0|^2-\frac\kappa3 s_i^2\,t_i^2\Big[
|\dot\alpha_0|^2\,|\dot\beta_0|^2-|\langle\dot\alpha_0,\dot\beta_0\rangle|^2
\Big]+ \mathit{O}(s_i^3\,t_i^3).
\end{align*}
\end{proof}

\begin{definition} Given a geodesic space $(\X,\d)$ we define  $\T^{\sf dir}_z\X$, its \emph{directional tangent space} at $z\in\X$, as the completion of the metric space
$$\mathring\T^{\sf dir}_z\X:=\big\{\gamma=(\gamma_t)_{0\le t\le \tau} \textrm{ geodesic in }\X, \tau>0, \gamma_0=z\big\}\big/\sim$$
where 
$$\d_z(\alpha,\beta):=\limsup_{t\to0}\frac1t \d(\alpha_t,\beta_t)$$
and $\alpha\sim\beta$ if $\d_z(\alpha,\beta)=0$.
\end{definition}

$\mathring\T^{\sf dir}_z\X$ (as well as its completion) is a cone. Given $\alpha\in \mathring\T^{\sf dir}_z\X$ and $r>0$, the tangent vector $r\alpha\in \mathring\T^{\sf dir}_z\X$ is defined by $(r\alpha)_t:=\alpha_{rt}$. 

For $\alpha,\beta\in \mathring\T^{\sf dir}_z\X$ we define
\begin{align*}\langle\alpha,\beta\rangle_z:=\frac12\big(
|\dot\alpha|^2+|\dot\beta|^2-\d_z^2(\alpha,\beta)\big)=|\dot\alpha|\, |\dot\beta|\, \cos\Big(\angle_z(\alpha,\beta)\Big)
\end{align*} with the upper angle
$$ \angle_z(\alpha,\beta):=\limsup_{t\to0} \angle_z(\alpha_t,\beta_t)$$
 as considered in \cite{BBI}.
We say that $\alpha,\beta\in \mathring\T^{\sf dir}_z\X$ are \emph{linearly independent} if
$$\angle_z(\alpha,\beta)\not\in\{0,\pi\}.$$
The latter property is equivalent to $$\d_z
(\alpha,\beta)\not\in\big\{ |\dot\alpha|+ |\dot\beta|, \big| |\dot\alpha|- |\dot\beta| \big|\big\},
$$
as well as to
$$\langle \alpha,\beta\rangle_z<|\dot\alpha|\,|\dot\beta|.$$

Given a function $F$ on $(0,\infty)^2$ we say that the limit
$$\lim_{\tiny\begin{array}{c}
{s,t\to0}\\{t/s+s/t\textrm{ bdd}}\end{array}}\!\!\!\!\!\!
 F(s,t)\ \textrm{exists and equals }L\in\R$$
if 
$\lim_{\ell\to\infty} F(s_\ell,t_\ell)=L$ for every sequence $(s_\ell,t_\ell)_{\ell\in\N}$ in  $(0,\infty)^2$ with $\sup_\ell(s_\ell/t_\ell+t_\ell/s_\ell)<\infty$ and $\lim_{\ell\to\infty}s_\ell=\lim_{\ell\to\infty}t_\ell=0$. 

Similarly with $\limsup_{\ell\to\infty}F(s_\ell,t_\ell)$ and $\liminf_{\ell\to\infty}F(s_\ell,t_\ell)$ in the place of $\lim_{\ell\to\infty}F(s_\ell,t_\ell)$.

\begin{definition} Given linearly independent $\alpha,\beta\in  \mathring\T^{\sf dir}_z\X$ and $\kappa,\lambda\in\R$, we say that
$$\overline\Sec_z(\alpha,\beta)\le\lambda\quad:\Longleftrightarrow\quad
3\!\!\!\limsup_{\tiny\begin{array}{c}
{s,t\to0}\\{t/s+s/t\textrm{ bdd}}\end{array}}\!\!\!
\frac{\d_z^2(s\alpha,t\beta)-\d^2\big(\alpha_s,\beta_t\big)}{s^2\,t^2\big(|\dot\alpha|^2|\dot\beta|^2-\langle \alpha,\beta\rangle_z^2\big)}
\le\lambda
$$
and, similarly,
$$\underline\Sec_z(\alpha,\beta)\ge \kappa\quad:\Longleftrightarrow\quad
3\!\!\!\liminf_{\tiny\begin{array}{c}
{s,t\to0}\\{t/s+s/t\textrm{ bdd}}\end{array}}\!\!\!
\frac{\d_z^2(s\alpha,t\beta)-\d^2\big(\alpha_s,\beta_t\big)}{s^2\,t^2\big(|\dot\alpha|^2|\dot\beta|^2-\langle \alpha,\beta\rangle_z^2\big)}
\ge\kappa.
$$
We say that $\Sec_z(\alpha,\beta)=\lambda$ if the limit
$$
3\!\!\!\lim_{\tiny\begin{array}{c}
{s,t\to0}\\{t/s+s/t\textrm{ bdd}}\end{array}}\!\!\!
\frac{\d_z^2(s\alpha,t\beta)-\d^2\big(\alpha_s,\beta_t\big)}{s^2\,t^2\big(|\dot\alpha|^2|\dot\beta|^2-\langle \alpha,\beta\rangle_z^2\big)}
$$
exists and equals $\lambda$.
\end{definition}

\begin{theorem}\label{alex-vs-sec} Assume that $(\X,\d)$ is a complete geodesic space  and consider any $z\in\X$. \begin{enumerate}[\rm (i)]
 \item If $(\X,\d)$ has curvature $\ge\kappa$ in the sense of Alexandrov in a neighborhood of $z$ then 
 $$\underline\Sec_z(\alpha,\beta)\ge \kappa$$
 for all linearly independent $\alpha,\beta\in  \mathring\T^{\sf dir}_z\X$.
\item If $(\X,\d)$ has curvature $\le\lambda$ in the sense of Alexandrov in a neighborhood of $z$ then 
 $$\overline\Sec_z(\alpha,\beta)\le \lambda$$
 for all linearly independent $\alpha,\beta\in  \mathring\T^{\sf dir}_z\X$.
\end{enumerate}
\end{theorem}

\begin{proof} In both cases,  for any pair of geodesics $(\alpha_s)_{s\in[0,a]}$ and $(\beta_t)_{t\in[0,b]}$ emanating from $z$, the limit
$\Dist_z(\alpha,\beta)=\lim_{t\to0}\frac1t \d(\alpha_t,\beta_t)$
  exists \cite{BBI}. Assuming now that the curvature is $\ge\kappa$ (or $\le\kappa$, resp.), by triangle comparison 
$$\d(\alpha_s,\beta_t) \stackrel{(\ge)}\le \Phi_{\kappa,s,t, |\dot\alpha|, |\dot\beta|}\Big(\d_z(\alpha,\beta)\Big)$$
where 
$ \Phi_{\kappa,\ldots}(.)$ stands for the formula which provides the exact value in the model space of constant curvature $\kappa$, cf.~Lemma \ref{asym-lem}. The claim thus follows by comparison with the model case.
\end{proof}

\begin{challenge} In both cases, find appropriate additional assumptions the converse implication to hold true.
\end{challenge}

\begin{remark} For any geodesic space $(\X,\d)$,
$$\textrm{NPC in  sense of Alexandrov}\quad\stackrel{\not\Leftarrow}\Longrightarrow\quad\textrm{NPC in  sense of Busemann}\quad\stackrel{\not\Leftarrow}\Longrightarrow\quad\textrm{NPC in above sense.}$$
The fact that the converse implications do not hold is illustrated by the following examples: 1) a Banach space which is not Hilbertian, and 2)  the flat torus.
\end{remark}

\begin{corollary}
For every Riemannian manifold $(\M,\g)$, every $z\in\M$ and every linearly independent pair of tangent vectors $v,w\in \T_z\M$, the sectional curvature $\Sec_z(v,w)$ is given by
\begin{align*}\Sec_z(v,w)&=3\lim_{t\to0}\frac{t^2|v-w|^2-\d^2\big(\exp_z(tv),\exp_z(tw)\big)}{t^4\big(|v|^2|w|^2-\langle v,w\rangle^2\big)}\\
&=3\!\!\!\lim_{\tiny\begin{array}{c}
{s,t\to0}\\{t/s+s/t\textrm{ bdd}}\end{array}}\!\!\!
\frac{|sv-tw|^2-\d^2\big(\exp_z(
sv),\exp_z(tw)\big)}{s^2\,t^2\big(|v|^2|w|^2-\langle v,w\rangle^2\big)}
\end{align*}
\end{corollary}
Note that $|v|^2|w|^2-\langle v,w\rangle^2\ge0$ for all $v,w\in \T_z\M$ and equality holds if and only if $v$ and $w$ are linearly dependent, i.e.~$v=sw$ for some $s\in\R$  or $w=0$.

\begin{remark} Whenever the sectional curvature on a geodesic space $(\X,\d)$
$$\Sec_z(\alpha,\beta):=
3\!\!\!\lim_{\tiny\begin{array}{c}
{s,t\to0}\\{t/s+s/t\textrm{ bdd}}\end{array}}\!\!\!
\frac{\d_z^2(s\alpha,t\beta)-\d^2\big(\alpha_s,\beta_t\big)}{s^2\,t^2\big(|\dot\alpha|^2|\dot\beta|^2-\langle \alpha,\beta\rangle_z^2\big)}
$$
exists, it coincides with the simpler expression
$$\dot\Sec_z(\alpha,\beta):=3\, \lim_{t\to0}
\frac{t^2\d_z^2(\alpha,\beta)-\d^2\big(\alpha_t,\beta_t\big)}{t^4\big(|\dot\alpha|^2|\dot\beta|^2-\langle \alpha,\beta\rangle_z^2\big)}.
$$
Thus the (existence of the) latter limit could  be used for an alternative definition of sectional curvature. The only disadvantage of this simpler approach would be that scaling only holds true in the restricted form of $\dot\Sec_z(a\alpha,a\beta)=\dot\Sec_z(\alpha,\beta)$ for all $a>0$ whereas $\Sec_z(a\alpha,b\beta)=\Sec_z(\alpha,\beta)$ for all $a,b>0$.
\end{remark}

\begin{remark} Another option for a synthetic definition of sectional curvature, again based on the observation of Lemma \ref{asym-lem}, is to consider for every $z\in \X$ and all linearly independent $\alpha,\beta\in  \mathring\T^{\sf dir}_z\X$,
$${\sf sec}_z(\alpha,\beta):=-\frac3{2\Lambda}\,\frac{d^2}{dt^2}\Big|_{t=0} \d^2\big(\alpha_{\sqrt t},\beta_{\sqrt t}\big)$$
 with $\Lambda:=|\dot\alpha|^2|\dot\beta|^2-\langle \alpha,\beta\rangle_z^2$. This leads to a similar, albeit not equivalent concept of sectional curvature on the geodesic space $(\X,\d)$.
\end{remark}

\section{Sectional curvature  of $\Wass_2(\M)$ 
}

Let us apply these concepts now to the (various) tangent space(s) of  $(\Wass_2(\M),\W_2)$. Here and in the sequel $(\M,\g)$ is a Riemannian manifold, smooth but not necessarily complete. We assume that
$$\big|\Sec_x\big|\le  K 
$$
for suitable 
$K$ and all $x$.

\subsection{Geometric and analytic tangent spaces}
We summarize some basic concepts and results from the work \cite{Gigli11}. 

\subsubsection{The analytic tangent space}
Let $\Gamma(\T\M)=\{V: \M\to \T\M \textrm{ Borel}: V(x)\in \T_x\M \ (\forall x)\}$ and 
let $L_{\T\M}^2(\mu):=\{V\in \Gamma(\T\M): \int|V|^2d\mu<\infty\}$ denote the space of $L^2$ vector fields. 

\begin{definition} We define the \emph{analytic (or deterministic) tangent space} of $(\Wass_2(\M),\W_2)$ at $\mu$ by
$$\T^{\sf ana}_\mu\Wass_2:=\overline{\big\{\nabla\varphi: \ \varphi\in \C^\infty_c(\M)\big\}}^{L^2(\mu)}\subset L_{\T\M}^2(\mu).$$
For $V,W\in  \T^{\sf ana}_\mu\Wass_2$ we put
$\exp_\mu(V):=\exp(V)_*\mu$
and
$\langle V,W\rangle_\mu=\int_\M \langle V,W\rangle_x\,d\mu(x)$
with $\langle V,W\rangle_x:= \g_x(V,W)$.

\end{definition} 

\begin{remark}
$\T^{\sf ana}_\mu\Wass_2=\overline{\big\{\nabla\varphi: \ \varphi\in \C^\infty_c(\M), \ t\varphi \textrm{ is $\c$-convex for some }t>0\big\}}^{L^2(\mu)}.$
\end{remark}
\begin{example} Consider $\mu=\delta_z$. Then
$\T^{\sf ana}_\mu\Wass_2=L^2_{\T\M}(\mu)\cong\T_z\M$.
\end{example}

\subsubsection{Tangent measures}
Recall that
$\T\M=\{(x,v): x\in \M, v\in \T_x\M\}$.
Given $\mu\in\Wass(\M)$, denote by $\Wass(\T\M)_\mu$ the set of  $P\in \Wass(\T\M)$ with first marginal $\mu$, that is, $(\pr_1)_* P=\mu$.
We define  the multiplication with positive scalars by
$$t\, P:= (\theta_t)_*P, \qquad \theta_t: (x,v)\mapsto (x,tv),$$
and the exponential map $\exp: \Wass_2(\T\M)_\mu\to \Wass_2(\M)$ by
$$\exp(P) = \Theta_*P, \qquad \Theta: (x,v)\mapsto \exp_xv.$$
For $\mu\in\Wass_2(\M)$ we denote by $\Wass_2(\T\M)_\mu$ the subset of $P\in \Wass_2(\T\M)_\mu$ with finite  modulus
$$|P|_\mu:=\left(\int_{\T\M}|v|^2dP(v)\right)^{1/2}.$$
Then obviously $|tP|_\mu=t\, |P|_\mu$ for all $t\ge0$, and the family 
$p_t:=\exp(tP)$ for $t\ge0$ obtained by pushing forward $P$ with the maps $\Theta_t: (x,v)\mapsto \exp_x(tv)$ defines a `pseudo geodesic' (i.e.~looks like a geodesic but is not necessarily minimizing) in $\Wass_2\M$.

Every $P\in \Wass(\T\M)_\mu$ can be decomposed into $\mu$ and the disintegration of $P$ w.r.t.~$\mu$ such that
$$\int_{\T\M}f(v)dP=\int_\M\int_{\T_x\M}f(v) dP_x(v)d\mu(x).$$
We define a transport distance on $\Wass_2(\T\M)_\mu$ by 
$$\W_\mu(P,Q):=\inf\Bigg\{ \int_M \int_{\T_x\M\times\T_x\M} |v-w|_x^2dN(x,v,w): \ N\in\Cpl_\mu(P,Q)
\Bigg\}^{1/2}$$
where $\Cpl_\mu(P,Q)$ denotes the set of all probability measures $N$ on $\{(x,v,w): x\in\M, v\in\T_x\M, w\in\T_x\M\}$ with marginals
 $$(\pr_1)_*N=\mu, \quad (\pr_1,\pr_2)_*N=P, \quad (\pr_1,\pr_3)_\mu=Q.$$
 As in the classical case, continuity and nonnegativity of the cost function imply that always a minimizer $N$ exists.  It is easy to see that $\W_\mu$ defines a complete, separable metric on $\Wass_2(\T\M)_\mu$.
 Moreover, it can equivalently be reformulated as 
 $$\W_\mu(P,Q):=\inf\Bigg\{ \int_M \int_{\T_x\M\times\T_x\M} |v-w|_x^2dN_x(v,w)\,d\mu(x): \ N_x\in\Cpl(P_x,Q_x)
\Bigg\}^{1/2}$$
for $P,Q\in \Wass_2(\T\M)_\mu$ (where $N_x$ has to be chosen as a measurable function of $x$). 
Indeed, the kernel $N_x\in\Cpl(P_x,Q_x)$ can always be chosen as the disintegration w.r.t.~the first marginal $d\mu(x)$ of the minimizer $N\in\Cpl_\mu(P,Q)$ in the previous formulation.

\begin{definition} The \emph{geometric (or probabilistic) tangent space} $\T^{\sf geo}_\mu\Wass_2$ is the closure of 
$$\mathring\T^{\sf geo}_\mu\Wass_2:=\Big\{P\in \Wass_2(\T\M)_\mu: \  {\exists C,t>0:
|v|\le C \textrm{ for $P$-a.e.~$v$ and }}\W_2(\mu,\exp(tP))=t |P|_\mu
\Big\}$$
in $(\Wass_2(\T\M)_\mu,\W_\mu)$. \end{definition}

\subsubsection{Embeddings and identifications}

\begin{theorem} For every $\mu\in\Wass_2(\M)$, the analytic tangent space is always isometrically embedded into the geometric tangent space.
The
isometric embedding is given by
$$ \T^{\sf ana}_\mu\Wass_2 \hookrightarrow \T^{\sf geo}_\mu\Wass_2: \quad 
V\in L^2_{\T\M}\mapsto dP(x,v):=d\delta_{V(x)}(v)d\mu(x)\in \Wass_2(\T\M)_\mu.
$$ 
\end{theorem}

\begin{theorem}  If $\mu\ll\mm$ (or, more generally, if $\mu$ is regular), then the analytic and the geometric tangent space are isometric:
 $$ \T^{\sf ana}_\mu\Wass_2\cong\T^{\sf geo}_\mu\Wass_2.$$
\end{theorem}

\begin{theorem} If $\M$ is bounded, then for each $\mu\in\Wass_2(\M)$ the  geometric tangent space $\T^{\sf geo}_\mu\Wass_2$ is isometric to the directional tangent space
  $\T^{\sf dir}_\mu\Wass_2$ as discussed in the previous section. \end{theorem}

\subsection{Sectional curvature and its decomposition into twisted and lifted parts}

\begin{definition}\label{def-sec-wass} For  $\mu\in\Wass_2(\M)$  and linearly independent $P,Q\in   \T^{\sf geo}_\mu\Wass_2$ we define
$$\Sec_\mu(P,Q):=
\!\!\!\lim_{\tiny\begin{array}{c}
{s,t\to0}\\{t/s+s/t\textrm{ bdd}}\end{array}}\!\!\!
\frac3{\Lambda\,s^2\,t^2}\Big[
\W^2_\mu(s P,tQ)
-\W_2^2\big(\exp(sP),\exp(tQ)\Big]
$$
provided the limit exists.
Here and in the sequel

$$\Lambda:=\Lambda_\mu(P,Q):=|P|_{\mu}^2\,|Q|_{\mu}^2-\langle P,Q\rangle_\mu^2$$
with 
\begin{align*}\langle P,Q\rangle_\mu&:=\frac12\big(|P|^2_{\mu}+|Q|^2_{\mu}-\W^2_\mu(P,Q)\big)\\
&=
\sup\bigg\{ \int_M \int_{\T_x\M\times\T_x\M}\langle v,w\rangle_x dN_x(v,w)\,d\mu(x): \ N_x\in\Cpl(P_x,Q_x)
\bigg\}
\end{align*}
and
$$\W_\mu(sP,tQ):=\inf\Bigg\{ \int_M \int_{\T_x\M\times\T_x\M} |sv-tw|_x^2dN_x(v,w)\,d\mu(x): \ N_x\in\Cpl(P_x,Q_x)
\Bigg\}^{1/2}.$$
Moreover, we put
$$\underline{\Sec}_\mu(P,Q):=\liminf_{s,t\to0}\ldots, \qquad\quad \overline{\Sec}_\mu(P,Q):=\limsup_{s,t\to0}\ldots$$
\end{definition}

\begin{remark} For all linearly independent $P,Q$ and all $a,b>0$,
\begin{align*}\Sec_\mu(aP,bQ)&=
\!\!\!\lim_{\tiny\begin{array}{c}
{s,t\to0}\\{t/s+s/t\textrm{ bdd}}\end{array}}\!\!\!
\frac3{\Lambda_\mu(aP,bQ)\,s^2\,t^2}\Big[
\W_\mu^2\big(s(aP),t(bQ)\big)
-\W_2^2\big(\exp(s(aP)),\exp(t(bQ))\Big]\\
&=
\!\!\!\lim_{\tiny\begin{array}{c}
{s,t\to0}\\{t/s+s/t\textrm{ bdd}}\end{array}}\!\!\!
\frac3{\Lambda_\mu(P,Q)\,(sa)^2\,(tb)^2}
\Big[
\W^2_\mu\big((sa)P,(tb)Q\big)
-\W_2^2\big(\exp((sa)P),\exp((tb)Q)\Big]\\
&=\Sec_\mu(P,Q).\end{align*}
\end{remark}

\begin{remark} Let us consider these concepts for the particular case of tangent measures
of the form $P_x=\delta_{V_x}, Q_x=\delta_{W_x}$ with $V,W\in \T^{\sf ana}_\mu\Wass_2\subset L^2_{\T\M}(\mu)$.
Recall that $|V|_{\mu}^2=\int_\M |V|_x^2\,d\mu(x)$ and
$\langle V,W\rangle_\mu=\int_\M \langle V,W\rangle_x\,d\mu(x)$,
and put
$$\Lambda:=\Lambda_\mu(V,W):=|V|_{\mu}^2|W|_{\mu}^2-\langle V,W\rangle_\mu^2.$$
Then for $\mu\in\Wass_2(\M)$  and linearly independent $V,W\in  \T^{\sf ana}_\mu\Wass_2$
we have
$$\Sec_\mu(V,W)=\!\!\!\lim_{\tiny\begin{array}{c}
{s,t\to0}\\{t/s+s/t\textrm{ bdd}}\end{array}}\!\!\!
\frac3{\Lambda\, s^2\,t^2}\bigg[|sV-tW|_\mu^2-
\W_2^2\big(\exp_\mu(sV),\exp_\mu(tW)\big)\bigg]$$
provided the limit exists 
and generally
$$\underline{\Sec}_\mu(V,W):=\liminf_{s,t\to0}\ldots, \qquad\quad \overline{\Sec}_\mu(V,W):=\limsup_{s,t\to0}\ldots$$
\end{remark}

\begin{theorem}\label{sec-sec} For 
arbitrary $\mu$,
linearly independent  $P,Q\in  \mathring\T^{\sf geo}_\mu\Wass_2$ and any  optimal coupling $N_x$ of $P_x$ and $Q_x$,
\begin{equation}\label{sec-dec}
\Sec_\mu(P,Q)= \Sec_\mu^\uparrow(P,Q)+\Sec_\mu^\nabla(P,Q)\end{equation}
with
\begin{equation}\label{sec-lift-def}
\Sec_\mu^\uparrow(P,Q):=\frac1\Lambda \int_\M \int_{\T_x\M\times \T_x\M}  \Sec_x(v,w)\cdot\big[|v|^2\, |w|^2-\langle v,w\rangle^2
\big]\,dN_x(v,w)\,d\mu(x)
\end{equation}
where $\Lambda:=\Lambda_\mu(P,Q):=|P|_{\mu}^2\,|Q|_{\mu}^2-\langle P,Q\rangle_\mu^2$,
and
\begin{align*}
\Sec_\mu^\nabla(P,Q):=
\!\!\!\lim_{\tiny\begin{array}{c}
{s,t\to0}\\{t/s+s/t\textrm{ bdd}}\end{array}}\!\!\!
\frac3{\Lambda\,s^2\,t^2}\bigg[\int_\M\int_{\T_x\M\times \T_x\M}
\d^2\Big(\exp_x(sv),&\exp_x(tw)\Big)\,dN_x(v,w)\,d\mu(x)\\
&
-
\W_2^2\Big(\exp(sP),\exp(tQ)\Big)
\bigg]\ge0.\end{align*}
More precisely, the limit defining $\Sec_\mu(P,Q)$ exists if and only if the limit defining $\Sec_\mu^\nabla(P,Q)$ exists, and in this case \eqref{sec-dec} holds.

More generally,
$$\underline\Sec_\mu(P,Q)= \Sec_\mu^\uparrow(P,Q)+\underline\Sec_\mu^\nabla(P,Q)$$ and
$$\overline\Sec_\mu(P,Q)= \Sec_\mu^\uparrow(P,Q)+\overline\Sec_\mu^\nabla(P,Q).$$

In the definition of $\Sec_\mu^\uparrow(P,Q)$, we use the convention that
$$ \Sec_x(v,w)\cdot\big[|v|^2\, |w|^2-\langle v,w\rangle^2\big]:=0$$
if $v$ and $w$ are linearly dependent.
\end{theorem}

\begin{proof} Let linearly independent  $P,Q\in  \mathring\T^{\sf geo}_\mu\Wass_2$ be given. Then 
$|v|\le C$ and $|w|\le C$ for $P$-a.e.~$v$ and $Q$-a.e.~$w$. Moreover, 
$\big|\Sec_x\big|\le K$
 by our standing assumption on $\M$.
Thus the integral on the RHS of \eqref{sec-lift-def} is always well-defined and finite.

Moreover, the integrand is bounded uniformly in $t\in(0,1)$ and thus by dominated convergence,
 \begin{align*}
 \Sec_\mu^\uparrow(P,Q)&=
 \frac1\Lambda\int_\M\int_{\T_x\M\times \T_x\M} 
\Sec_x(v,w)\cdot\big[|v|^2\, |w|^2-\langle v,w\rangle^2
\big]
\,dN_x(v,w)\,d\mu(x)\\
&=
\frac1\Lambda \int_\M\int_{\T_x\M\times \T_x\M} 
\!\!\!\lim_{\tiny\begin{array}{c}
{s,t\to0}\\{t/s+s/t\textrm{ bdd}}\end{array}}\!\!\!
\frac{3}{s^2\,t^2}\Bigg[|sv-tw|^2-
\d^2\Big(\exp_x(sv),\exp_x(tw)\Big)
\Bigg]\,dN_x(v,w)\,d\mu(x)\\
&=
\!\!\!\lim_{\tiny\begin{array}{c}
{s,t\to0}\\{t/s+s/t\textrm{ bdd}}\end{array}}\!\!\!
\frac3{\Lambda s^2\,t^2}\Bigg[
\W^2_\mu(sP,tQ)-
 \int_\M\int_{\T_x\M\times \T_x\M}
\d^2\Big(\exp_x(sv),\exp_x(tw)\Big)\,dN_x(v,w)\,d\mu(x)
\Bigg]\\
&=
{\underline\Sec_\mu(P,Q)-\underline\Sec_\mu^\nabla(P,Q)}\\
\end{align*}
and similarly
$ \Sec_\mu^\uparrow(P,Q)=\overline\Sec_\mu(P,Q)-\overline\Sec_\mu^\nabla(P,Q)$.

The
nonnegativity of $\underline\Sec_\mu^\nabla(P,Q)$ follows from the fact that the push forward of the measure $dN(v,w)$ under the map $(v,w)\mapsto (\exp(sv),\exp(tw))$
is a coupling of $\exp(sP)$ and $\exp(tQ)$.
\end{proof}

\begin{remark}
For  linearly independent $V,W\in  \mathring\T^{\sf ana}_\mu\Wass_2(\M)$, 
$$\Sec_\mu(V,W)= \Sec_\mu^\uparrow(V,W)+\Sec_\mu^\nabla(V,W)$$
with
$$\Sec_\mu^\uparrow(V,W)= \frac1\Lambda\int_\M \Sec_x(V,W)\cdot\big[|V|_x^2\, |W|_x^2-\langle V,W\rangle_x^2
\big]\,d\mu(x)$$
and
$$\Sec_\mu^\nabla(V,W)=\!\!\!\lim_{\tiny\begin{array}{c}
{s,t\to0}\\{t/s+s/t\textrm{ bdd}}\end{array}}\!\!\!
 \frac3{\Lambda\, s^2\,t^2}\bigg[\int_\M
\d^2\Big(\exp_x(sV),\exp_x(tW)\Big)d\mu(x)-
\W_2^2\Big(\exp_\mu(sV),\exp_\mu(tW)\Big)
\bigg].$$
Note that  $\underline{\Sec}_\mu^\nabla(V,W)\ge0$ and thus always $\underline{\Sec}_\mu(V,W)\ge \Sec_\mu^\uparrow(V,W)$.
\end{remark}

\begin{remark}
Put $R_x(v,w):=R_x(v,w,v)w=\Sec_x(v,w)\, \big[ |v|^2|w|^2-\langle v,w\rangle^2\big]$ on $(\T_x\M)^2$ and analogously $R_\mu$ on $(\T_\mu\Wass)^2$. Then for all $V,W$,
$$\underline{R}_\mu(V,W)\ge \int_\M R_x(V,W)\,d\mu(x).$$
\end{remark}

\subsection{Lifting of curvature bounds}
\begin{remark} If $\Sec_x\ge\kappa$ for all $x$ then for all $\mu$ 
and all linearly independent $P,Q\in\mathring\T^{\sf geo}_\mu\Wass_2$,
\begin{equation}\label{K-vs-k}
\underline\Sec_\mu(P,Q)\ge \kappa\cdot \frac{
\int_\M \int_{\T_x\M\times \T_x\M}  \big[|v|^2\, |w|^2-\langle v,w\rangle^2
\big]\,dN_x(v,w)\,d\mu(x)
}{|P|_{\mu}^2|Q|_{\mu}^2-\langle P,Q\rangle_\mu^2},\end{equation}
and similarly for all linearly independent $V,W\in\mathring\T^{\sf ana}_\mu\Wass_2$,
$$\underline\Sec_\mu(V,W)\ge \kappa\cdot \frac{\int_\M \big[|V|_x^2\, |W|_x^2-\langle V,W\rangle_x^2
\big]\,d\mu(x)}{|V|_{\mu}^2|W|_{\mu}^2-\langle V,W\rangle_\mu^2}.$$
There is no straightforward  way to estimate the ratio on the RHS, besides the fact that it is nonnegative if $\kappa$ is nonnegative.
\end{remark}

\begin{corollary}\label{sec-pos} Assume that $\Sec_x\ge0$  for all $x\in\M$. Then 
$$\underline{\Sec}_\mu(P,Q)\ge 0$$
for all $\mu\in\Wass_2(\M)$ and all linearly independent $P,Q\in  \mathring\T^{\sf geo}_\mu\Wass_2$.
\end{corollary}

One way to control the  ratio on the right hand side of \eqref{K-vs-k} is to impose a bound on the variation of the lengths 
of tangent vectors under $P$.

\begin{proposition} Assume that $\Sec_x\ge -\kappa$ for some $\kappa\ge0$.
Let $\mu\in\Wass_2(\M)$ and  linearly independent $P,Q\in  \mathring\T^{\sf geo}_\mu\Wass_2$ be given.
\begin{enumerate}[\rm (i)]
 \item
If $|v|=|P|_\mu$ for $P$-a.e. $v$ then
$$\Sec_\mu(P,Q)\ge -\kappa.$$
\item
If $\langle P,Q\rangle_\mu=0$ and
$|v|\le L\cdot |P|_\mu$  for $P$-a.e. $v$
for some $L$. Then
$$\Sec_\mu(P,Q)\ge -L^2\,\kappa.$$
\end{enumerate}
\end{proposition}
 
 \begin{proof} (i) Under the given assumption,
 \begin{align*}
& \int_\M \int_{\T_x\M\times \T_x\M}  \big[|v|^2\, |w|^2-\langle v,w\rangle^2
\big]\,dN_x(v,w)\,d\mu(x)\\
&\le
 |P|_\mu^2 \cdot\int  |w|^2dN_x(v,w)d\mu(x)- \int \langle v,w\rangle_x^2dN_x(v,w)
\,d\mu(x)\\
&\le
 |P|_\mu^2 \,  |Q|^2_\mu- \left(\int \langle v,w\rangle_xdN_x(v,w)
\,d\mu(x)\right)^2\\
&=
 |P|_\mu^2 \,  |Q|^2_\mu-  \langle P,Q\rangle_\mu^2.
\end{align*}

(ii) Similarly 
 \begin{align*}
& \int_\M \int_{\T_x\M\times \T_x\M}  \big[|v|^2\, |w|^2-\langle v,w\rangle^2
\big]\,dN_x(v,w)\,d\mu(x)\\
&\le
 L^2\,|P|_\mu^2 \cdot\int  |w|^2dN_x(v,w)d\mu(x)\\
 &=
L^2\, 
 |P|_\mu^2 \,  |Q|^2_\mu=
 L^2\,\Big(|P|_\mu^2 \,  |Q|^2_\mu-  \langle P,Q\rangle_\mu^2\Big).
\end{align*}

 \end{proof}

\begin{corollary} Assume that $\Sec_x\ge-\kappa$  (for all $x$)  for some $\kappa\ge0$ and that $V$ has constant length, i.e.
$|V|_x=\ell$  (for $\mu$-a.e. $x$) for  some $\ell$. Then
$$\underline{\Sec}_\mu(V,W)\ge- \kappa.$$
\end{corollary}

\begin{example} Assume $\mu\ll\mm$ and  $\varphi(x):=\d(x,A)$ for some $A\subset\M$ with $\mu(A)=0$. Then  $V=\nabla \varphi$ satisfies the assumption in the preceding Corollary.
\end{example}

\begin{corollary} Assume that $\Sec_x\ge-\kappa$ ($\forall x$) for some $\kappa\ge0$ and that $\diam(\X)\le D$. Let linearly independent $P,Q\in\mathring\T^{\sf geo}_\mu\Wass_2$ be given with  $\langle P,Q\rangle_\mu=0$, $|P|_\mu=1$, and ${\sf inj}(P)\ge r$ in the sense that $(\exp(tP))_{t\in[0,r]}$ is a  geodesic.
Then  
$$\Sec_\mu(P,Q)\ge -\Big(\frac Dr\Big)^2\kappa.
$$
\end{corollary}

\begin{proof} If $(\exp(tP))_{t\in[0,r]}$ is a geodesic then for $P$-a.e. $v$
$$r|v|\le D.$$
Hence, with $L:=\frac Dr$ 
$$|v|\le \frac Dr =L \, |P|_\mu$$
since by assumption $|P|_\mu=1$. Thus the claim follows by the previous Proposition.
\end{proof}

\subsection{Curvature bounds at Dirac measures}

\begin{lemma} Assume that $\Sec_x\le\lambda$  for all $x\in\M$ and some $\lambda\ge0$. Then 
\begin{equation}\label{K-vs-k-dirac}\overline{\Sec}_{\delta_z}(P,Q)\le \lambda\cdot \!\!\!\limsup_{\tiny\begin{array}{c}
{s,t\to0}\\{t/s+s/t\textrm{ bdd}}\end{array}}\!\!\!
\frac{
 \int_{\T_z\M\times \T_z\M}  \big[|v|^2\, |w|^2-\langle v,w\rangle^2
\big]\,dN_{s,t}(v,w)
}{|P|^2|Q|^2-\langle P,Q\rangle^2}\end{equation}
for all $z\in\M$, all linearly independent $P,Q\in  \mathring\T^{\sf geo}_{\delta_z}\Wass_2$, and all
$N_{s,t}\in \Wass(\T_z\M\times \T_z\M)$ which are $\c_{s,t}$-optimal couplings of $P$ and $Q$
 w.r.t.~the cost function
$$\c_{s,t}(v,w):=\d^2\big(\exp_z(sv), \exp_z(tw)\big).$$
\end{lemma}

\begin{proof}
Fix $z\in\M$ and linearly independent $P,Q\in \T^{\sf geo}_{\delta_z}\Wass_2\cong  \Wass(\T_z\M)$ such that
$p_s:=\exp_{\delta_z}(sP)$ and $q_t:=\exp_{\delta_z}(tQ)$ for $s,t\in [0,t_0]$ are minimizing geodesics in $(\Wass,\W_2)$  for some $t_0>0$. 
For $s,t\in(0,t_0)$, choose $\c_{s,t}$-optimal  $N_{s,t}\in\Cpl(P,Q)$.
Then
\begin{align*}
\W^2_{\delta_z}(sP,tQ)&-
\W_2^2\big(\exp_{\delta_z}(sP), \exp_{\delta_z}(tQ)\big)\\
&\le\int\bigg[\int |sv-tw|^2_{z}-\d^2\big(\exp_z(sv),\exp_z(tw)\big)\bigg]\,dN_{s,t}(v,w)
\end{align*}
and thus
\begin{align*}
\overline\Sec_{\delta_z}(P,Q)&:=
 \!\!\!\limsup_{\tiny\begin{array}{c}
{s,t\to0}\\{t/s+s/t\textrm{ bdd}}\end{array}}\!\!\!
 \frac3{\Lambda s^2t^2}\Big[ 
 \W^2_{\delta_z}(sP,tQ)
 -\W_2^2\big(\exp_{\delta_z}(sP), \exp_{\delta_z}(tQ)\big)
\Big]\\
 &\le
 \!\!\!\limsup_{\tiny\begin{array}{c}
{s,t\to0}\\{t/s+s/t\textrm{ bdd}}\end{array}}\!\!\!
\frac1\Lambda \int_{\T_z\M\times \T_z\M} \frac{3}{s^2t^2}\Bigg[|sv-tw|^2-
\d^2\Big(\exp_z(sv),\exp_z(tw)\Big)
\Bigg]\,dN_{s,t}(v,w)\\
&=
\!\!\!\limsup_{\tiny\begin{array}{c}
{s,t\to0}\\{t/s+s/t\textrm{ bdd}}\end{array}}\!\!\!
\frac1\Lambda\int_{\T_z\M\times \T_z\M} 
\Sec_z(v,w)\cdot\big[|v|^2\, |w|^2-\langle v,w\rangle^2
\big]
\,dN_{s,t}(v,w).
\end{align*}
\end{proof}
As observed before, there is no straightforward way of estimating the RHS of \eqref{K-vs-k-dirac}. 
The only exception is when $\lambda=0$.
\begin{corollary}\label{dirac-curv} Assume that $\Sec_x\le0$  for all $x\in\M$. Then 
\begin{equation*}\overline{\Sec}_{\delta_z}(P,Q)\le 0\end{equation*}
for all $z\in\M$ and all linearly independent $P,Q\in  \mathring\T^{\sf geo}_{\delta_z}\Wass_2$.
\end{corollary}
For non-vanishing curvature bound, again the challenge remains to estimate  the ratio on the RHS of \eqref{K-vs-k-dirac}. This can be done in various ways as outlined in the previous subsection.
Let us present one result.
\begin{corollary}\label{dir-neg} Assume that $\Sec_x\le\lambda$  ($\forall x$)  for some $\lambda\ge0$ and that $P$ has constant length, i.e.
$|v|=\ell$  (for $P$-a.e. $v$) for  some $\ell$. Then
$$\overline{\Sec}_{\delta_z}(P,Q)\le\lambda.$$
\end{corollary}

\subsection{Sectional curvature  of $\Wass_2(\M)$ in the Euclidean case} 

Assume that $\M=
\R^n$. We will explicitly calculate the sectional curvature $\Sec_\mu$ for the two most important cases of $\mu$, namely,  Dirac measures and absolutely continuous measures. 

\begin{proposition} For every $z\in\M$ and every $P,Q\in \mathring\T^{\sf geo}_{\delta_z}\Wass_2$,
$$\Sec_{\delta_z}(P,Q)=0.$$
\end{proposition}

\begin{proof} This is an immediate consequence of Corollaries \ref{sec-pos} and \ref{dir-neg}.
\end{proof}

Now let us address the case of absolutely continuous $\mu$.

\begin{theorem}\label{sec-proj} Assume $\mu=\rho\Leb^n$ with $\rho\in\C^\infty_c(\R^n)$ and $\varphi,\psi\in\C_c^\infty(\{\rho>0\})$.
Then
$$\Sec_\mu(\nabla\varphi,\nabla\psi)=
\frac3\Lambda\cdot\inf\bigg\{\int_{\R^n} \big|\nabla^2\psi\nabla\varphi-\nabla\eta\big|^2d\mu: \ \eta\in \C_c^\infty(\R^n)
\bigg\}.$$
\end{theorem}
\begin{remark} (i)
Given a vector field $F$, let 
$\delta(F):=\inf\{\int |F-\nabla\eta|^2d\mu: \eta\in
H^1(\R^n,\mu)\}^{1/2}$
denote its $L^2(\mu)$-distance from the space of gradients.
The Theorem then states that $\Sec_\mu(\nabla\varphi,\nabla\psi)=\frac3\Lambda\cdot \delta^2(\nabla^2\psi\nabla\varphi)$.
Note that the right hand side is symmetric in $\varphi$ and $\psi$, i.e. 
$$\delta( \nabla^2\psi\nabla\varphi)=\delta( \nabla^2\varphi\nabla\psi)$$
since $\nabla^2\varphi\nabla\psi+\nabla^2\psi\nabla\varphi=\nabla\langle\nabla\varphi,\nabla\psi\rangle$.
Also note that $\delta(F)=0$ for all $F$ in dimension $n=1$.

(ii) As observed in \cite[Lemma 4]{Lott07} and elaborated in \cite[Prop.~3.2]{Ricci-on-Wasserstein},
the minimizer $\eta$ for  $\delta(F)$ is explicitly given as 
$$\eta=\Delta_{\rho}^{-1}\big(\mathrm{div}_{\rho} F\big).$$
Note that on the set of $u\in L^2(\R^n,\rho\Leb^n)$ with $\int u\rho dx=0$, the weighted Laplacian $\Delta_{\rho}$ is invertible, and  that the function $\nabla_{\rho}\cdot F$ belongs to this set.

Here the weighted divergence operator is defined through
$-\int u\, (\mathrm{div}_\rho F)\, \rho dx=\int \nabla u\, \cdot F\, \rho dx$ for all compactly supported smooth vector fields $F$ and functions $u$, 
and the weighted Laplacian as $\Delta_{\rho}u:=\mathrm{div}_{\rho} \nabla u$.
Indeed, with $\eta$ defined this way,
\begin{align*}
\int \nabla u\cdot \nabla\eta \,\rho dx=-\int u\Delta_\rho\eta \,\rho dx=
-\int u\, (\mathrm{div}_\rho F)\, \rho dx=\int \nabla u\, \cdot F\, \rho dx.
\end{align*}
Thus $\nabla\eta-F$ is orthogonal in $L^2(\rho dx)$ to all vector fields of the form $\nabla u$.

(iii) By othogonality of the projection,
\begin{align*}\delta^2(F)&=\|F\|^2_{L^2}-\|\nabla\eta\|^2_{L^2}\\
&=\int\bigg[
|F|^2+ \big(\mathrm{div}_{\rho} F\big)\Delta_{\rho}^{-1}\big(\mathrm{div}_{\rho} F\big)\bigg]\rho\,dx
\end{align*}
for any vector field $F$ and associated projection $\nabla\eta$.  

\end{remark}

\begin{proof}[Proof of the Theorem] 
{\it (i) The non-optimal transport map.}
Put 
$p_t:=(\Id+t\nabla\varphi)_*\mu$ and $q_t:=(\Id+t\nabla\psi)_*\mu$. Furthermore, put 
$U_t:=(\Id+t\nabla\psi)\circ(\Id+t\nabla\varphi)^{-1}$ such that
$(U_t)_*p_t=q_t$. Then
$$U_t=\Id+t\nabla f+t^2 F_t
$$
with $f:=\psi-\varphi$, $F_t=F_0+tF_0'+\mathit{O}(t^2)$  for 
$F_0:=\nabla^2(\varphi-\psi)\nabla\varphi$ and some $F_0'$.
Indeed, 
$$(\Id+t\nabla\varphi)^{-1}=\Id-t\nabla\varphi+t^2\,\nabla^2\varphi \nabla\varphi+\mathit{O}(t^3).
$$
Thus
$$U_t:=(\Id+t\nabla\psi)\circ(\Id+t\nabla\varphi)^{-1}=\Id+t(\nabla\psi-\nabla\varphi)+t^2\Big(\nabla^2\varphi-\nabla^2\psi\Big)\nabla\varphi+\mathit{O}(t^3).
$$

{\it (ii) The optimal transport map.} By polar factorization, there exist maps $T_t,S_t$ such that $U_t=T_t\circ S_t$ where $T_t$ is the optimal transport map from $p_t$ to $q_t$ and $S_t$ is $p_t$-preserving.  $T_t$ is given in terms of a Kantorovich potential (which is smooth in $x$ and smoothly depends on $t$). That is, there exists a $\C^\infty_c$ function 
$\theta_t$ such that
$T_t=\Id+\nabla\theta_t$. By elliptic regularity theory, $\theta_t$ smoothly depends on $t$ in a small neighborhood around 0. Furthermore, $\theta_0=0$ and $\theta'_0=f$. Thus there exists 
 a $\C^\infty_c$ function 
$\eta_t$ with
$\eta_t=\eta_0+t\eta_0'+\mathit{O}(t^2)$  such that
$$T_t=\Id+t\nabla f+t^2\nabla\eta_t.$$

{\it (iii) The projection onto gradients.}
Since $S_{s,t}$ is $p_s$-preserving,
$\int u(U_{s,t})dp_s=\int u\big(T_{s,t}\circ S_{s,t}\big)dp_s=\int u(T_{s,t})dp_s$
for every  $u\in\C^1_c(\M)$, that is,
\begin{equation}\label{Taylor-P}
\int u\Big(x+t\nabla f(x)+t^2F_0(x)+\mathit O(t^3)\Big)dp_t(x)=\int u\Big(x+t\nabla f(x)+t^2\nabla\eta_0(x)+\mathit O(t^3)\Big)dp_t(x).
\end{equation}
Therefore, Taylor expansion of $u(x+\ldots)$ around $x$  yields
$$\int\langle\nabla u,F_0-\nabla\eta_0\rangle dp_t=O(t)$$
and thus in turn
$$\int\langle\nabla u,F_0-\nabla\eta_0\rangle dp_0=0$$
for all  $u\in\C^1_c(\M)$.
That is, $\nabla\eta_0$ is the projection of $F_0$ onto the space of gradients in $L^2(p_0)$.

{\it (iv) Several almost orthogonality results.}
Now let us slightly modify the previous argument and consider the Taylor expansion of $u(x+t\nabla f(x)+t^2 F_t(x))$ around $x+t\nabla f(x)$. Then \eqref{Taylor-P} yields
\begin{equation}\label{1111}\int\langle\nabla u(\Id+t\nabla f),F_t-\nabla\eta_t\rangle dp_t=O(t^2)\end{equation}
which of course immediately implies
\begin{equation}\label{2222}\int\langle\nabla u,F_t-\nabla\eta_t\rangle dp_t=O(t^1).\end{equation}
For the particular choice $u=f$, \eqref{1111}, however, yields
\begin{equation}\label{3333}\int\langle\nabla f,F_t-\nabla\eta_t\rangle dp_t=O(t^2)\end{equation}
since $\nabla f(x+t\nabla f(x))=\nabla f_t(x)+O(t^2)$ with $f_t:=f+\frac t2 |\nabla f|^2$ and
$\int\langle\nabla  |\nabla f|^2,F_t-\nabla\eta_t\rangle dp_t=O(t^1)$
according to \eqref{2222}.

{\it (v) The transport cost estimate.}
\begin{align*}
\int_{\R^n}\Big|(\Id+t\nabla\varphi) &-(\Id+t\nabla\psi)\Big|^2d\mu-\W^2_2(p_t,q_t)
=
\int_{\R^n} \Big[\big|\Id- U_t\big|^2-\big|\Id- T_t\big|^2\Big]dp_t\\
&=t^2\,\int \Big[ \big|\nabla f+t F_t\big|^2-\big|\nabla f+t \nabla\eta_t\big|^2
\Big]dp_t\\
&=
t^4\,\int \Big[|F_t|^2-|\nabla\eta_t|^2\Big]dp_t+2t^3\,\int \langle\nabla f, F_t-\nabla\eta_t\rangle dp_t\\
&\stackrel{\eqref{3333}}=t^4\,\int \Big[|F_t-\nabla\eta_t|^2+2\langle\nabla\eta_t, F_t-\nabla\eta_t\rangle\Big]dp_t
+O(t^5)\\
&=t^4\,\int \Big[|F_t-\nabla\eta_t|^2+2\langle\nabla\eta_0, F_t-\nabla\eta_t\rangle\Big]dp_t
+O(t^5)\\
&\stackrel{\eqref{2222}}=t^4\,\int \big|F_t-\nabla\eta_t\big|^2dp_t
+O(t^5)\\
&=t^4\,\int \big|F_0-\nabla\eta_0\big|^2dp_0
+O(t^5).
\end{align*}

{\it (v) The conclusion: convergence of the 1-parameter functional}
\begin{align*}
\dot\Sec_\mu^\nabla(\nabla\varphi,\nabla\psi)&:=
  \lim_{t\to0}\frac3{\Lambda\,t^4}\bigg[
  \int_{\R^n}\Big|(\Id+t\nabla\varphi) -(\Id+t\nabla\psi)\Big|^2d\mu-\W^2_2(p_t,q_t)
\bigg]\\
&=\frac3\Lambda\, \int_{\R^n} \big|F_0-\nabla\eta_0\big|^2d\mu\\
&=\frac3\Lambda\, \inf\bigg\{\int_{\R^n} \big|F_0-\nabla\eta\big|^2d\mu: \ \eta\in\C^\infty_c(\R^n)\bigg\}
\end{align*}
with $F_0:=\nabla^2(\varphi-\psi)\nabla\varphi$.

{\it (vi) The final extension to the 2-parameter limit.}
Let $(s_\ell,t_\ell)_{\ell\in\N}$ be a sequence with $\sup_\ell(s_\ell/t_\ell+t_\ell/s_\ell)<\infty$ and $\lim_\ell s_\ell=\lim_\ell t_\ell=0$ such that
$$\limsup_{\ell\to\infty} K(s_\ell,t_\ell)= \!\!\!\limsup_{\tiny\begin{array}{c}
{s,t\to0}\\{t/s+s/t\textrm{ bdd}}\end{array}}\!\!\!
K(s,t)$$ 
where
$$K(s,t):=
\frac3{\Lambda\,s^2\,t^2}\bigg[\int_{\R^n}
\big|s\nabla\varphi-t\nabla\psi\big|^2
\,d\mu
-
\W_2^2\big(\exp_\mu(s\nabla\varphi),\exp_\mu(t\nabla\psi)\big)
\bigg].
$$
Passing to a subsequence, we may assume that the $\limsup_\ell$ is indeed a $\lim_\ell$ and that
$a:=\lim_\ell\frac{s_\ell}{t_\ell}$ exists. Replacing $\varphi$ by $a\varphi$, we may assume that without restriction $a=1$. Thus it remains to prove that the transport cost estimate in (iv) remains valid (with $\ell$-independent error term) for $\varphi_\ell:=a_\ell\,\varphi$
in the place of $\varphi$ where $a_\ell:=\frac{s_\ell}{t_\ell}\to 1$.
This is easily verified.
\end{proof}

\begin{remark}
The first result in the spirit of Theorem(s) \ref{sec-sec} (and \ref{sec-proj}) have been obtained by Otto \cite{Ot01}  and Lott 
\cite{Lott07} (with  
partial extensions  proposed in 
\cite{Giglimemo}, \cite{lott2017tangent}, and  \cite{Ricci-on-Wasserstein}).
Their approach is based on the so-called Otto calculus, and thus is limited to measures $\mu$ with smooth densities and analytic tangent vectors represented in terms of smooth functions. 
Our approach confirms their calculations and extends the result to full generality.
\end{remark}

\section{The Kantorovich-Hellinger geometry}

\subsection{The space of finite measures}
Throughout the sequel, we assume that 
$(\M,\d)$ is a Polish metric space (i.e.~there exists a complete separable metric which generates the same topology).
Our prime examples are Riemannian manifolds (not necessarily complete).

For  $p\in[0,\infty)$, we denote by 
${\mathcal M}_p(\M )$ the space of measures $\mu$ on $\M$ (equipped with its Borel field) with
$\int \d^p(x,z)\,d\mu(x)<\infty$ for all $z\in\M$. The latter condition is equivalent to $\int \big(1+\d^p(x,z)\big)\,d\mu(x)<\infty$ for some -- hence all -- $z\in\M$.
In particular,
$\Mea:=\Mea(\M )$ denotes the space of \emph{finite  measures} on $\M$. 
Observe that the well-known Kantorovich-Wasserstein metric $\W_2$ on $\Wass_2(\M)$ canonically extends to a pseudo metric on
$\Mea(\M)$, again denoted by the same symbol,
\begin{align*}\W_2(\mu_1,\mu_2)&=\inf\left\{\int_{\M\times\M}\d^2(x,y)\,d q(x,y): \, q\in\Cpl(\mu_1,\mu_2)\right\}^{1/2}.
\\&=
\begin{cases}
0, &\text{if }\mu_1(\M)=\mu_2(\M)=0\\
r\, \W_2\Big(\frac1{r^2}\mu_1,\frac1{r^2}\mu_2\Big),\quad &\text{if }\mu_1(\M)=\mu_2(\M)=r^2, r\in(0,\infty)\\
\infty, &\text{if } \mu_1(\M)=\mu_2(\M)=\infty\text{ or } \mu_1(\M)\not=\mu_2(\M).
\end{cases}
\end{align*}

We will equip $\Mea$ with the \emph{Hellinger-Kantorovich metric} $\HK$,  a metric of fundamental importance -- both from theoretical perspectives and from applications.
It was introduced and analyzed in detail within the last decade by three independent teams: Liero-Mielke-Savar\'e, 
Kondratyev-Monsaingeon-Vorotnikov, Chizat-Peyr\'e-Schmitzer-Vialard, and it can be defined (at least in the Riemannian case) in three entirely different, equivalent ways:

\begin{itemize}
\item as projection of the Kantorovich-Wasserstein metric on the cone $\Co=\R_+\times\M/\sim$ over  $\M$,
\item as  $\inf$-convolution of the Kantorovich-Wasserstein and the Kakutani-Hellinger metrics,
\item in terms of the continuity equation with growth term,
\end{itemize}
cf.  \cite{LMS-optimal, LMS-inventiones, LMS-fine}, \cite{kondratyev2016new}, \cite{CPSV1,CPSV2},  \cite{DST-infimal}.
For related recent developments, see also 
\cite{gallouet2018camassa}, 
\cite{kondratyev2019spherical},
\cite{schiavo2025hellinger}, \cite{laschos2026evolutionary}, \cite{gallouet2025regularity}. 

\subsubsection{The cone picture}
Given the metric space $(\M,\d)$ we denote by $(\Co,\d^\Co)$ the metric space where
$\Co:=\M\times\R_+/\sim$ with $(x,r)\sim(x',r')\Leftrightarrow (x,r)=(x',r')$ or $r=r'=0$, and
\begin{align*}\d^\Co\big((x,r)),(x',r')\big)^2&=r^2+{r'}^2-2rr'\cos\big( \d(x,x')\wedge\pi\big)\\
&=(r-r')^2+rr'\, \d_\pi(x,x')^2
\end{align*}
with $\d_\pi(x,x'):=2 \sin\big(\frac12(\d(x,x')\wedge\pi)\big)$, cf.~\cite{BBI}.
The (equivalence class of) point(s) 
 $o=(x,0)$ (with arbitrary $x\in\M$) is the \emph{vertex} of the cone.

We define canonical maps between the space of finite measures on $\M$ and the space of probability measures with finite second moments on $\Co$:
\begin{itemize} 
\item
Lifting:
$${\sf Li}: \Mea\to\Wass_2(\Co), \quad \mu\mapsto \hat\mu:=\frac1{r^2_\mu}\mu\otimes\delta_{r_\mu}\ \textrm{with }r_\mu:=\sqrt{\mu(\M)}.$$
\item
Projection:
$$\Pr: {\mathcal M}_2(\Co)\to\Mea, \quad \nu\mapsto\underline\nu:=\bigg(A\mapsto \int_{A\times\R_+}r^2d\nu(x,r)\bigg).$$
\end{itemize}
In particular, note that for $\nu\in \Wass_2(\Co)$,
$$\underline\nu(\M)=\int_{\Co}r^2d\nu(x,r)=\W^2_\Co(\nu,\delta_o).$$
The projection is the left inverse of the lifting, i.e.~$\Pr\circ\,{\sf Li}={\sf Id}_\Mea$, but not the other way round since the projection is not injective: for every $\mu\in\Mea\setminus\{0\}$ there exist plenty of $\nu\in{\mathcal M}_2(\Co)$ with $\Pr(\nu)=\mu$.

\begin{definition} The Hellinger-Kantorovich metric $\HK$ on $\Mea$ is defined as
\begin{align*}\HK(\mu_0,\mu_1)&:=\inf\Big\{\W_2^{\Co}(\nu_0,\nu_1): \ \nu_0,\nu_1\in {\mathcal M}_2(\Co), \ \Pr(\nu_0)=\mu_0, \ \Pr(\nu_1)=\mu_1
\Big\}\\
&=\inf\Big\{\int_{\Co\times\Co}\d^2_\Co\,dq: \ q\in{\mathcal M}_2(\Co\times\Co), \ \Pr({\pi_0}_*q)=\mu_0, \ \Pr({\pi_1}_*q)=\mu_1
\Big\}^{1/2}.
\end{align*}
Here $\pi_0$ and $\pi_1$ denote the projections $\Co\times\Co\to\Co$ onto the first and second, 
resp., factors.
\end{definition} This obviously defines a metric on $\Mea$.
Let us state several basic technical properties.
\begin{lemma}[Reduction Lemma]\label{red} Given $\mu_0,\mu_1\in\Mea$, define 
\begin{align*} S_i:=\supp[\mu_i], \qquad S_i':=\big\{x\in S_i: \d(x,S_{1-i})<\pi/2\big\}, \qquad S_i'':=S_i\setminus S_i'
\end{align*}
for $i=0,1$, and
$$\mu_i':=\one_{S_i'}\,\mu_i, \qquad \mu_i'':=\one_{S_i''}\,\mu_i.$$
 Then the $\HK$-distance of the pair $\mu_0,\mu_1$ can be expressed in terms of the $\HK$-distance of the `reduced pair' $\mu_0',\mu_1'$ and the total masses of the remainders $\mu_0''$ and $\mu_1''$:
\begin{align*}
\HK^2(\mu_0,\mu_1)&= \HK^2(\mu'_0,\mu'_1)+ \HK^2(\mu''_0,\mu''_1),\\
\HK^2(\mu''_0,\mu''_1)&=\mu_0''(\M)+\mu_1''(\M).
\end{align*}
\end{lemma}
\begin{proof} 
\cite[Lemma 7.19]{LMS-inventiones} and \cite[Thm.~2.5]{LMS-fine}, the latter dealing only with the  case $\M=\R^n$, but the general case is completely analogously.
\end{proof}

\begin{lemma}[Dilation Invariance] Given any measurable function $\vartheta: \Co\times\Co\to(0,\infty)$, define the map $\widehat\vartheta: \Co^2 \to\Co^2$ by
$$\widehat\vartheta\Big((x,s),(y,t)\Big):=\bigg(\Big(x,\frac s{\vartheta\big((x,s),(y,t)\big)}\Big),\Big(y,\frac t{\vartheta\big((x,s),(y,t)\big)}\Big)\bigg).$$
Then for every $q\in{\mathcal M}(\Co\times\Co)$, the measure $q_\vartheta:=\vartheta^2\cdot (\widehat\vartheta_*q)$ 
has the same projected marginals and the same transport cost, that is,
$\Pr({\pi_i}_*q)=\Pr({\pi_i}_*q_\vartheta)$ for $i=0,1$ and
$$\int \d_\Co^2\,dq=\int \d_\Co^2\,dq_\vartheta.$$
\end{lemma}

\begin{proof} Straightforward calculation, see \cite{LMS-inventiones}.
\end{proof}

The particular choice of constant dilation maps, that is, $\vartheta\equiv\alpha$ for some $\alpha\in\R_+^*$, proves that in the definition of $\HK$ we may restrict to \emph{probability} measures $q$. 
\begin{corollary} For all $\mu_0,\mu_1\in\Mea$,
\begin{align*}\HK(\mu_0,\mu_1)&:=\inf\Big\{\W_2^{\Co}(\nu_0,\nu_1): \ \nu_0,\nu_1\in \Wass_2(\Co), \ \Pr(\nu_0)=\mu_0, \ \Pr(\nu_1)=\mu_1
\Big\}\\
&=\inf\Big\{\int_{\Co\times\Co}\d^2_\Co\,dq: \ q\in\Wass_2(\Co\times\Co), \ \Pr({\pi_0}_*q)=\mu_0, \ \Pr({\pi_1}_*q)=\mu_1
\Big\}^{1/2}.
\end{align*}

\end{corollary}

The dilation invariance in particular implies that,
if $q\in{\mathcal M}(\Co^2) $ is a minimizer in the definition of $\HK(\mu_0,\mu_1)$, then so is $q_\vartheta$ for each dilation $\vartheta$ as above.
Similarly, for each minimizing sequence.

This allows us to show that, in the definition of $\HK(\mu_0,\mu_1)$, one of the factors can be fixed a priori (or `pinned') when choosing the minimizing pairs $\nu_0,\nu_1$ of liftings.

\begin{lemma}[Pinning Lemma]\label{simplification} Let $\mu_0,\mu_1\in\Mea$ be a reduced pair and 
let any $\rho_0\in\Wass_2(\Co)$ be given with $\Pr(\rho_0)=\mu_0$. Then
$$\HK(\mu_0,\mu_1):=\inf\Big\{\W_{\Co}(\rho_0,\nu_1): \ \nu_1\in \Wass_2(\Co),  \ \Pr(\nu_1)=\mu_1
\Big\}.$$
\end{lemma}

\begin{proof} (i) 
Let us first consider the case 
that
$d\rho_0(x,r)=\frac1{r^2_0(x)}d\delta_{r_0(x)}(r)\,d\mu_0(x)$
with some measurable function $r_0:\M\to(0,\infty)$.
By assumption, we may restrict ourselves to minimizers $q$ in the definition of $\HK(\mu_0,\mu_1)$ which are supported on $\Co^*\times\Co^*$.
Choose $\vartheta\big((x,s),(y,t)\big):=\frac s{r_0(x)}$ provided $s>0$, and $:=1$ otherwise. Then 
${\pi_0}_*q_\vartheta=\rho_0$.
Indeed, with $\nu_0:={\pi_0}_*q$,
\begin{align*}
q_\vartheta(A\times \Co)&=\int_{\Co^2} \vartheta^2((x,s),(y,t))\, \one_A(\widehat\vartheta((x,s),(y,t)))\,dq((x,s),(y,t))\\
&=\int_\Co\frac{s^2}{r_{0}^2(x)}\one_A(x,r_{0}(x))\,d\nu_0(x,s)\\
&=\int_\M \frac{1}{r_{0}^2(x)} \one_A(x,r_{0}(x))\,d\mu_0(x)\\
&=\int_\Co \one_A(x,r)\,d\rho_0(x,r)
\end{align*}
for every measurable subset $A\subset\Co$.
Thus
$$\HK^2(\mu_0,\mu_1)=\int\d^2_\Co\,dq=\int\d^2_\Co\,dq_\vartheta=\W_\Co^2(\rho_0,\nu_1)
$$
with some $\nu_1$ which projects to $\mu_1$, namely, $\nu_1:={\pi_1}_*q_\vartheta$.

(ii) Given an arbitrary $\rho_0$ which projects to $\mu_0$, choose a sequence 
\begin{equation}\label{special-n}
d\rho_0^n(x,r)=\frac1{r^2_n(x)}d\delta_{r_n(x)}(r)\,d\mu_n(x)
\end{equation}
 of measures  on $\Co$ with $\W_\Co(\rho_0^n,\rho_0)\to0$  as $n\to\infty$. Note that the set of measures of type \eqref{special-n} with arbitrary choice of measures $\mu_n$ and functions $r_n$ is dense in $\Wass_2(\Co)$. (This can be easily concluded from the fact that the set of discrete measures is dense.)
 
 Put $\mu^n_0:=\Pr(\rho_0^n)$, let $q^n$ denote a minimizer in the definition of $\HK(\mu_0^n,\mu_1)$ and $\vartheta_n$ the dilation defined as above, now with $r_n$ in the place of $r_0$, and finally  
 $\nu^n_1:={\pi_1}_*q^n_{\vartheta_n}$.
 
 (iii) We claim that the family $(\nu_1^n)_n$ is compact in $\Wass(\Co)$. To see this, first observe that the converging sequence $(\rho_0^n)_n$ defines a  compact family in $\Wass(\Co)$.
 Thus also its image under the projection map $\Pr$ is compact, that is,  the family $(\mu_0^n)_n$ is compact in
 $\Mea(\M)$. Moreover, the latter is bounded since $\sup_n \mu_0^n(\M)=\sup_n\W^2_\Co(\rho_0^n,0)<\infty$. Therefore, according to \cite[Lemma 7.3]{LMS-inventiones}, the family $(q^n_{\vartheta_n})_n$ is compact in $\Wass(\Co^2)$. Thus also the family of its first marginals $(\nu_1^n)_n$ is compact in $\Wass(\Co)$.
 
 (iv) With compactness at hands, we conclude that there exists an accumulation point 
 $$\nu^\infty_1=\lim_{\ell\to\infty} \nu^{n_\ell}_1$$
 for a suitable subsequence. Then
 \begin{align*}
 \W_\Co(\rho_0,\nu^\infty_1)&=
 \lim_{\ell\to\infty} \W_\Co(\rho_0^{n_\ell},\nu^{n_\ell}_1)=
 \lim_{\ell\to\infty} \HK(\mu_0^{n_\ell},\mu_1)= \HK(\mu_0,\mu_1).
 \end{align*}
 The last equality here follows from the fact that $\W_\Co(\rho_0^n,\rho_0)\to0$ implies
 $\HK(\mu_0^n,\mu_0)\to0$.
\end{proof}
 
For another 'normalization' of  minimizers $q$ in the defining set of $\HK(\mu_0,\mu_1)$, 
see 
\cite[Equ. (5.9)]{LMS-optimal}.

The next example illustrates the necessity of the restrictive assumptions in the   previous results.
\begin{example}[{\cite{LMS-optimal}}] Let $z_0,z_1$ be two points in $\M$ with 
$\d(z_0,z_1)>\frac\pi2$.
For $i=0,1$, put $\mu_i=\delta_{z_i}$, $\hat\mu_i:=\delta_{(z_i,1)}
$ and $\rho_i:=\mu_i+\delta_o$. Then with $q:=\delta_{(z_0,1),o}+\delta_{o,(z_1,1)}\in\Cpl(\rho_0,\rho_1)$,
$$\HK^2(\mu_0,\mu_1)=\W^2_\Co(\rho_0,\rho_1)=\int \d_\Co^2\,dq=2$$
whereas
$$\W_\Co^2(\hat\mu_0,\hat\mu_1)=\d_\Co^2\Big((z_0,1),(z_1,1)\Big)=4\,\sin^2\Big(\frac12\big(\d(z_0,z_1)\wedge\pi\big)\Big)>2.$$
\end{example}

\subsubsection{The inf-convolution}

\begin{definition}
Given two length metrics $\d_1, \d_2$ on a space $\M$, their inf-convolution $\d_*:=\d_1 \, \triangledown \, \d_2$ is the pseudo length metric on $\M$ defined by
\begin{align*}\d_*(x,y):=\lim_{N\to\infty}\, & \inf\bigg\{N\sum_{i=1}^N \Big( \d_1^2(x_{2i-2},x_{2i-1})+ \d_2^2(x_{2i-1},x_{2i})\Big):\\
&\qquad\quad\qquad\qquad
x_1,\ldots, x_{2N-1}\in\M, \ x_0=x, x_{2N}=y\bigg\}^{1/2}
\end{align*}
provided this limit exists for all $x,y\in\M$.
\end{definition}
Obviously, $\d_1 \, \triangledown 
 \, \d_2\le \d_1$ and  $\d_1  \, \triangledown \, \d_2\le \d_2$.
In general, it is not guaranteed that $\d_1 \, \triangledown \, \d_2(x,y)\not=0$ for $x\not=y$.
(Thus ``pseudo'' metric.)

\begin{definition} Given a Polish space $\M$, the Kakutani-Hellinger metric ${\sf He}$ on $\Mea(\M)$ is defined as
$${\sf He}(\mu_1,\mu_2):=\|u_1-u_2\|_{L^2(\rho)}
$$
where $\rho$ is any measure on $\M$ with $\mu_1\ll\rho, \mu_2\ll\rho$ and $u_i:=\sqrt{\frac{d\mu_i}{d\rho}}$ for $i=1,2$.
\end{definition}
It is a geodesic metric. The space $(\Mea(\M),{\sf He})$ has nonnegative and nonpositve curvature in the sense of Alexandrov.

\begin{theorem}[{\cite[Thm.~1.1]{DST-infimal}}] Let $(\M,\d)$ be a complete separable geodesic space. Then the Hellinger-Kantorovich metric on $\Mea(\M)$ 
is the inf-convolution of the Kantorovich-Wasserstein metric and the Kakutani-Hellinger distance:
$$\HK= \W_2\, \triangledown \,{\sf He}.$$
\end{theorem}

\subsubsection{Dynamic formulation and transport-growth equation}

\begin{theorem}[{\cite[Thm.~8.12]{LMS-inventiones}}] For any complete separable length space $\M$ and all $\mu_0,\mu_1\in\Mea(\M)$,
\begin{align*}
\frac12\HK^2(\mu_0,\mu_1)&=\sup\bigg\{
\int_\M \xi_1\,d\mu_1-\int_\M \xi_0\,d\mu_0: \quad \xi\in\mathcal C^1\big([0,1], \Lip_b(\M)\big),\\
&\qquad\qquad\qquad
\partial_t\xi+\frac12|{\sf D}_\M\xi_t|^2(x)+2\xi_t^2(x)\le0 \text{ in }\M\times(0,1)
\bigg\}
\end{align*}
where $|{\sf D}_\M f|(x):=\limsup_{y\to x}\frac{|f(y)-f(x)|}{\d(x,y)}$ for any Lipschitz function $f:\M\to\R$. 
\end{theorem}

\begin{corollary}[{\cite[Thm.~8.18]{LMS-inventiones}}] In the case $\M=\R^n$, the $\HK$-distance is given for any pair $\mu_0,\mu_1\in\Mea$  by the Benamou-Brenier formula with transport and growth terms
\begin{align*}
\HK^2(\mu_0,\mu_1)&=\min\bigg\{\int_0^1\int_\M \Big(|v_t|^2+\frac14|w_t|^2\Big)\,d\mu_t\,dt: \quad \mu\in\mathcal C\big([0,1]; \Mea(\R^n)\big),\\
&\qquad\qquad \qquad\mu_{t=i}=\mu_i, \ \partial_t\mu_t+\nabla\cdot(v_t\mu_t)=w_t\mu_t\text{ in }\mathcal D'\big(\R^n\times(0,1)\big)\bigg\}.
\end{align*}
\end{corollary}

\subsection{The geometry of the cone}

To proceed, we will now restrict to the Riemannian case where $(\M,\g)$ is a smooth Riemannian manifold with uniformly bounded sectional curvature and without boundary (but not necessarily complete).

Let $\Co$ denote the cone over the base space $\M$ and let   $o=(x,0)$ denote its vertex. 
We write $\langle v,w\rangle_{x}$ for the metric tensor $\g^\M(v,w)_x$ on $\M$, and 
analogously $\langle (v,s),(w,t)\rangle_{(x,r)}$ for the metric tensor $\g^\Co((v,s),(w,t))_x$ on the punctured cone $\Co^*:=\Co\setminus\{o\}$. The cone $\Co$ is the `warped product' $=\M\,{}_r\!\!\times\R_+$ with warping function $r$. Thus
for $(v,s)\in\T_{(x,r)}\Co\cong \T_x\M\times\R$,
$$|(v,s)|_{(x,r)}^2=r^2|v|^2_x+s^2.$$
\begin{lemma} 
The exponential map on $\Co$ is given by
$$\exp^\Co_{(x,r)}(v,s):=\Bigg(\exp^\M_x\bigg(\arctan\bigg(\frac{r|v|}{r+s}\bigg)\frac{v}{|v|}\bigg), \sqrt{(r+s)^2+r^2|v|^2}\Bigg)$$
 for $(x,r)\in \Co^*$ and $(v,s)\in\T_{(x,r)}\Co\cong \T_x\M\times\R$ with $s>-r$ and $|v|<(1+s/r)\pi/2$.
In particular, for suitable $t_*>0$, the map 
$$[0,t_*]\ni t\mapsto \gamma_t:=
\Bigg(\exp^\M_x\bigg(\arctan\bigg(\frac{rt|v|}{r+ts}\bigg)\frac{v}{|v|}\bigg), \sqrt{(r+ts)^2+r^2t^2|v|^2}\Bigg)\in \M$$
is a geodesic in $\Co$ with
$\dot\gamma_0=(v,s)$ and 
$|\dot\gamma_0|=\big|(v,s)\big|_{x,r}=r^2\, |v|^2+s^2$.
\end{lemma}
\begin{proof}
Assume that there is a unique geodesic $(\alpha_t)_{t\in[0,1]}$ in $\M$ which connects $x$ and $x'$. This geodesic can be ``lifted'' to obtain the geodesic $(z_t)_{t\in[0,1]}$ in $\Co$ which connects $(x,r)$ and $(x',r')$, cf.~\cite{BBI}, \cite{LMS-optimal}.
This lifting only depends on the range $\M_\alpha:=\alpha([0,1])$ of the geodesic, or in other words, only on the metric space $(\M_\alpha,\d)$. The rest, i.e.~$\M\setminus\M_\alpha$, is irrelevant.
Thus for this lifting construction we can assume without restriction that $(\M,\d)=(\R^1,|.|)$.
The claim then follows from 
\cite[equ. (2.7)]{LMS-fine} with $u$ replaced by $\frac sr$.
\end{proof}

\begin{lemma} Assume that $\dim\M\ge2$. Then for each $(x,r)\in\Co^*$, all  
linearly independent $v,w\in\T_x\M$, and all $s,t\in\R$,
$$\Sec^\Co_{(x,r)}\big((v,s),(w,t)\big)={r^2}\,\Big(\Sec^\M_x(v,w)-1\Big)
\cdot
\frac{
|v|_{x}^2 |w|_{x}^2-\langle v,w\rangle_{x}^2}
{|(v,s)|_{(x,r)}^2 |(w,t)|_{(x,r)}^2-\langle (v,s),(w,t)\rangle_{(x,r)}^2}.
$$
In particular,
$$\Sec^\Co_{(x,r)}\big((v,0),(w,0)\big)=\frac1{r^2}\Big(\Sec^\M_x(v,w)-1\Big), \qquad\quad
\Sec_{(x,r)}\big((v,0),(0,s)\big)=0.$$
\end{lemma}

\begin{proof} According to B.~O'Neill \cite[Chapter 7, Prop.~42]{ONeill},
cf.~A.~Bisson \cite[Prop.~6.6]{Bisson}, the Riemannian curvature tensor on $\Co$ satisfies
\begin{align*}
R^\Co_{V+S,W+T}(V+S) &=
R^\Co_{S,T}(S)+R^\Co_{S,W}(S)+R^\Co_{V,T}(S)+R^\Co_{V,W}(S)\\
&\quad+R^\Co_{S,T}(V)+R^\Co_{S,W}(V)+R^\Co_{V,T}(V)+R^\Co_{V,W}(V)\\
&=0+0+0+0+0+0+0+R^\M_{V,W}(V)-\frac1{r^2}\Big(|V|^2W-\langle V,W\rangle V\Big)
\end{align*}
for all vector fields $V,W$ on the fibre $\M$, all vector fields $S,T$ on the base $\R$, and the lifts of these vector fields to $\Co$ 
denoted by the same symbols.
Therefore
\begin{align*}
\langle R^\Co_{V+S,W+T}(V+S),W+T\rangle_{x,r}&=r^2\Big[\langle R^\M_{V,W}(V),W\rangle_{x}-
|V|_x^2|W|^2_x+\langle V,W\rangle_x^2 
\Big]\\
&=r^2\,\Big(\Sec^\M_x(V,W)-1\Big)\cdot\Big[|V|_x^2|W|^2_x-\langle V,W\rangle_x^2 
\Big]\\
\end{align*}
and thus
\begin{align*}
\Sec_{x,r}^\Co\big(V+S,W+T)&=\frac{\langle R^\Co_{V+S,W+T}(V+S),W+T\rangle_{x,r}}{|V+S|_{x,r}^2|W+T|^2_{x,r}-\langle V+S,W+T\rangle_{x,r}^2 }\\
&=r^2\,\Big(\Sec^\M_x(V,W)-1\Big)\cdot
\frac{|V|_x^2|W|^2_x-\langle V,W\rangle_x^2 }{|V+S|_{x,r}^2|W+T|^2_{x,r}-\langle V+S,W+T\rangle_{x,r}^2 }\\
&=\frac1{r^2}\,\Big(\Sec^\M_x(V,W)-1\Big)\cdot
\frac{|V|_{x,r}^2|W|^2_{x,r}-\langle V,W\rangle_{x,r}^2 }{|V+S|_{x,r}^2|W+T|^2_{x,r}-\langle V+S,W+T\rangle_{x,r}^2 }.
\end{align*}
\end{proof}

\begin{lemma}[{\cite{BBI}}]\label{sec-1} Assume that $\dim\M=1$, that is, $\M$ is a circle or an interval (of finite or infinite length).  Then for each $(x,r)\in\Co^*$, all  
linearly independent $v,w\in\T_x\M$, and all $s,t\in\R$,
$$\Sec^\Co_{(x,r)}\big((v,s),(w,t)\big)=0.
$$
\end{lemma}

\subsection{Exponential maps on $\Wass_2(\Co)$ and on $\Mea(\M)$} 
We say that a vector field $Z$ on $\Co^*$ is \emph{admissible} if it is given as a measurable map
$Z=(Z^{\sf hor},Z^{\sf vert}): \Co^*\to \T\M\times\R$ with
 $|Z^{\sf hor}(x,r)|<\frac\pi2(1+\frac{1}rZ^{\sf vert}(x,r))$  for each $(x,r)\in\Co^*$ (which in particular implies $Z^{\sf vert}(x,r)>-r$).
\begin{definition} For every $\nu\in\Wass_2(\Co^*)$ and every admissible vector field $Z$ on $\Co^*$, 
\begin{align*}\exp^{\Wass_2}_\nu(Z)&:=\exp^\Co(Z)_*\nu\\
&:=\textrm{push forward of $\nu$ under the map }
(x,r)\mapsto \exp_{x,r}^\Co\big(Z(x,r)\big).\end{align*}
\end{definition}
\begin{definition} For every $\mu\in\Mea\setminus\{0\}$
and every admissible vector field $Z$ on $\Co^*$,
\begin{align*}\exp^{\Mea}_\mu(Z)&:=\Pr\Big(\exp^{\Wass_2}_{\hat\mu}(Z)\Big).\end{align*}
\end{definition}

\begin{remark} We will be particularly interested in admissible vector fields on $\Co^*$ of the form
$$Z(x,r)=\Big(V(x), 2r\, S(x)\Big)\in \T_x\M\times\R\cong\T_{x,r}\Co$$
with a vector field $V$ on $\M$ and a function $S$ on $\M$.
Then
$$|Z|_{x,r}^2=r^2|V|_x^2+4r^2S(x)^2=: r^2\, |V,2S|_x^2$$
and thus
$$|Z|^2_{\hat\mu}=\frac1{r_\mu^2}\int_\M  |V,2S|_x^2\,d\mu(x)\cdot \int r^2 d\delta_{r_\mu}(r)=\int_\M |V,2S|_x^2\,d\mu(x)=:\|V,2S\|_\mu^2.$$
Admissibility of $Z$ is equivalent to  $|V|<\frac\pi2(1+2S)$ on $\M$ (which in particular implies $S>-\frac12$). 
\end{remark}

\begin{proposition} If $(\nu_t)_{t\in [0,\epsilon]}$ with $\nu_t:=\exp^{\Co}_\nu\big(t(V,2rS)\big)$ as above is a $(\Wass_2(\Co), \W_2)$-geodesic then
$$\W_2(\nu_0,\nu_t)=t\,\|V,2S\|_\mu$$
with 
$\mu=\Pr(\nu)$ and
$\|V,2S\|^2_\mu:=\int_\M |V,2S|^2_xd\mu(x)=\int_\M\big(|V|_x^2+4S(x)^2\big)d\mu(x)$.
\end{proposition}

\begin{proof}
\begin{align*}
\frac1{t^2}\W_2^2(\nu_t,\nu)
&=\int_\Co |V,2rS|^2_{x,r}\,d\nu(x,r)\\
&=\int_\Co\Big[r^2 |V|_x^2+4r^2 S(x)^2\Big] \,d\nu(x,r)\\
&=\int_\M\big(|V|_x^2+4S(x)^2\big)d\mu(x)=\|V,2S\|_\mu^2.
\end{align*}
\end{proof}

\begin{lemma}
For every $\mu\in\Mea\setminus\{0\}$ and every admissible vector field of the form $Z=(V,2rS)$ with some vector field $V$ on $\M$ and some function $S$ on $\M$,
the measure $\Pr\big(\exp^{\Wass_2}_{\nu}(Z)\big)$ is independent of the choice of $\nu\in\Wass_2(\Co)$ among those with $\Pr\nu=\mu$.
\end{lemma}

\begin{proof} Put $\mu':=\Pr\big(\exp^{\Wass_2}_{\nu}(Z)\big)$. Then
with 
$\tau(x):=\frac{1}{|V(x)|}\arctan\big(\frac{|V(x)|}{1+2S(x)}\big)$,
\begin{align*}
\int_\M f\,d\mu'&=\int_\Co f(x) r^2 d\Big(\exp^{\Wass_2(\Co)}_{\nu}\big((V,2rS)\big)\Big)(x,r)\\
&
=\int_\Co f\Big(\exp_x^\M\big(\tau(x)\, V(x)\big) \Big)\cdot \Big[\big(1+ 2S(x)\big)^2+|V(x)|^2\Big]\,r^2\,d\nu(x,r)\\
&
=\int_\M f\Big(\exp_x^\M\big(\tau(x)\, V(x)\big) \Big)\cdot \Big[\big(1+2S(x) \big)^2+|V(x)|^2\Big]d\mu(x).
\end{align*}
\end{proof}
Thus for this particular class of admissible vector fields $Z$ we also could write
\begin{align*}\exp^{\Mea}_\mu(Z)&:=\Pr\Big(\exp^{\Wass_2}_{\Pr^{-1}\mu}(Z)\Big).\end{align*}

\begin{lemma}
For every $\mu\in\Mea$,
every vector field $V$ on $\M$ and every function $S$ on $\M$ with $|V|<\frac\pi2(1+2S)$,
the measures $\mu_t:=\exp^{\Mea}_\mu\big(t(V,2rS)\big)$ for $t\in[0,1]$ are characterized by
\begin{align*}
\int_\M f\,d\mu_t=\int_\M f\Big(\exp_x^\M\big(\tau_t(x)\, V(x)\big) \Big)\cdot \Big[\big(1+2t S(x)\big)^2+t^2|V(x)|^2\Big]d\mu(x)
\end{align*}
for all bounded Borel $f$ on $\M$. Here
$\tau_t(x):=
\frac{1}{|V(x)|}\arctan\big(\frac{t|V(x)|}{1+2tS(x)}\big)$.
\end{lemma}
\begin{proof}
With 
 $\tau_t(x)$ as above,
\begin{align*}
\int_\M f\,d\mu_t&=\int_\Co f(x) r^2 d\Big(\exp^{\Wass_2}_{\hat\mu}\big(t(V,2rS)\big)\Big)(x,r)\\
&
=\int_M f\Big(\exp_x^\M\big(\tau_t(x)\, V(x)\big) \Big)\cdot \Big[\big(r+2rt S(x)\big)^2+t^2r^2|V(x)|^2\Big]\bigg(\frac1{r^2_\mu}\mu(dx)\otimes \delta_{r_\mu}(dr)\bigg)
\\
&
=\int_M f\Big(\exp_x^\M\big(\tau_t(x)\, V(x)\big) \Big)\cdot \Big[\big(1+2t S(x) \big)^2+t^2|V(x)|^2\Big]d\mu(x).
\end{align*}
\end{proof}

\subsection{The analytic tangent space}

\subsubsection{The $\HK$-version of the Brenier-McCann Theorem}

\begin{theorem}[{\cite[Problem 2.9, Cor.~3.5]{LMS-fine}}] \label{Brenier}
Assume that $\M=\R^n$, that $\mu_0\ll\mm$ and that $\mu_1$ is supported in $B_{\pi/2}(\supp[\mu_0])$.
Then there exists a unique $\HK$-geodesic connecting $\mu_0$ and $\mu_1$. It is given 
as
$$\mu_t=
\exp^{\Mea}_\mu\big(t\varphi\big)=\Pr\Big(
\exp^\Co\big(t\, (\nabla\varphi,2r\varphi)\big)_*\hat\mu\Big)
$$
in terms of a function $\varphi\in H^1(\mu)$.
Furthermore,
$$\HK(\mu_0,\mu_1)=\|\nabla\varphi,2\varphi\|_\mu=:\|\varphi\|_{H^1(\mu)}.$$
\end{theorem}

\subsubsection{Tangent space and exponential map}
Inspired by the previous results for absolutely continuous measures on $\R^n$, we now define the analytic tangent space for arbitrary Riemannian manifolds $(\M,\g)$ and arbitrary $\mu\in\Mea(\M)$.
\begin{definition}
Analytic tangent space: \ 
$\T^{\sf ana}_\mu\Mea=\overline{ \C^\infty_c(\M)}^{\|.\|_{H^1(\mu)}}=H^1(\mu)$ with
$$\|\varphi\|^2_{H^1(\mu)}=\int_\M \Big(|\nabla\varphi|^2_x+4\varphi(x)^2\Big)d\mu(x).$$
Exponential map
$$\exp_\mu^\Mea\big(\varphi\big):=\Pr\Big(
\exp^\Co\big( (\nabla\varphi,2r\varphi)\big)_*\hat\mu\Big).
$$
\end{definition}
One easily verifies that for any $\varphi\in \C^\infty_c(\M)$ and sufficiently small $\tau>0$, the curve $t\mapsto \exp_\mu^\Mea\big(t\varphi\big)$ for $t\in [0,\tau]$ will be a $\HK$-geodesic in $\Mea(\M)$.
  
\subsubsection{Asymptotics of the exponential map}

For later use, let us analyze the asymptotic expansion of the exponential map on $\Co$ in  case $\M=\R^n$. 
\begin{lemma} For $\varphi\in \C^\infty_c(\M)$ put $V_t=\exp^\Co(t(\nabla\varphi,2r\varphi))=:(V_t^{\sf hor},V_t^{\sf vert})$. Then
as $t\to0$,
\begin{align*}V_t^{\sf hor}(x)&=x+t\nabla\varphi-t^2\nabla\varphi^2+t^3\Big(4\varphi^2-\frac13|\nabla\varphi|^2
\Big)\nabla\varphi+\mathit{O}(t^4),\\
V_t^{\sf vert}(x,1)&=1+2t\varphi+\frac12t^2|\nabla\varphi|^2+\mathit{O}(t^4).
\end{align*}
\end{lemma}
\begin{proof}
\begin{align*}V_t^{\sf hor}(x)&=x+\arctan\left(\frac{t|\nabla\varphi|}{1+2t\varphi}\right)\cdot\frac{\nabla\varphi}{|\nabla\varphi|^2}\\
&=x+\frac{t\nabla\varphi}{1+2t\varphi}-\frac13t^3|\nabla\varphi|^2\nabla\varphi
+\mathit{O}(t^4)\\
&=x+t\nabla\varphi-2t^2\varphi\nabla\varphi+4t^3\varphi^2\nabla\varphi-\frac13t^3|\nabla\varphi|^2\nabla\varphi
+\mathit{O}(t^4),\\
V_t^{\sf vert}(x,1)&=\Big((1+2t\varphi)^2+t^2|\nabla\varphi|^2
\Big)^{1/2}\\
&=(1+2t\varphi)\cdot\bigg(1+\frac12t^2
\frac{|\nabla\varphi|^2}{(1+2t\varphi)^2}\bigg)+\mathit{O}(t^4)\\
&=
1+2t\varphi+\frac12t^2|\nabla\varphi|^2+t^3\Big(\varphi|\nabla\varphi|^2-\varphi|\nabla\varphi|^2
\Big)+\mathit{O}(t^4).
\end{align*}
\end{proof}

\section{Sectional curvature of $\Mea(\M)$}
\subsection{Definition, decomposition, and scaling}
\begin{definition} For all $\mu\in\Mea$ and all $\varphi,\psi\in\T^{\sf ana}_\mu\Mea$,
 \begin{align*}
 \Sec_\mu(\varphi,\psi)&=\!\!\! \lim_{\tiny\begin{array}{c}
{s,t\to0}\\{t/s+s/t\textrm{ bdd}}\end{array}}\!\!\!
\frac3{\Lambda s^2\, t^2}\Bigg[ |s\varphi-t\psi|_{H^1(\mu)}^2-
 \HK^2\Big(\exp_\mu^\Mea(s\varphi),\exp_\mu^\Mea(t\psi)\Big)
\Bigg]
\end{align*}
with
$\Lambda:=|\varphi|_{H^1(\mu)}^2\cdot|\psi|_{H^1(\mu)}^2-\langle \varphi,\psi\rangle_{H^1(\mu)}^2$, provided the limit exists.
\end{definition}

\begin{theorem} \label{def-dec-scal}
\begin{enumerate}[\rm (i)]
\item
The sectional curvature on $\Mea$
 can be decomposed as
$$
 \Sec_\mu(\varphi,\psi)= \Sec_\mu^\uparrow(\varphi,\psi)+ \Sec_\mu^\nabla(\varphi,\psi)$$
into its lifted part
 \begin{align*}
 \Sec_\mu^\uparrow(\varphi,\psi)&=
{ \!\!\! \lim_{\tiny\begin{array}{c}
{s,t\to0}\\{t/s+s/t\textrm{ bdd}}\end{array}}\!\!\!}
\frac3{\Lambda s^2\,t^2}\Bigg[ |s\varphi-t\psi|_{H^1(\mu)}^2\\
& \qquad\qquad-
 \int_\Co
\d_\Co^2\Big(\exp^\Co_{x,r}\big(s(\nabla\varphi,2r\varphi)\big),\exp^\Co_{x,r}\big(t(\nabla\psi,2r\psi)\big)\Big)\,d\hat\mu(x,r)
\Bigg]
\end{align*}
and its twisted part
 \begin{align*}
 \Sec_\mu^\nabla(\varphi,\psi)=
 \!\!\! {\lim_{\tiny\begin{array}{c}
{s,t\to0}\\{t/s+s/t\textrm{ bdd}}\end{array}}\!\!\!}\frac3{\Lambda s^2\,t^2}\Bigg[  \int_\Co
\d_\Co^2\Big(\exp^\Co_{x,r}\big(s(\nabla\varphi,2r\varphi)\big),&\exp^\Co_{x,r}\big(t(\nabla\psi,2r\psi)\big)\Big)\,d\hat\mu(x,r)\\
&- \HK^2\Big(\exp_\mu^\Mea(s\varphi),\exp_\mu^\Mea(t\psi)\Big)
\Bigg]
\end{align*}
provided these limits exist. 
\item The twisted part is always nonnegative:
$$ \Sec_\mu^\nabla(\varphi,\psi)\ge0.$$
\item For all $a,b>0$,
$$\Sec_{\mu}(a\varphi,b\psi)=\Sec_{\mu}(\varphi,\psi).$$
\item
For any $a>0$,
$$ \Sec_{a\mu}(\varphi,\psi)=\frac1a\,\Sec_{\mu}(\varphi,\psi).$$
Similarly,  $\Sec_{a\mu}^\uparrow(\varphi,\psi)=\frac1a\, \Sec_{\mu}^\uparrow(\varphi,\psi)$ and
$ \Sec_{a\mu}^\nabla(\varphi,\psi)=\frac1a\, \Sec_{\mu}^\nabla(\varphi,\psi)$.
\end{enumerate}
\end{theorem}

\begin{proof} 
(i) and (iii) are obvious.

(ii) follows from the fact that the push forward of $\hat\mu$ under the map 
$$(x,r)\mapsto \Big(\Pr \Big[\exp^\Co_{x,r}\big(s(\nabla\varphi,2r\varphi)\big)\Big],\Pr\Big[\exp^\Co_{x,r}\big(t(\nabla\psi,2r\psi)\big)\Big]\Big)$$
is a (not necessarily optimal) coupling of the measures $\exp_\mu^\Mea(s\varphi)$ and $\exp_\mu^\Mea(t\psi)$.

(iv)
Obviously
$\Lambda_{a\mu}(\varphi,\psi)=a^2\,\Lambda_\mu(\varphi,\psi)$ for $\Lambda_\mu(\varphi,\psi):=|\varphi|_{H^1(\mu)}^2\cdot|\psi|_{H^1(\mu)}^2-\langle \varphi,\psi\rangle_{H^1(\mu)}^2$.
Furthermore,
$|\varphi-\psi|_{H^1(a\mu)}^2=a\, |\varphi-\psi|_{H^1(\mu)}^2$ and
$$ \HK^2\Big(\exp_{a\mu}^\Mea(s\varphi),\exp_{a\mu}^\Mea(t\psi)\Big)= a\,\HK^2\Big(\exp_\mu^\Mea(s\varphi),\exp_\mu^\Mea(t\psi)\Big)$$
since $\exp_{a\mu}^\Mea(s\varphi)=a\,\exp_{\mu}^\Mea(s\varphi)$ and
$\HK(a\mu_0,a\mu_1)=\sqrt a\, \HK(\mu_0,\mu_1)$.

The scaling property of
 $\Sec_\mu^\uparrow\big(\varphi,\psi \big)$ is obvious according to the explicit formula given in the next subsection.
 This then also entails the scaling property of $\Sec_\mu^\nabla\big(\varphi,\psi \big)$.
\end{proof}

\begin{remark}
The scaling property (iv) of the sectional curvature reflects the fact that $(\Mea,\HK)$ is a cone and
$\HK^2(a\mu, 0)=a\, \HK^2(\mu, 0)$.

For a further scaling property of the sectional curvature, see Theorem \ref{scal-space}.
\end{remark}

\subsection{The lifted part of the sectional curvature}

Recall that in the case $n\ge2$,
$$\Sec^\Co_{(x,r)}\big((v,s),(w,t)\big)={r^2}\,\Big(\Sec^\M_x(v,w)-1\Big)
\cdot
\frac{
|v|_{x}^2 |w|_{x}^2-\langle v,w\rangle_{x}^2}
{|(v,s)|_{(x,r)}^2 |(w,t)|_{(x,r)}^2-\langle (v,s),(w,t)\rangle_{(x,r)}^2}
$$
for all $(x,r)\in\Co^*$ and all linearly independent $(v,s), (w,t)\in\T_{(x,r)}\Co\cong \T_x\M\times\R$.

\begin{theorem}\label{lift-cone}  Assume $\dim\M\ge2$. Then for all $\mu\in\Mea$ and all $\varphi,\psi\in\C_c^\infty(\M)\subset \T_\mu^{\sf ana}\Mea$,
$$\Sec_\mu^\uparrow\big(\varphi,\psi \big)=\frac1\Lambda \int_{\M}\Big(\Sec^\M_{x}\big(\nabla\varphi,\nabla\psi\big)-1\Big)\cdot \Big[|\nabla\varphi|_x^2\, |\nabla\psi|_x^2-\langle \nabla\varphi,\nabla\psi\rangle_x^2
\Big]\,d\mu(x)$$
where 
\begin{align*}\Lambda&:=|\varphi|_{H^1(\mu)}^2\cdot|\psi|_{H^1(\mu)}^2-\langle \varphi,\psi\rangle_{H^1(\mu)}^2\\
&=\int_\M\Big( |\nabla\varphi|^2_x+4\varphi(x)^2\Big)d\mu(x)\cdot
\int_\M\Big( |\nabla\psi|^2_x+4\psi(x)^2\Big)d\mu(x)\\
&\quad-\bigg(\int_\M\Big(\langle \nabla\varphi,\nabla\psi\rangle_x+4\varphi(x)\psi(x)\Big)d\mu(x)\bigg)^2.
\end{align*}
\end{theorem}

\begin{proof}
 \begin{align*}
 \Sec_\mu^\uparrow(\varphi,\psi)&=
 \!\!\! {\lim_{\tiny\begin{array}{c}
{s,t\to0}\\{t/s+s/t\textrm{ bdd}}\end{array}}\!\!\!}
\frac3{\Lambda s^2\,t^2}\Bigg[ |s\varphi-t\psi|_{H^1(\mu)}^2-
 \int_\Co
\d_\Co^2\Big(\exp^\Co_{x,r}\big(s(\nabla\varphi,2r\varphi)\big),\exp^\Co_{x,r}\big(t(\nabla\psi,2r\psi)\big)\Big)\,d\hat\mu(x,r)
\Bigg]\\
 &=
   \!\!\! {\lim_{\tiny\begin{array}{c}
{s,t\to0}\\{t/s+s/t\textrm{ bdd}}\end{array}}\!\!\!}
\frac3{\Lambda s^2\,t^2}\Bigg[ 
 \int_\M\bigg(  \Big|s(\nabla\varphi,2\varphi)-t(\nabla\psi,2\psi)\Big|^2_{x,1}\\
 &\qquad\qquad\qquad-
\d_\Co^2\Big(\exp^\Co_{x,1}\big(s(\nabla\varphi,2\varphi)\big),\exp^\Co_{x,1}\big(t(\nabla\psi,2\psi)\big)\Big)\bigg)\,d\mu(x)
\Bigg]\\
 &=\frac3\Lambda\int_\M    \!\!\! {\lim_{\tiny\begin{array}{c}
{s,t\to0}\\{t/s+s/t\textrm{ bdd}}\end{array}}\!\!\!}
\frac1{s^2\,t^2}\bigg( 
 \Big|s(\nabla\varphi,2\varphi)-t(\nabla\psi,2\psi)\Big|_{x,1}\\
 &\qquad\qquad\qquad-
\d_\Co^2\Big(\exp^\Co_{x,1}\big(s(\nabla\varphi,2\varphi)\big),\exp^\Co_{x,1}\big(t(\nabla\psi,2\psi)\big)\Big)\bigg)\,d\mu(x)
\\
&=
\frac1\Lambda\int_\M
\Sec^\Co_{x,1}\big((\nabla\varphi,2\varphi),(\nabla\psi,2\psi)\big)\\
&\qquad\qquad\cdot
\bigg[\Big( |\nabla\varphi|^2_x+4\varphi(x)^2\Big)\cdot
\Big( |\nabla\psi|^2_x+4\psi(x)^2\Big)-
\Big(\langle \nabla\varphi,\nabla\psi\rangle_x+4\varphi(x)\psi(x)\Big)^2
\bigg]
\,d\mu(x)\\
&=
\frac1\Lambda\int_\M
\Big(\Sec^\M_{x}\big(\nabla\varphi,\nabla\psi\big)-1\Big)\cdot
\Big( |\nabla\varphi|^2_x\cdot
 |\nabla\psi|^2_x-
\langle \nabla\varphi,\nabla\psi\rangle_x^2
\Big)
\,d\mu(x).
\end{align*}
Interchanging integration and limit is justified since $$\frac1{s^2\,t^2}\bigg( 
 \Big|s(\nabla\varphi,2\varphi)-t(\nabla\psi,2\psi)\Big|_{x,1}
-
\d_\Co^2\Big(\exp^\Co_{x,1}\big(s(\nabla\varphi,2\varphi)\big),\exp^\Co_{x,1}\big(t(\nabla\psi,2\psi)\big)\Big)\bigg)$$
is  bounded, uniformly in $x$, for sufficiently small $s,t$ with bounded ratio
according to the asymptotic estimate in Lemma \ref{asym-lem} and the fact that
 the sectional curvature of the cone at the points $(x,1)$ is bounded, i.e.~$|\Sec^\Co_{x,1}|\le C$, due to our standing assumption of bounded sectional curvature of $\M$. 
\end{proof}

\begin{corollary}
Assume that $\M=\mathbb S^n$ for $n\ge2$. Then
$$\Sec_\mu^\uparrow\big(\varphi,\psi \big)=0$$
for all $\mu, \varphi$ and $\psi$.
\end{corollary}

\begin{corollary}
Assume that $\M=\R^n$ or $\M=\mathbb T^n:=(\mathbb S^1)^n$ for $n\ge2$. Then
$$\Sec_\mu^\uparrow\big(\varphi,\psi \big)=-\frac1\Lambda \int_{\M}\Big[|\nabla\varphi|_x^2\, |\nabla\psi|_x^2-\langle \nabla\varphi,\nabla\psi\rangle_x^2
\Big]\,d\mu(x)$$
with $\Lambda$ as above.
\end{corollary}

\begin{proposition}\label{1d-hell} Assume that $\dim\M=1$. Then 
for all $\mu\in\Mea$ and all $\varphi,\psi\in\C_c^\infty(\M)\subset \T_\mu^{\sf ana}\Mea$,
$$\Sec_\mu^\uparrow\big(\varphi,\psi \big)=0.$$
\end{proposition}

\begin{proof}
Lemma \ref{sec-1} and proof of Theorem \ref{lift-cone}.
\end{proof}

\subsection{The twisted part of the sectional curvature}
Let $\M$ be arbitrary (either $\dim\M\ge2$ or $\dim\M=1$).
Recall that the twisted part of the sectional curvature is
 \begin{align*}
 \Sec_\mu^\nabla(\varphi,\psi):=
 \lim_{t\to0}\frac3{\Lambda t^4}\Bigg[  \int_\Co
\d_\Co^2\Big(\exp^\Co_{x,r}\big(t(\nabla\varphi,2r\varphi)\big),&\exp^\Co_{x,r}\big(t(\nabla\psi,2r\psi)\big)\Big)\,d\hat\mu(x,r)\\
&- \HK^2\Big(\exp_\mu^\Mea(t\varphi),\exp_\mu^\Mea(t\psi)\Big)
\Bigg]
\end{align*}
and note that obviously $\Sec_\mu^\nabla(\varphi,\psi)\ge0$.
\begin{theorem} 
Assume $\mu=\rho\Leb^n$ with $\rho\in\C^\infty_c(\R^n)$ and $\varphi,\psi\in\C_c^\infty(\{\rho>0\})$.
\begin{enumerate}[\rm (i)]
\item
\begin{align*}
\Sec_\mu^\nabla(\varphi,\psi)
 =\frac3\Lambda\, \inf\bigg\{
\int_\M\Big[
|F-\nabla \eta|^2+4\eta^2
\Big]\,d\mu: \ \eta\in
H^1(\R^n,\mu)
\bigg\}
\end{align*}
with 
$$F:=\frac12\Big(\nabla^2\varphi\nabla\psi-\nabla^2\psi\nabla\varphi\Big).$$

\item
 The unique minimizer $\eta_0$ of the above functional is explicitly given as 
$$\eta_0=-\big(4-\Delta_{\rho}\big)^{-1}\big(\mathrm{div}_{\rho} F\big)$$
and
\begin{align*}
\inf_\eta
\int_\M\Big[
|F-\nabla \eta|^2+4\eta^2
\Big]\,d\mu
&=\int|F|^2d\mu-\int \Big[\nabla \eta_0|^2+4\eta_0^2\Big]d\mu\\
&=\int \big[|F|^2- (\mathrm{div}_\rho F) \, \big(4-\Delta_\rho\big)^{-1} (\mathrm{div}_\rho F)\Big]\,d\mu\\
&=4\int\Big|\Big(-\Delta_\rho(4-\Delta_\rho)\Big)^{-1/2}\mathrm{div}_{\rho}  F\Big|^2\,d\mu.
\end{align*}
For the well-definedness of the last expression, observe that the operator $-\Delta_\rho$ is invertible on the set of functions $u$ with $\int u\rho dx=0$ and the function $u:=\mathrm{div}_\rho F$ matches this condition.
\end{enumerate}
\end{theorem}

\begin{proof}
{\it (i) The non-optimal transport map.}
Put 
$$\alpha_t:=\exp^\Mea_\mu(t\varphi)=\Pr(p_t),\qquad
p_t:=\exp^{\Wass_2}_{\hat\mu}\big((t\nabla\varphi,2tr\varphi)\big)=(V_t)_*\mu,$$  

$$V_t(x,r)=\exp_{x,r}^\Co\big((t\nabla\varphi,2tr\varphi)\big)=\big(V_t^{\sf hor}(x),V_t^{\sf vert}(x,r)\big),$$

$$V_t^{\sf hor}(x)=x+\arctan\bigg(\frac{t|\nabla\varphi(x)|}{1+2t\varphi(x)}\bigg)
\frac{\nabla\varphi(x)}{|\nabla\varphi(x)|}
$$

$$V_t^{\sf vert}(x,r)=r\,\big[(1+2t\varphi(x))^2+t^2|\nabla\varphi(x)|^2\big]^{1/2}.$$
Similarly,
$$\beta_t:=\exp^\Mea_\mu(t\psi)=\Pr(q_t),\qquad
q_t:=\exp^{\Wass_2}_{\hat\mu}\big((t\nabla\psi,2tr\psi)\big)=(W_t)_*\mu.$$  
Then
$$W_t^{\sf hor}(x)=x+t\nabla\psi(x)-2t^2\psi(x)\nabla\psi(x)+\mathit{O}(t^3),$$
$$W^{\sf vert}_t(x,r)=r\cdot\Big[1+2t\psi(x)+\frac12t^2|\nabla\psi|^2(x)+\mathit{O}(t^3)\Big],$$
and
$$(V_t^{-1})^{\sf hor}(x)=x
-t\nabla\varphi+t^2\,\nabla^2\varphi \nabla\varphi+
2t^2\varphi(x)\nabla\varphi(x)+\mathit{O}(t^3),$$
$$(V_t^{-1})^{\sf vert}_t(x,r)=r\cdot\Big[1-2t\varphi(x)
+4t^2\varphi^2(x)
+\frac32t^2|\nabla\varphi|^2(x)+\mathit{O}(t^3)\Big].$$

Furthermore, put 
$U_t:=W_t\circ V_t^{-1}$
such that
$(U_t)_*p_t=q_t$. Then
\begin{align*}U_t^{\sf hor}(x)=&
x+t\nabla\psi-2t^2\psi\nabla\psi-t^2\nabla^2\psi\nabla\varphi\\
&
-t\nabla\varphi+t^2\,\nabla^2\varphi \nabla\varphi+
2t^2\varphi\nabla\varphi+\mathit{O}(t^3)\\
=&x+t\nabla(\psi-\varphi)+t^2\nabla\Big(-\psi^2+\varphi^2+\frac12|\nabla\varphi|^2\Big)-t^2\nabla^2\psi\nabla\varphi
+\mathit{O}(t^3),\end{align*}
\begin{align*}U_t^{\sf vert}(x,r)&=r\cdot\Big[1+2t\psi-2t\varphi +\frac12t^2|\nabla\psi|^2
+4t^2\varphi^2
+\frac32t^2|\nabla\varphi|^2-4t^2\psi\,\varphi-2t^2\nabla\psi\,\nabla\varphi+\mathit{O}(t^3)\Big].
\end{align*}
Choose $F_t, f_t$ such that
$$U_t^{\sf hor}=
x+t\nabla\xi-t^2\nabla\xi^2+t^2F_t,
\qquad
U_t^{\sf vert}(x,1) =1+2t\xi+\frac12t^2|\nabla\xi|^2+2t^2f_t$$
with
$\xi:=\psi-\varphi$.
Note that
\begin{align*}F_0&:=\nabla\Big(-\psi^2+\varphi^2+\frac12|\nabla\varphi|^2\Big)-\nabla^2\psi\nabla\varphi
+\nabla\xi^2\\
&=
-2\nabla\big(\xi\varphi\big)-\nabla^2\xi\nabla\varphi,\\
f_0&:=\frac14\Big(|\nabla\psi|^2
+8\varphi^2
+3|\nabla\varphi|^2-4\nabla\psi\nabla\varphi-8\psi\,\varphi
-|\nabla\xi|^2\Big)\\
&=-2\xi\varphi- \frac12\nabla\xi\nabla\varphi.
\end{align*}

\bigskip

{\it (ii) The optimal transport map.} 
The Pinning Lemma \ref{simplification} together with the McCann-Brenier Theorem on the cone imply that for every  absolutely continuous lift of $\alpha_t$ there exists  some lift of $\beta_t$ and an $\HK$-optimal map $T_t$  on $\Co$ which transports 
the lift of $\alpha_t$ to the lift of $\beta_t$.
More precisely, $(\Pr\circ T_t)_*p_t=\beta_t$ and $\HK(\alpha_t,\beta_t)=\W_\Co(p_t, (T_t)_*p_t)$ with
$$T_t(x,r)=\exp_{x,r}^\Co\big((\nabla\theta_t,2r\theta_t)\big)=\big(T_t^{\sf hor}(x),T_t^{\sf vert}(x,r)\big)$$
for some function $\theta_t$ on $\M$. 
By elliptic regularity theory, $\theta_t$ smoothly depends on $t$ in a small neighborhood around 0. Furthermore, $\theta_0=0$ and $\theta'_0=\psi-\varphi$. Thus there exists 
 a $\C^\infty_c$ function 
$\eta$   such that
$\theta_t=t\xi+t^2\eta+\mathit{O}(t^3).$
Therefore,
\begin{align*}T^{\sf hor}_t(x)&=
x+\nabla\theta_t-\nabla\theta_t^2+\mathit{O}(t^3)\\
&=
x+t\nabla \xi-t^2\nabla\xi^2+t^2\nabla\eta+\mathit{O}(t^3),\end{align*}
\begin{align*}T^{\sf vert}_t(x,r)
&=r\cdot\Big[1+2\theta_t+\frac12|\nabla\theta_t|^2+\mathit{O}(t^3)\Big]
\\
&=
r\cdot\Big[1+2t\xi+\frac12t^2|\nabla\xi|^2+2t^2\eta
+\mathit{O}(t^3)\Big].
\end{align*}
Choose $G_t, g_t$ such that
$$T_t^{\sf hor}=
x+t\nabla\xi-t^2\nabla\xi^2+t^2G_t,
\qquad
T_t^{\sf vert}(x,1) =1+2t\xi+\frac12t^2|\nabla\xi|^2+2t^2g_t.$$
\bigskip

{\it (iii) The projection.}
We know that
$$(\Pr\circ T_t)_*p_t=\beta_t\quad\textrm{as well as}\quad
(\Pr\circ U_t)_*p_t=\beta_t.$$ 
Furthermore, for every  $u\in\C^1_c(\M)$,
\begin{align*}
\int_\M u\, d\Big((\Pr\circ U_t)_*p_t\Big)&=\int_\Co u(x) r^2 d\Big((U_t)_*p_t\Big)(x,r)\\
&=\int_\Co u\Big(U_t^{\sf hor}(x)\Big)\cdot U_t^{\sf vert}(x,r)^2\,dp_t(x,r)\\
&=\int_\Co u\Big(U_t^{\sf hor}(x)\Big)\cdot U_t^{\sf vert}(x,1)^2\,r^2\,dp_t(x,r)
\\
&=\int_\M u\Big(U_t^{\sf hor}(x)\Big)\cdot U_t^{\sf vert}(x,1)^2\,d\alpha_t(x)
\end{align*}
and similarly with $T_t$ in the place of $U_t$. Thus
$$\int_\M u\Big(U_t^{\sf hor}(x)\Big)\cdot U_t^{\sf vert}(x,1)^2\,d\alpha_t(x)=\int_\M u\Big(T_t^{\sf hor}(x)\Big)\cdot T_t^{\sf vert}(x,1)^2\,d\alpha_t(x).$$
That is,
\begin{align*}
\int_\M &u\Big(x+t\nabla\xi-t^2\nabla
\xi^2+t^2F_t
\Big)
\Big[1+2t\xi+\frac12t^2|\nabla\xi|^2+2t^2f_t
\Big]^2\,d\alpha_t(x)\\
&=
\int_\M u\Big(x+t\nabla\xi-t^2\nabla
\xi^2+t^2G_t
\Big) 
\Big[1+2t\xi+\frac12t^2|\nabla\xi|^2+2t^2g_t
\Big]^2\,d\alpha_t(x).
\end{align*}
Taylor expansion of $u$ around $x$ yields (after cancellation of equal terms),
\begin{align}\nonumber
&\int_\M\bigg(4 u\,(f_t-g_t)\, [
1+2t\xi]+4t\,\langle\nabla u,\nabla\xi\rangle\,(f_t-g_t)\\
&\quad+ \big\langle\nabla u, F_t-G_t\big\rangle\,[
1+4t\xi]
+t\big\langle\nabla\xi\,\nabla^2u, F_t-G_t\big\rangle\bigg)\,d\alpha_t(x)=\mathit{O}(t^2).
\label{xxx}
\end{align}
Thus
\begin{equation}\label{yyy}
\int_\M\Big[4 u\cdot(f_t-g_t)
+\big\langle\nabla u, F_t-G_t\big\rangle\Big]\,d\alpha_t(x)=\mathit{O}(t^1)
\end{equation}
for all $u\in\C^1_c(\M)$. In particular, since, $\alpha_t=\big(1+\mathit{O}(t)\big)\,\mu$ and $G_0=\nabla\eta, g_t=\eta$,
\begin{equation}\label{zyx}\int_\M\Big[\langle\nabla u,F_0-\nabla\eta\rangle +4\, u\cdot(f_0-\eta)\Big]d\mu=0\end{equation}
for all  $u\in\C^1_c(\M)$.

Moreover, applying \eqref{yyy} to $2\xi^2+\frac12|\xi|^2$ and \eqref{xxx} to $\xi$ yields
\begin{align}
\label{zzz}
\int_\M&\bigg[4 \Big(\xi+\frac 12t|\nabla\xi|^2\Big)\,(f_t-g_t)
+\Big\langle\nabla \xi, F_t-G_t\Big\rangle\bigg]\,d\alpha_t(x)\\
&\stackrel{\eqref{yyy}}{=}\int_\M\bigg[4 \Big(\xi+2t\xi^2+ t|\nabla\xi|^2\Big)\,(f_t-g_t)
+\Big\langle\nabla\Big( \xi+2t\xi^2+\frac12t|\nabla\xi|^2\Big), F_t-G_t\Big\rangle\bigg]\,d\alpha_t(x)+\mathit{O}(t^2)\nonumber\\
&=\int_\M\bigg(4 \xi\,(f_t-g_t)\, [
1+2t\xi]+4t\,\langle\nabla \xi,\nabla\xi\rangle\,(f_t-g_t)\nonumber\\
&\quad+ \big\langle\nabla \xi, F_t-G_t\big\rangle\,[
1+4t\xi]+t\big\langle\nabla\xi\,\nabla^2\xi, F_t-G_t\big\rangle\bigg)\,d\alpha_t(x)+\mathit{O}(t^2)\nonumber\\
&\stackrel{\eqref{xxx}}{=}\mathit{O}(t^2)\nonumber.
\end{align}

Define the Hilbert space  $(\mathcal H, \|.\|_\mu)$ as the closure of the set of pairs
$(H,h)$ where the $H$'s are compactly supported smooth vector fields on $\M$ and the $h$'s are compactly supported  smooth functions on $\M$, and the norm of such a pair is given by 
$$\|(H,h)\|^2_\mu:=\int_\M \Big(| H|_x^2+h(x)^2\big)d\mu(x).$$
The space $H^1(\mu)$ is isometrically embedded into $\mathcal H$ with $\|(\nabla u,2u)\|_\mu=\|u\|_{H^1(\mu)}$.
With this notation, our pervious result \eqref{zyx} states that
$$\big\langle (\nabla u,2u), (F_0,2f_0)-(\nabla \eta,2\eta)
\big\rangle_\mu=0
$$
for all $u\in H^1(\mu)$. That is, the vector $(F_0,2f_0)-(\nabla \eta,2\eta)\in \mathcal H$ is orthogonal to the subspace $H^1(\mu)$. Therefore, $\eta$ is the projection of $(F_0,2f_0)$ onto $H^1(\mu)$, and for all $u\in H^1(\mu)$,
$$\big\|(\nabla  u,2u)-(F_0,2f_0)\|_\mu^2=\big\|(\nabla  \eta,2\eta)-(F_0,2f_0)\|_\mu^2+\|u-\eta\|_{H^1\mu)}^2.$$

\bigskip

{\it (iv) The transport cost estimate.} Our goal is to calculate 
 \begin{align*}
 \int_\Co
\d_\Co^2\Big(V_t(x,r),&W_t(x,r)\Big)\,d\hat\mu(x,r)
- \HK^2\Big(\exp_\mu^\Mea(t\varphi),\exp_\mu^\Mea(t\psi)\Big)\\
&=\int_\Co\bigg[
\d_\Co^2\Big((x,r),U_t(x,r)\Big)-\d_\Co^2\Big((x,r),T_t(x,r)\Big)\bigg]\,dp_t(x,r).
\end{align*}
To do so, note that
\begin{align*}
\d_\Co^2\Big((x,r),T_t(x,r)\Big)&=\Big|r-T^{\sf vert}_t(x,r)\Big|^2+rT^{\sf vert}_t(x,r)\cdot 4\sin^2\bigg(\frac 12
\Big|x-T^{\sf hor}_t(x,r)\Big|\bigg)
\end{align*}
and thus
\begin{align*}
\frac1{r^2}\,\d_\Co^2\Big((x,r),T_t(x,r)\Big)
&=\Big|2t\xi
+\frac12t^2|\nabla\xi|^2+2t^2 g_t
\Big|^2\\
&+\Big[1+2t\xi
+\frac12t^2|\nabla\xi|^2+2t^2g_t\Big]
\cdot 
 4\sin^2\bigg(\frac 12
\Big[
t\nabla\xi-t^2\nabla\xi^2+t^2G_t
\Big]\bigg).
\end{align*}
Similarly,
\begin{align*}
\frac1{r^2}\,\d_\Co^2\Big((x,r),U_t(x,r)\Big)&=\Big|2t\xi
+\frac12t^2|\nabla\xi|^2+2t^2f_t
\Big|^2\\
&+\Big[1+2t\xi
+\frac12t^2|\nabla\xi|^2+2t^2f_t\Big]\cdot 
4 \sin^2\bigg(\frac 12
\Big[
t\nabla\xi-t^2\nabla\xi^2+t^2F_t
\Big]\bigg).
\end{align*}
Hence,
\begin{align*}
\frac1{r^2}\Big[\d_\Co^2\Big((x,r),U_t(x,r)\Big)&-\d_\Co^2\Big((x,r),U_t(x,r)\Big)\Big]
\\
=&
\Big|2t\xi
+\frac12t^2|\nabla\xi|^2+2t^2 f_t
\Big|^2-\Big|2t\xi
+\frac12t^2|\nabla\xi|^2+2t^2 g_t
\Big|^2\\
&+\Big[1+2t\xi
+\frac12t^2|\nabla\xi|^2+2t^2f_t\Big]
\cdot 
\Big[
t\nabla\xi-t^2\nabla\xi^2+t^2F_t
\Big]^2\\
&-\Big[1+2t\xi
+\frac12t^2|\nabla\xi|^2+2t^2g_t\Big]
\cdot 
\Big[
t\nabla\xi-t^2\nabla\xi^2+t^2G_t
\Big]^2+\mathit{O}(t^5)\\
=&4t^4\big(f_t^2-g_t^2\big)+4t^3\big(2\xi
+\frac12t|\nabla\xi|^2\big)(f_t-g_t)\\
&+2t^4(f_t-g_t)\,|\nabla\xi|^2\\
&+t^4\big(|F_t|^2-|G_t|^2\big)+2(1+2t\xi)\,t^3\langle \nabla\xi-t\nabla\xi^2,
F_t-G_t\rangle+\mathit{O}(t^5)\\
=&4t^4(f_t-g_t)^2\\
&+2t^3\Big(4\xi+2t|\nabla\xi|^2
+4tg_t\Big)\, \big(f_t-g_t\big)\\
&+t^4|F_t-G_t|^2\\
&+2\,t^3\Big\langle\nabla\xi+tG_t,
F_t-G_t\Big\rangle+\mathit{O}(t^5).
\end{align*}
If we integrate this with respect to $\alpha_t$ then
according to \eqref{zzz} the term
$$8t^3\Big(\xi+\frac12t|\nabla\xi|^2\Big)\,(f_t-g_t)+2t^3\langle\nabla\xi,F_t-G_t\rangle$$
is of order $\mathit{O}(t^5)$,
and so is
$$8t^4g_t(f_t-g_t)+2t^4\langle G_t,F_t-G_t\rangle$$
according to \eqref{yyy} (taking into account that $g_0=\eta, G_0=\nabla\eta$). Thus
 \begin{align*}
 \int_\Co
\d_\Co^2\Big(V_t(x,r),&W_t(x,r)\Big)\,d\hat\mu(x,r)
- \HK^2\Big(\exp_\mu^\Mea(t\varphi),\exp_\mu^\Mea(t\psi)\Big)\\
&=\int_\Co\bigg[
\d_\Co^2\Big((x,r),U_t(x,r)\Big)-\d_\Co^2\Big((x,r),T_t(x,r)\Big)\bigg]\,dp_t(x,r).
\\
&=\int_\Co\bigg[4t^4(f_t-g_t)^2
+t^4|F_t-G_t|^2\bigg]\,r^2\,dp_t(x,r)+\mathit{O}(t^5)
\\
&=
t^4\,\int_\M\bigg[4(f_t-g_t)^2
+|F_t-G_t|^2\bigg]\,d\alpha_t(x)+\mathit{O}(t^5)
\\
&=
t^4\,\int_\M\bigg[4(f_0-\eta)^2
+|F_0-\nabla\eta|^2\bigg]\,d\mu(x)+\mathit{O}(t^5)
\end{align*}

\bigskip

{\it (v) The conclusion.}
\begin{align*}
\Sec_\mu^\nabla(\varphi,\psi)&:=
  \lim_{t\to0}\frac3{\Lambda\,t^4}\bigg[
  \int_\Co
\d_\Co^2\Big(V_t(x,r),W_t(x,r)\Big)\,d\hat\mu(x,r)
- \HK^2\Big(\exp_\mu^\Mea(t\varphi),\exp_\mu^\Mea(t\psi)\Big)
 \bigg]\\
&=\frac3\Lambda\, 
\int_\M\Big[
|F_0-\nabla\eta|^2+4|f_0-\eta|^2
\Big]\,d\mu\\
&=\frac3\Lambda\, \inf\bigg\{
\int_\M\Big[
|F_0-\nabla u|^2+4|f_0-u|^2
\Big]\,d\mu: \ u\in\C^\infty_c(\R^n)\bigg\}.
\end{align*}

Recall that
\begin{align*}F_0=
-2\nabla\big(\xi\varphi\big)-\nabla^2\xi\nabla\varphi,\qquad f_0=-2\xi\varphi- \frac12\nabla\xi\nabla\varphi
\end{align*}
\bigskip
and 
note that
$$\inf_u \big\|(\nabla  u,2u)-(F_0,2f_0)\|_\mu^2=\inf_u\big\|(\nabla  u,2u)-(F^*,2f^*)\|_\mu^2$$
with 
$$F^*:=
F_0-\nabla f_0=
\frac12\Big(\nabla^2\varphi\nabla\psi-\nabla^2\psi\nabla\varphi\Big), \qquad f^*:=0.$$
\bigskip

By variational calculus, 
any minimizer $u_0$ of the functional
$u\mapsto \int_\M\big[
|F_0-\nabla u|^2+4|f_0-u|^2
\big]\,d\mu$
satisfies
$$
\int_\M\Big[
\big\langle F_0-\nabla u_0,\nabla v\big\rangle+4(f_0-u_0)\,v\Big]
\,d\mu=0$$
and thus
$$
\int_\M\Big[-\mathrm{div}_\mu F_0+4f_0+\Delta_\mu u_0-4u_0
\Big]\,v
\,d\mu=0$$
for all $v\in\C_c^1(\M)$.
Hence, $-\mathrm{div}_\mu F_0+4f_0+\Delta_\mu u_0-4u_0=0$ a.e. and
$$u_0=\big(4-\Delta_{\mu}\big)^{-1}\big(-\mathrm{div}_{\mu} F_0+ 4f_0\big).$$

Finally, observe that $F=\nabla f$ for $f:=(\Delta_\rho)^{-1}\mathrm{div}_{\mu} F$.
\begin{align*}&\int \big[|F|^2- (\mathrm{div}_\rho F) \, \big(4-\Delta_\rho\big)^{-1} (\mathrm{div}_\rho F)\Big]\,d\mu\\
&=4\int (\mathrm{div}_\rho F)\Big(-\Delta_\rho(4-\Delta_\rho)\Big)^{-1}(\mathrm{div}_{\rho}  F)\,d\mu\\
&=4\int\Big|\Big(-\Delta_\rho(4-\Delta_\rho)\Big)^{-1/2}\mathrm{div}_{\rho}  F\Big|^2\,d\mu.
\end{align*}

\bigskip

{\it (vi) The final extension to the 2-parameter limit.}
Following the argumentation from the proof of Theorem \ref{sec-proj}, 
it remains  to prove that the transport cost estimate in (iv) remains valid (with $i$-independent error term) for $\varphi_\ell:=a_\ell\,\varphi$
in the place of $\varphi$ where $a_\ell:=\frac{s_\ell}{t_\ell}\to 1$.
Again this is easily verified.
\end{proof}

\begin{remark}
Recall the well-known fact that the sectional curvature in the Kantorovich-Wasserstein space $(\Wass_2(M), \W_2)$ vanishes whenever the base space $\M$ is 1-dimensional. This is no longer the case for the sectional curvature in the Hellinger-Kantorovich space $(\Mea(M), \HK)$.
Indeed, 
$$\int\Big|\Big(-\Delta_\rho(4-\Delta_\rho)\Big)^{-1/2}\mathrm{div}_{\rho}  F\Big|^2\,d\mu>0$$ whenever $\mathrm{div}_{\rho}  F\not\equiv0$.
\end{remark}

\subsection{Spacial scaling} The explicit representation formulas for lifted and twisted parts of the sectional curvature on $\Mea(\R^n)$ allow us to deduce a further important scaling property of the sectional curvature.
We restrict ourselves to formulate it in the case $\M=\R^n$.

\begin{theorem}\label{scal-space} Given  $\mu=\rho\Leb^n$ with $\rho\in\C^\infty_c(\R^n)$ and $\varphi,\psi\in\C_c^\infty(\{\rho>0\})$ as before, define 
$$
\varphi_\ell(x):=\varphi(\ell x), \qquad 
\psi_\ell(x):=\psi(\ell x), \qquad \rho_\ell(x):=\ell^{n+2}\,\rho(\ell x), $$
and $\mu_\ell:=\rho_\ell\Leb^n$
for $\ell\in\N$. Then
$$ \Sec_{\mu_\ell}(\varphi_\ell,\psi_\ell)\to  \Sec^*_\mu(\varphi,\psi)$$
as $\ell\to\infty$ with 
$ \Sec^*_\mu(\varphi,\psi):=
\frac{3}{\Lambda^*}\, \inf\big\{
\int_\M\Big[
|F-\nabla \eta|^2
\Big]\,\rho\,dx: \ \eta\in
\mathcal C_c^\infty(\R^n)
\big\}
$ and
$\Lambda^*:=\big(\int|\nabla\varphi|^2d\mu\big)\cdot \big(\int|\nabla\psi|^2d\mu\big)-
\big(\int\langle\nabla\varphi,\nabla\psi\rangle d\mu\big)^2
$.

 If $\mu$ is a probability measure then in particular $\Sec^*_\mu(\varphi,\psi)$ coincides with the sectional curvature $\Sec_\mu(\nabla\varphi,\nabla\psi)$ in the space $(\Wass_2,\W_2)$ as
defined and analyzed in Def.~\ref{def-sec-wass} and Thm.~\ref{sec-proj}.
\end{theorem}

\begin{proof}
\begin{align*}\frac1{\ell^8}\,\Lambda_\ell&:=\frac1{\ell^8}\,\bigg[|\varphi_\ell|_{H^1(\mu_\ell)}^2\cdot|\psi_\ell|_{H^1(\mu_\ell)}^2-\langle \varphi_\ell,\psi_\ell\rangle_{H^1(\mu_\ell)}^2\bigg]\\
&=\frac1{\ell^8}\,\Bigg[\int_\M\Big( \ell^2\, |\nabla\varphi|^2+4\varphi^2\Big)(\ell x)\cdot \rho(\ell x) \, \ell^{n+2}\, dx\cdot
\int_\M\Big( \ell^2\,|\nabla\psi|^2+4\psi^2\Big)(\ell x) \cdot \rho(\ell x) \, \ell^{n+2}dx\\
&\quad-\bigg(\int_\M\Big( \ell^2\,\langle \nabla\varphi,\nabla\psi\rangle+4\varphi\psi\Big)(\ell x)\cdot \rho(\ell x) \, \ell^{n+2}\, dx\bigg)^2\Bigg]\\
&=\int_\M\Big(  |\nabla\varphi|^2+\frac4{\ell^2}\varphi^2\Big)(x)\, \rho(x) \, dx\cdot
\int_\M\Big(|\nabla\psi|^2+\frac4{\ell^2}\psi^2\Big)(x) \, \rho(x) \, dx\\
&\quad-\bigg(\int_\M\Big(\langle \nabla\varphi,\nabla\psi\rangle+\frac4{\ell^2}\varphi\psi\Big)(x)\, \rho(x) \, dx\bigg)^2\\
&\to \int_\M|\nabla\varphi|^2(x)\, \rho(x) \, dx\cdot
\int_\M|\nabla\psi|^2(x) \, \rho(x) \, dx
-\bigg(\int_\M\langle \nabla\varphi,\nabla\psi\rangle(x)\, \rho(x) \, dx\bigg)^2=:\Lambda^*
\end{align*}
as $\ell\to\infty$.
Thus
\begin{align*}\Sec_{\mu_\ell}^\uparrow\big(\varphi_\ell,\psi_\ell \big)
&=-\frac1{\Lambda_\ell} \int_{\M} \Big[|\nabla\varphi_\ell|^2\, |\nabla\psi_\ell|^2-\langle \nabla\varphi_\ell,\nabla\psi_\ell\rangle^2
\Big]\,d\mu_\ell\\
&=-\frac{\ell^6}{\Lambda_\ell} \int_{\M} \Big[|\nabla\varphi|^2\, |\nabla\psi|^2-\langle \nabla\varphi,\nabla\psi_\ell\rangle^2
\Big]\,d\mu\\
&\to 0
\end{align*}
as $\ell\to\infty$.
Moreover, with 
$F:=\frac12\big(\nabla^2\varphi\nabla\psi-\nabla^2\psi\nabla\varphi\big)$
and $F_\ell(x):=\frac12\big(\nabla^2\varphi_\ell\nabla\psi_\ell-\nabla^2\psi_\ell\nabla\varphi_\ell\big)(x)=\ell^3\, F(\ell x),$
\begin{align*}
\Sec_{\mu_\ell}^\nabla(\varphi_\ell,\psi_\ell)
 & =\frac3{\Lambda_\ell}\, \inf\bigg\{
\int_\M\Big[
|F_\ell-\nabla \eta|^2+4\eta^2
\Big]\,\rho_\ell\,dx: \ \eta\in
\mathcal C_c^\infty(\R^n)
\bigg\}\\
& =\frac{3\ell^8}{\Lambda_\ell}\, \inf\bigg\{
\int_\M\Big[
|F-\nabla \eta|^2+\frac4{\ell^2}\eta^2
\Big]\,\rho\,dx: \ \eta\in
\mathcal C_c^\infty(\R^n)
\bigg\}\\
& \to\frac{3}{\Lambda^*}\, \inf\bigg\{
\int_\M\Big[
|F-\nabla \eta|^2
\Big]\,\rho\,dx: \ \eta\in
\mathcal C_c^\infty(\R^n)
\bigg\}
\end{align*}
as $\ell\to\infty$.
\end{proof}

\subsection{Sectional curvature for geometric and extended tangent spaces}
\subsubsection{The geometric tangent space  at the zero measure}
For arbitrary non-vanishing $\rho,\nu\in \Mea$ consider $\HK$-geodesics $\rho_t:=t^2\rho, \nu_t:=t^2\nu$ emanating in the zero measure $\rho_0=\nu_0=0$. Then the directional distance
$$\HK^{\sf dir}_0\Big((\rho_t)_{t\ge0},(\nu_t)_{t\ge0}\Big):=
\lim_{t\to0}\frac1t \HK(\rho_t,\nu_t)$$
exists and coincides with $\HK(\rho,\nu)$. Indeed, the RHS is independent of $t$. Thus we define

\begin{definition} For the zero measure $0\in\Mea$, we define
$$\T^{\sf geo}_0\Mea:=\sqrt{\Mea}, \qquad\exp_0^\Mea\big(\sqrt\rho\big):=\rho,
\qquad \d_{\T_0} (\sqrt\rho,\sqrt\nu):=\HK(\rho,\nu).$$
\end{definition}
This is consistent with a
scalar multiplication on $\T^{\sf geo}_0\Mea$ given by $t\cdot \sqrt\rho=\sqrt{t^2\rho}$  and leads to
$$\exp_0^\Mea\big(t\sqrt\rho\big)=\exp_0^\Mea\big(\sqrt{t^2\rho}\big)=t^2\rho.$$

\begin{theorem} For all linearly independent $\sqrt\rho,\sqrt\nu\in\T^{\sf geo}_0\Mea
$,
$$\Sec^{\Mea}_{0}(\sqrt\rho,\sqrt\nu)=0.$$
\end{theorem}
\begin{proof} Since $\HK\big(\exp_0^\Mea(s\sqrt\rho),\exp_0^\Mea(t\sqrt\nu)\big)=\HK(s^2\rho,t^2\nu)=\d_{\T_0}(s\sqrt\rho,t\sqrt\nu)$ for all $s,t>0$,
\begin{align*}\Sec^\Mea_0(\sqrt\rho,\sqrt\nu):=
\!\!\!\limsup_{\tiny\begin{array}{c}
{s,t\to0}\\{t/s+s/t\textrm{ bdd}}\end{array}}\!\!\!
\frac3{\Lambda\,s^2\,t^2}\Big[\d^2_{\T_0} (s\sqrt\rho,t\sqrt\nu)-\HK^2\big(\exp^\Mea_0(s\sqrt\rho),\exp^\Mea_0(t\sqrt\nu)\Big]=0.
\end{align*}
\end{proof}

\subsubsection{The extended tangent space}
Combining Theorem \ref{Brenier} with the Reduction Lemma \ref{red} yields

\begin{theorem} Assume that $\M=\R^n$ and $\mu_0\ll\mm$. Then every  $\HK$-geodesic connecting $\mu_0$ and $\mu_1$  
is of the form
$$\mu_t=
\exp^{\Mea}_\mu\big(t\varphi\big)=\Pr\Big(
\exp^\Co\big(t\, (\nabla\varphi,2r\varphi)\big)_*\hat\mu\Big) +t^2\nu
$$
with a function $\varphi\in H^1(\mu)$ and a measure $\nu\in\Mea\big({\M\setminus B_{\pi/2}(\supp[\mu_0])}\big)$ satisfying $\nu\le\mu_1$.

The connecting $\HK$-geodesic is unique if 
$\d(x,y)\not=\frac\pi2$ for $\mu_0\otimes\mu_1$-a.e.~$(x,y)$.
In this case,
$$\HK^2(\mu_0,\mu_1)=\|\varphi\|_{H^1(\mu_0)}^2+|\nu|.$$
\end{theorem}

\begin{definition}
Extended tangent space: \ 
$\T^{\sf ext}_\mu\Mea=H^1(\mu)\otimes 
\sqrt\Mea\big({\M\setminus B_{\pi/2}(\supp[\mu])}\big)$  with
$$\big\|(\varphi_1,\sqrt{\nu_1}),(\varphi_2,\sqrt{\nu_2})
\big\|^2_{\T_\mu^{\sf ext}}:=
\|\varphi_1-\varphi_2\|^2_{H^1(\mu)}+\HK^2(\nu_1,\nu_2).$$
Scalar multiplication
$$t\,(\varphi,\sqrt\nu)=\big(t\varphi,\sqrt{t^2\nu}\big).$$
Exponential map
$$\exp_\mu^\Mea\big(\varphi,\sqrt\nu\big):=\Pr\Big(
\exp^\Co\big(\nabla\varphi,2r\varphi\big)_*\hat\mu\Big)+\nu.
$$
\end{definition}

\begin{theorem} For all $\mu\in\Mea$, all $\rho,\nu\in
\Mea\big(\M\setminus \overline{B_{\pi/2}(\supp[\mu])}\big)$, and all $\varphi,\psi\in H^1(\mu)$ for which $\Sec^\Mea_{\mu}(\varphi,\psi)$ exists,
$$\Sec^\Mea_{\mu}\big((\varphi,\sqrt\rho),(\psi,\sqrt\nu)\big)=
\frac{\Lambda_\mu({\varphi,\psi})}{\Lambda_\mu({\varphi,\rho,\psi,\nu})}\,\Sec^\Mea_{\mu}(\varphi,\psi)$$
with $\Lambda_\mu({\varphi,\psi}):=
|\varphi|_{H^1(\mu)}^2\cdot|\psi|_{H^1(\mu)}^2-\langle \varphi,\psi\rangle_{H^1(\mu)}^2$ as usual and 
$$\Lambda_\mu({\varphi,\rho,\psi,\nu}):=\Big[|\varphi|_{H^1(\mu)}^2+|\rho|
\Big]\cdot\Big[|\psi|_{H^1(\mu)}^2+|\nu|\Big]-\Big[\langle \varphi,\psi\rangle_{H^1(\mu)} +\langle\rho,\nu\rangle_\HK\Big]^2
$$
where
$\langle\rho,\nu\rangle_\HK:=\frac12\big(|\rho|+|\nu|-\HK^2(\rho,\nu)\big)
$.
\end{theorem}

\begin{proof}
 \begin{align*}\Lambda_\mu({\varphi,\rho,\psi,\nu})&\cdot
 \Sec^\Mea_{\mu}\big((\varphi,\sqrt\rho),(\psi,\sqrt\nu)\big)\\
 &=\!\!\! \lim_{\tiny\begin{array}{c}
{s,t\to0}\\{t/s+s/t\textrm{ bdd}}\end{array}}\!\!\!
\frac3{s^2\, t^2}\Bigg[ |s\varphi-t\psi|_{H^1(\mu)}^2
+\HK^2(s^2\rho,t^2\nu)\\
&\qquad\qquad
-
 \HK^2\Big(\exp_\mu^\Mea(s\varphi)+s^2\rho,\ \exp_\mu^\Mea(t\psi)+t^2\nu\Big)
\Bigg]\\
&=\!\!\! \lim_{\tiny\begin{array}{c}
{s,t\to0}\\{t/s+s/t\textrm{ bdd}}\end{array}}\!\!\!
\frac3{s^2\, t^2}\Bigg[ |s\varphi-t\psi|_{H^1(\mu)}^2
-
 \HK^2\Big(\exp_\mu^\Mea(s\varphi),\ \exp_\mu^\Mea(t\psi)\Big)
\Bigg]\\
&=\Lambda_\mu({\varphi,\psi})\cdot
 \Sec^\Mea_{\mu}\big(\varphi,\psi\big).
\end{align*}
\end{proof}

\subsection{Curvature bounds in the sense of Alexandrov}

Let $(\M,\d)$ be a complete separable geodesic space. Recall that $\HK$ is equivalently defined by means of the cone construction. 

\begin{theorem}\label{alex-lms} The geodesic space $(\Mea(\M), \HK)$ has nonnegative curvature in the sense of Alexandrov
 if and only if $(\M,\d)$ has curvature $\ge1$ in the sense of Alexandrov
 or $\dim\M=1$.
\end{theorem}
\begin{proof} Thm.~8.8 in \cite{LMS-inventiones} (and taking into account that every 1-dimensional space has locally curvature $\ge1$).
\end{proof}

\begin{remark}\label{alex-rem} No upper or lower curvature bounds in the sense of Alexandrov can hold with $\kappa\not=0$. More precisely, 
\begin{itemize}
\item For every $\kappa>0$,
$(\Mea(\M), \HK)$ does not have curvature $\ge \kappa$ in the sense of Alexandrov.
\item For every $\kappa>0$,
$(\Mea(\M), \HK)$ does not have curvature $\le -\kappa$ in the sense of Alexandrov.

\item If $(\Mea(\M), \HK)$ has curvature $\ge -\kappa$ in the sense of Alexandrov for some $\kappa>0$ then it has  curvature $\ge 0$.
\item If $(\Mea(\M), \HK)$ has curvature $\le \kappa$ in the sense of Alexandrov for some $\kappa>0$ then it has  curvature $\le 0$.
\end{itemize}
\end{remark}
\begin{proof}  All these assertions follow from the mass scaling property (cf.~Theorem \ref{def-dec-scal}
(iv)): if a quadruple $(\mu_1,\mu_2,\mu_3,\mu_4)$ in $\Mea(\M)$ satisfies the characterizing inequality for curvature $\ge\kappa$ (or curvature $\le \kappa$) in the sense of Alexandrov 
then for any $r>0$ the quadruple $(r\mu_1,r\mu_2,r\mu_3,r\mu_4)$ satisfies the characterizing inequality for curvature $\ge\frac1r\,\kappa$ (or curvature $\le \frac1r\,\kappa$, resp.).
\end{proof}

\begin{theorem} The geodesic space $(\Mea(\mathbb R^n), \HK)$ does not admit any upper or lower curvature bound in the sense of Alexandrov.
\end{theorem}

\begin{proof} According to previous Theorem and Remark, it only remains to prove that $(\Mea(\mathbb R^n), \HK)$ does not have nonpositive curvature.
To do so, choose  any $\rho\in\C^\infty_c(\R^n)$ and $\varphi,\psi\in\C_c^\infty(\{\rho>0\})$ such that
$F:=\nabla^2\varphi\,\nabla\psi$ is \emph{not a gradient}. In other words, such that
$F\not= \nabla\Big(\Delta_\rho^{-1}\mathrm{div}_\rho(F)\big)$.
Then 
$$\Sec^*_\mu(\varphi,\psi)=\frac3{\Lambda^*}\Big\| F-
 \nabla\Big(\Delta_\rho^{-1}\mathrm{div}_\rho(F)\big)
\Big\|^2_{L^2(\rho)}>0.
$$
Now assume that $(\Mea,\HK)$ where NPC. 
Define as before
$\varphi_\ell(x):=\varphi(\ell x)$, 
$\psi_\ell(x):=\psi(\ell x)$, $\rho_\ell(x):=\ell^{n+2}\,\rho(\ell x)$,
and $\mu_\ell:=\rho_\ell\Leb^n$
for $\ell\in\N$. Then
$$\Sec_{\mu_\ell}(\varphi_\ell,\psi_\ell)\le0$$
for all $\ell$ according to Theorem \ref{alex-vs-sec}.
According to Theorem \ref{scal-space}, spatial rescaling then implies
$$\Sec^*_\mu(\varphi,\psi)=\lim_{\ell\to\infty}\Sec_{\mu_\ell}(\varphi_\ell,\psi_\ell)\le0$$
which 
leads to a contradiction.
\end{proof}

For the analogous conclusion with the torus in the place of the Euclidean space, see  Corollary \ref{NNC-torus}.

\section{The Torus}
Let us now consider the $n$-dimensional torus $\M:=\mathbb T^n:=(\mathbb S^1)^n$ with
$\mathbb T=\mathbb S^1=\R/(2\pi\Z)$ and $d\mu(x):=(2\pi)^{-n}dx$ the  normalized volume measure.
\subsection{Hellinger-Kantorovich geometry for the torus}
In $n=1$, for $k\in\Z$ and $x\in\mathbb T$  put $f_k(x)=e^{ikx}$ and 
$$e_k(x):=\begin{cases}
\sqrt2\, \cos(kx), &k>0\\
\sqrt2\, \sin(kx), &k<0\\
1, &k=0.
\end{cases}$$
Then
$$e_k(x)=\frac1{\sqrt2}\sum_{\sigma\in\{-1,1\}} s_k^\sigma\, f_{\sigma k}(x)$$
with complex  unit numbers  $s_k$ for $k\in\Z$ 
defined as 
$$s_k:=\begin{cases}
1, &k>0\\
-i, &k<0\\
\frac{1-i}{\sqrt 2}, &k=0
\end{cases},
\quad\textrm{and thus }\ 
s_k^{-1}=\overline s_k=\begin{cases}
1, &k>0\\
i\,  &k<0\\
\frac{1+i}{\sqrt 2}, &k=0
\end{cases}.
$$
In the multidimensional case, for parameters $k\in\Z^n$ and  points $x\in \mathbb T^n$ define
$$f_k(x):=\prod_{p=1}^n f_p(x_p)=e^{i\langle k,x\rangle}$$
and 
\begin{align*}
e_k(x):=\prod_{p=1}^n e_p(x_p)&=
2^{-n/2}\sum_{\sigma\in\{-1,1\}^n} \prod_{p=1}^n s_{k_p}^{\sigma_p}\, f_{\sigma_pk_p}(x_p)\\
&=2^{-n/2}\sum_{\sigma\in\{-1,1\}^n} s_k^\sigma\, f_{k^\sigma}(x)
\end{align*}
with 
$s_k^\sigma:=\prod_{p=1}^n s_{k_p}^{\sigma_p}\in\mathbb C, |s_k^\sigma|=1$ and \emph{conjugate} parameters $$k^\sigma:=(\sigma_1k_1,\ldots,\sigma_nk_n)\in\Z^n.$$
Note that
$$\Delta e_k=-|k|^2\,e_k$$
with $|k|:=\left(\sum_p k_p^2\right)^{1/2}$.

\begin{lemma} The family
$(e_k)_{ k\in\Z^n}$
is a complete ON basis of $L^2(\M,\mu)$ and a complete orthogonal basis of  $H^1(\M,\mu)$. More precisely,
$$\int e_k\, e_\ell\,d\mu=0, \qquad
\int \langle\nabla e_k, \nabla e_\ell\rangle\,d\mu=0$$
for all $k\not=\ell$ and
$$\int e_k^2d\mu=1, \qquad \int |\nabla e_k|^2d\mu=|k|^2.$$
Thus in particular,
\begin{align*} \Lambda_{k\ell}&:=|e_{k}|_{H^1(\mu)}^2\cdot|e_{\ell}|_{H^1(\mu)}^2-\langle e_{k},e_{\ell}\rangle_{H^1(\mu)}^2
=(4+|k|^2)\cdot(4+|\ell|^2).
\end{align*}
for all  $k\not=\ell$.
\end{lemma}
\begin{proof}
Straightforward calculations yield
$$\nabla e_k=2^{-n/2}\sum_{\sigma\in\{-1,1\}^n} s_k^\sigma\, \nabla f_{k^\sigma}=i\,2^{-n/2}\sum_{\sigma\in\{-1,1\}^n} s_k^\sigma\,  f_{k^\sigma}\, k^\sigma
$$
and
\begin{align*}\int |\nabla e_k|^2d\mu&=|k|^2, \qquad \int |e_k|^2d\mu=1.
\end{align*}
Similarly, $\int |\nabla e_\ell|^2d\mu=|\ell|^2$, $\int |e_\ell|^2d\mu=1$.
Furthermore,
$\int e_k\,e_\ell\,d\mu=0$ and $\Delta e_k=-|k|^2 e_k$. Thus

\begin{align*}
\int \langle \nabla e_k,\nabla e_\ell\rangle d\mu=-\int \Delta e_k\, e_\ell\,d\mu=|k|^2\,\int e_k\, e_\ell\,d\mu=
|k|^2\,\one_{k=\ell}.
\end{align*}
\end{proof}

\begin{lemma} In the case $n=1$,
\begin{align*}
s_k^\sigma\,s_\ell^\tau\,\overline{s_k^{\sigma'}}\,\overline{s_\ell^{\tau'}}\,\one_{
k^\sigma+\ell^\tau=k^{\sigma'}+\ell^{\tau'}}&=
\one_{\sigma'=\sigma,\tau'=\tau}+\Big(\one_{k=-\ell\not=0}\one_{\sigma\tau=1}-\one_{k=\ell\not=0}\one_{\sigma\tau=-1}\Big)\,
\one_{\sigma'=-\sigma,\tau'=-\tau}\\
&-i\,\sigma\one_{k=0}\one_{\sigma'=-\sigma,\tau'=\tau}
-i\,\tau\one_{\ell=0}\one_{\sigma'=\sigma,\tau'=-\tau}
-\sigma\tau\,\one_{k=0,\ell=0}\, \one_{\sigma'=-\sigma,\tau'=-\tau}.
\end{align*}
\end{lemma}

\begin{proof} We distinguish 4 cases.
\begin{description}
\item[\it Generic case  $k\not=\pm\ell$ and $k\ell\not=0$] here  $
(\sigma-\sigma')k+(\tau-\tau')\ell=0$ implies $\sigma=\sigma', \tau=\tau'$ and thus
$$s_k^\sigma\,,\overline{s_k^{\sigma'}}=1, \qquad s_\ell^\tau\,\overline{s_\ell^{\tau'}}=1.$$

\item[\it Anticorrelated case $k=-\ell\not=0$] here $(\sigma-\sigma')k+(\tau-\tau')\ell=0$ implies
$$\sigma+\tau=\sigma'+\tau'$$ and thus
$$s_k^\sigma\,s_\ell^\tau\,\overline{s_k^{\sigma'}}\,\overline{s_\ell^{\tau'}}=s_{k}^{\sigma+\tau-\sigma'-\tau'}=1$$

\item[\it Correlated case $k=\ell\not=0$] here $(\sigma-\sigma')k+(\tau-\tau')\ell=0$ implies
$$\sigma+\tau=-(\sigma'+\tau')$$
and thus
$$s_k^\sigma\,s_\ell^\tau\,\overline{s_k^{\sigma'}}\,\overline{s_\ell^{\tau'}}=s_k^{\sigma-\sigma'}\,s_{-k}^{\tau-\tau'}=
\pm1.$$

\item[\it Degenerate case $k\ell=0$]  e.g. if $k=0$, $\ell\not=0$ then $(\sigma-\sigma')k+(\tau-\tau')\ell=0$ implies
$\tau=\tau'$ whereas $\sigma'$ is arbitrary.
Thus
$$s_k^\sigma\,s_\ell^\tau\,\overline{s_k^{\sigma'}}\,\overline{s_\ell^{\tau'}}=\Big(\frac{1-i}{\sqrt2}\Big)^{\sigma-\sigma'}\in\{i,-i,1\}
.$$
If $k=\ell=0$ then $$s_k^\sigma\,s_\ell^\tau\,\overline{s_k^{\sigma'}}\,\overline{s_\ell^{\tau'}}=\Big(\frac{1-i}{\sqrt2}\Big)^{\sigma+\tau-\sigma'-\tau'}\in\{i,-i,1,-1\}
.$$

\end{description}
Then we verify the claim by considering all 16 possible choices for $\sigma,\tau,\sigma_,\tau'$.
\end{proof}

\begin{proposition} For general $n$,
\begin{align*}
s_k^\sigma\,s_\ell^\tau\,\overline{s_k^{\sigma'}}\,\overline{s_\ell^{\tau'}}\,\one_{
k^\sigma+\ell^\tau=k^{\sigma'}+\ell^{\tau'}}&= 
\prod_{p=1}^n
\bigg[\one_{\sigma'_p=\sigma_p,\tau'_p=\tau_p}\\
&\qquad\quad
+\Big(\one_{k_p=-\ell_p\not=0}\one_{\sigma_p\tau_p=1}-\one_{k_p=\ell_p\not=0}\one_{\sigma_p\tau_p=-1}\Big)\,
\one_{\sigma'_p=-\sigma_p,\tau'_p=-\tau_p}\\
&\qquad\quad-i\,\sigma_p\one_{k_p=0}\one_{\sigma'_p=-\sigma_p,\tau'_p=\tau_p}
-i\,\tau_p\one_{\ell_p=0}\one_{\sigma_p'=\sigma_p,\tau_p'=-\tau_p}\\
&\qquad\quad
-\sigma_p\tau_p\,\one_{k_p=0,\ell_p=0}\, \one_{\sigma_p'=-\sigma_p,\tau_p'=-\tau_p}\bigg].
\end{align*}

\end{proposition}

\subsection{Sectional curvature of $\Mea(\mathbb T^n)$}

\subsubsection{The lifted part}

\begin{theorem}\label{torus-lifted} If $n\ge2$, for all linearly independent $k,\ell$,
\begin{align*}\Sec_\mu^\uparrow\big(e_{k},e_{\ell} \big)&=
-\frac{1}{4^n\, (4+|k|^2)\,(4+|\ell|^2)}
\sum_{\sigma,\tau,\sigma',\tau'\in\{-1,1\}^n}\Big[ \langle k^\sigma,k^{\sigma'}\rangle\langle \ell^\tau,\ell^{\tau'}\rangle
  -\langle k^\sigma,\ell^{\tau}\rangle
  \langle k^{\sigma'},\ell^{\tau'}\rangle
  \Big]\\
  &\quad\cdot
\prod_{p=1}^n
\bigg[\one_{\sigma'_p=\sigma_p,\tau'_p=\tau_p}+\Big(\one_{k_p=-\ell_p\not=0}\one_{\sigma_p\tau_p=1}-\one_{k_p=\ell_p\not=0}\one_{\sigma_p\tau_p=-1}\Big)\,\one_{\sigma'_p=-\sigma_p,\tau'_p=-\tau_p}\\
&\qquad\quad-i\,\sigma_p\one_{k_p=0}\one_{\sigma'_p=-\sigma_p,\tau'_p=\tau_p}
-i\,\tau_p\one_{\ell_p=0}\one_{\sigma_p'=\sigma_p,\tau_p'=-\tau_p}
-\sigma_p\tau_p\one_{k_p=0,\ell_p=0}\, \one_{\sigma_p'=-\sigma_p,\tau_p'=-\tau_p}\bigg].
\end{align*}
In particular, for generic $k,\ell$, 
$$\Sec_\mu^\uparrow\big(e_{k},e_{\ell} \big)=-\frac{|k|^2\, |\ell|^2-\langle k^2,\ell^2\rangle}{(4+|k|^2)\,(4+|\ell|^2)}
$$
with
$\langle k^2,\ell^2\rangle:=\sum_p k_p^2\ell_p^2$.

On the contrary, if $n=1$, for all linearly independent $k,\ell$, 
$$Sec_\mu^\uparrow\big(e_{k},e_{\ell} \big)=0.$$
\end{theorem}

\begin{proof} Assume $n\ge2$.
\begin{align*}
-\Lambda_{k\ell}&\cdot\Sec_\mu^\uparrow\big(e_{k},e_{\ell} \big)= \int_{\M}\Big(|\nabla e_{k}|^2\, |\nabla e_{\ell}|^2-|\langle \nabla e_{k},\nabla e_{\ell}\rangle|^2
\Big)\,d\mu\\
&=
2^{-2n}\sum_{\sigma,\tau,\sigma',\tau'\in\{-1,1\}^n} s_k^\sigma\, s_k^{-\sigma'} \,s_\ell^\tau\, s_\ell^{-\tau'}
\Big[ \langle k^\sigma,k^{\sigma'}\rangle
  \langle \ell^\tau,\ell^{\tau'}\rangle
  -\langle k^\sigma,\ell^{\tau}\rangle
  \langle k^{\sigma'},\ell^{\tau'}\rangle
  \Big]\,\one_{k^\sigma+\ell^\tau=k^{\sigma'}+\ell^{\tau'}}\\
  &=2^{-2n}\sum_{\sigma,\tau,\sigma',\tau'\in\{-1,1\}^n}\Big[ \langle k^\sigma,k^{\sigma'}\rangle\langle \ell^\tau,\ell^{\tau'}\rangle
  -\langle k^\sigma,\ell^{\tau}\rangle
  \langle k^{\sigma'},\ell^{\tau'}\rangle
  \Big]\\
  &\quad\cdot
\prod_{p=1}^n
\bigg[\one_{\sigma'_p=\sigma_p,\tau'_p=\tau_p}+\Big(\one_{k_p=-\ell_p\not=0}\one_{\sigma_p\tau_p=1}-\one_{k_p=\ell_p\not=0}\one_{\sigma_p\tau_p=-1}\Big)\,
\one_{\sigma'_p=-\sigma_p,\tau'_p=-\tau_p}\\
&\qquad\qquad-i\,\sigma_p\one_{k_p=0}\one_{\sigma'_p=-\sigma_p,\tau'_p=\tau_p}
-i\,\tau_p\one_{\ell_p=0}\one_{\sigma_p'=\sigma_p,\tau_p'=-\tau_p}
-\sigma_p\tau_p\one_{k_p=0,\ell_p=0}\, \one_{\sigma_p'=-\sigma_p,\tau_p'=-\tau_p}\bigg].
\end{align*}
For generic $k,\ell$, we have that $\sigma'=\sigma, \tau'=\tau$ and thus
$\langle k^\sigma,k^{\sigma'}\rangle\langle \ell^\tau,\ell^{\tau'}\rangle
=|k^\sigma|^2\, |\ell^\tau|^2=|k^2\, |\ell|^2$. Furthermore,
\begin{align*}2^{-2n}\sum_{\sigma,\tau\in\{-1,1\}^n}
\langle k^\sigma,\ell^{\tau}\rangle^2
 &=2^{-n}\sum_{\sigma\in\{-1,1\}^n}
\langle k^\sigma,\ell\rangle^2\\
&=2^{-n}\sum_{\sigma\in\{-1,1\}^n}\bigg(\sum_p\sigma_pk_p\ell_p\bigg)\bigg(\sum_q\sigma_qk_q\ell_q\bigg)\\
&=2^{-n}\sum_{\sigma\in\{-1,1\}^n}\sum_p\sigma^2_pk^2_p\ell^2_p=\sum_p\sigma^2_pk^2_p\ell^2_p.
\end{align*}
\end{proof}

\begin{theorem}\label{lower-bound} For all  $k\not=\ell$,
$$\Sec_\mu^\uparrow\big(e_{k},e_{\ell} \big)\ge -\Big(\frac32\Big)^{n}.$$
For generic $k,\ell$,
$$\Sec_\mu^\uparrow\big(e_{k},e_{\ell} \big)\ge -1.$$
\end{theorem}

\begin{proof} In the generic case, the claim is obvious. In the general case, we have to estimate
\begin{align*}\Delta:=&2^{-2n}\sum_{\sigma,\tau,\sigma',\tau'\in\{-1,1\}^n} \prod_{p=1}^n t_{\sigma_p,\tau_p,\sigma_p',\tau_p',k_p,\ell_p} \cdot\langle k^\sigma,k^{\sigma'}\rangle\langle \ell^\tau,\ell^{\tau'}\rangle
\end{align*}
with 
\begin{align*}
t_{\sigma_p,\tau_p,\sigma_p',\tau_p',k_p,\ell_p}&:=
\one_{\sigma'_p=\sigma_p,\tau'_p=\tau_p}+\Big(\one_{k_p=-\ell_p\not=0}\one_{\sigma_p\tau_p=1}-\one_{k_p=\ell_p\not=0}\one_{\sigma_p\tau_p=-1}\Big)\,\one_{\sigma'_p=-\sigma_p,\tau'_p=-\tau_p}\\
&-i\,\sigma_p\one_{k_p=0}\one_{\sigma'_p=-\sigma_p,\tau'_p=\tau_p}
-i\,\tau_p\one_{\ell_p=0}\one_{\sigma_p'=\sigma_p,\tau_p'=-\tau_p}
-\sigma_p\tau_p\one_{k_p=0,\ell_p=0}\, \one_{\sigma_p'=-\sigma_p,\tau_p'=-\tau_p}.
\end{align*}
To do so, for any $q\in\{1,\ldots,n\}$ we decompose the sum and the product into $q$-dependent and $q$-independent terms:
\begin{align*}
\Delta&=
2^{-2(n-1)}
\sum_{\sigma_q,\tau_q,\sigma_q',\tau_q'\in\{-1,1\}, q\not=p} \prod_{p\not=q} t_{\sigma_p,\tau_p,\sigma_p',\tau_p',k_p,\ell_p}\\
&\qquad\cdot
\Bigg[\frac14
\sum_{\sigma_q,\tau_q,\sigma_q',\tau_q'\in\{-1,1\}} 
t_{\sigma_q,\tau_q,\sigma_q',\tau_q',k_q,\ell_q}\,
\langle k^\sigma,k^{\sigma'}\rangle\langle \ell^\tau,\ell^{\tau'}\rangle\Bigg]\\
&=2^{-2(n-1)}
\sum_{\sigma_q,\tau_q,\sigma_q',\tau_q'\in\{-1,1\}, q\not=p} \prod_{p\not=q} t_{\sigma_p,\tau_p,\sigma_p',\tau_p',k_p,\ell_p}\\
&\qquad\cdot
\Bigg[\frac14
\sum_{\sigma_q,\tau_q,\sigma_q',\tau_q'\in\{-1,1\}} 
t_{\sigma_q,\tau_q,\sigma_q',\tau_q',k_q,\ell_q}\,
\bigg(\sigma_q\sigma'_qk_q^2+\sum_{r\not=q}\sigma_r\sigma'_rk_r^2\bigg)\cdot
\bigg(\tau_q\tau'_q\ell_q^2+\sum_{r\not=q}\tau_r\tau'_r\ell_r^2\bigg)
\Bigg].
\end{align*}
The expression in the last line can be rewritten as 
\begin{align*}
&\frac14\sum_{\sigma_q,\tau_q,\sigma_q',\tau_q'\in\{-1,1\}} 
t_{\sigma_q,\tau_q,\sigma_q',\tau_q',k_q,\ell_q}\,
\big(\sigma_q\sigma'_qk_q^2+A_1\big)\,
\big(\tau_q\tau'_q\ell_q^2+A_2\big)\\
&=\frac14\sum_{\sigma_q,\tau_q\in\{-1,1\}} \bigg[\big(k_q^2+A_1\big)\,
\big(\ell_q^2+A_2\big)\cdot\Big[1{+\one_{k_q=-\ell_q\not=0}\one_{\sigma_q\tau_q=1}-\one_{k_q=\ell_p\not=0}\one_{\sigma_q\tau_q=-1}}-\sigma_q\tau_q\one_{k_q=0,\ell_q=0}
\Big]
\\
&\qquad -i\,\sigma_q\one_{k_q=0} A_1(\ell_q^2+A_2)-i\,\tau_q\,\one_{\ell_q=0}(k_q^2+A_1)A_2\bigg]\\
&=\big(k_q^2+A_1\big)\,
\big(\ell_q^2+A_2\big)\,\Big[1+\frac12\one_{k_q=-\ell_q}-\frac12\one_{k_q=\ell_q}\Big]
\end{align*}
Iterating this w.r.t. $q$ yields
\begin{align*}\Delta&=\bigg(\sum_qk_q^2\bigg)\,
\bigg(\sum_q\ell_q^2\bigg)\,\prod_q
\Big[1+\frac12\one_{k_q=-\ell_q}-\frac12\one_{k_q=\ell_q}\Big]
\le
\Big(\frac32\Big)^{n} |k|^2\, |\ell|^2.\end{align*}
This proves the claim.
\end{proof}

\subsubsection{The twisted part}
For $k,\ell\in\Z^n$ consider
$$\tilde F_{k\ell}:=\frac12\Big(\nabla^2f_k\cdot\nabla f_\ell-\nabla^2f_\ell\cdot\nabla f_k\Big)
=\frac{-i}2(k-\ell)\, \langle k,\ell \rangle\,e^{i\langle k+\ell,x\rangle}$$
and
$$\tilde\eta_{k\ell}=-(4-\Delta)^{-1}\big(\textrm{div}\, \tilde F_{k\ell}\big)=
\frac{-1}2
\frac{|k|^2-|\ell|^2}{4+|k+\ell|^2}\, \langle k,\ell \rangle\,e^{i\langle k+\ell,x\rangle}.
$$
Then
\begin{align*}F_{k\ell}&:=\frac12\Big(\nabla^2e_k\cdot\nabla e_\ell-\nabla^2e_\ell\cdot\nabla e_k\Big)\\
&=
\frac1{2^{n}}\sum_{\sigma,\tau\in\{-1,1\}^n}s_k^\sigma\,s_\ell^\tau\, \tilde F_{k^\sigma\ell^\tau}\\
&=\frac{-i}{2^{n+1}}\sum_{\sigma,\tau\in\{-1,1\}^n}s_k^\sigma\,s_\ell^\tau\,\big(k^\sigma-\ell^\tau\big) \langle k^\sigma,\ell^\tau \rangle\,e^{i\langle k^\sigma+\ell^\tau,x\rangle},
\end{align*}
\begin{align*}
\textrm{div}\,F_{k\ell}&:=\frac1{2^{n+1}}\sum_{\sigma,\tau\in\{-1,1\}^n}s_k^\sigma\,s_\ell^\tau\,\big(|k^\sigma|^2-|\ell^\tau|^2\big) \langle k^\sigma,\ell^\tau \rangle\,e^{i\langle k^\sigma+\ell^\tau,x\rangle},
\end{align*}
and
\begin{align*}\eta_{k\ell}&:=-(4-\Delta)^{-1}\big(\textrm{div}\, F_{k\ell}\big)\\
&=\frac{-1}{2^{n+1}}\sum_{\sigma,\tau\in\{-1,1\}^n}s_k^\sigma\,s_\ell^\tau\
\frac{|k|^2-|\ell|^2}{4+|k^\sigma+\ell^\tau|^2}\, \langle k^\sigma,\ell^\tau \rangle\,e^{i\langle k^\sigma+\ell^\tau,x\rangle}.
\end{align*}
Thus
\begin{align*}
\int F\,\bar F\,dm
&= \frac1{2^{2n+2}}\sum_{\sigma,\tau,\sigma',\tau'\in\{-1,1\}^n}s_k^\sigma\,s_\ell^\tau\,\overline{s_k^{\sigma'}}\,\overline{s_\ell^{\tau'}}\,\big\langle k^\sigma-\ell^\tau, k^{\sigma'}-\ell^{\tau'}\big\rangle \langle k^\sigma,\ell^\tau \rangle\, \langle k^{\sigma'},\ell^{\tau'} \rangle\,
\one_{
k^\sigma+\ell^\tau=k^{\sigma'}+\ell^{\tau'}}
\end{align*}
and
\begin{align*}
\int \ \textrm{div}F\,\bar \eta \,dm
&= \frac{-1}{2^{2n+2}}\sum_{\sigma,\tau,\sigma',\tau'\in\{-1,1\}^n}s_k^\sigma\,s_\ell^\tau\,\overline{s_k^{\sigma'}}\,\overline{s_\ell^{\tau'}}\,
\frac{(|k|^2-|\ell|^2)^2}{4+|k^\sigma+\ell^\tau|^2}\,
\langle k^\sigma,\ell^\tau \rangle\, \langle k^{\sigma'},\ell^{\tau'} \rangle\,
\one_{
k^\sigma+\ell^\tau=k^{\sigma'}+\ell^{\tau'}}
\end{align*}
hence
\begin{align*}
\delta_{k\ell}^2&:=\int F\,\bar F\,dm+\int \ \textrm{div}F\,\bar \eta \,dm\\
&= \frac1{2^{2n+2}}\sum_{\sigma,\tau,\sigma',\tau'\in\{-1,1\}^n}s_k^\sigma\,s_\ell^\tau\,\overline{s_k^{\sigma'}}\,\overline{s_\ell^{\tau'}}\,\bigg[\big\langle k^\sigma-\ell^\tau, k^{\sigma'}-\ell^{\tau'}\big\rangle
-\frac{(|k|^2-|\ell|^2)^2}{4+|k^\sigma+\ell^\tau|^2}
\bigg]\langle k^\sigma,\ell^\tau \rangle\, \langle k^{\sigma'},\ell^{\tau'} \rangle\,
\one_{
k^\sigma+\ell^\tau=k^{\sigma'}+\ell^{\tau'}}
\end{align*}

\begin{theorem}\label{torus-twisted}
\begin{align*}
\Sec_\mu^\nabla(e_k,e_\ell)
&=\frac{3}{2^{2n+2}}\,\frac1
 {(4+|k|^2)\cdot(4+|\ell|^2)}\\
 &\cdot\sum_{\sigma,\tau,\sigma',\tau'\in\{-1,1\}^n}
\bigg[\big\langle k^\sigma-\ell^\tau, k^{\sigma'}-\ell^{\tau'}\big\rangle
-\frac{(|k|^2-|\ell|^2)^2}{4+|k^\sigma+\ell^\tau|^2}
\bigg]\langle k^\sigma,\ell^\tau \rangle\, \langle k^{\sigma'},\ell^{\tau'} \rangle\\
&\qquad\cdot
\prod_{p=1}^n
\bigg[\one_{\sigma'_p=\sigma_p,\tau'_p=\tau_p}+\Big(\one_{k_p=-\ell_p\not=0}\one_{\sigma_p\tau_p=1}-\one_{k_p=\ell_p\not=0}\one_{\sigma_p\tau_p=-1}\Big)\,
\one_{\sigma'_p=-\sigma_p,\tau'_p=-\tau_p}\\
&\qquad\qquad-i\,\sigma_p\one_{k_p=0}\one_{\sigma'_p=-\sigma_p,\tau'_p=\tau_p}
-i\,\tau_p\one_{\ell_p=0}\one_{\sigma_p'=\sigma_p,\tau_p'=-\tau_p}
-\sigma_p\tau_p\one_{k_p=0,\ell_p=0}\, \one_{\sigma_p'=-\sigma_p,\tau_p'=-\tau_p}\bigg].
\end{align*}
\end{theorem}

\begin{corollary} In the generic case where $k_p\not=\pm\ell_p$  and $k_p\ell_p\not=0$ ($\forall p$),
\begin{align*}
\delta_{k\ell}^2&=\frac1{ 2^{2n+2}}\sum_{\sigma,\tau\in\{-1,1\}^n}\bigg[\big|k^\sigma-\ell^\tau\big|^2
-\frac{(|k|^2-|\ell|^2)^2}{4+|k^\sigma+\ell^\tau|^2}
\bigg]\langle k^\sigma,\ell^\tau \rangle^2\\
&=\frac1{ 2^{n+2}}\sum_{\sigma\in\{-1,1\}^n}\bigg[\big|k^\sigma-\ell\big|^2
-\frac{(|k|^2-|\ell|^2)^2}{4+|k^\sigma+\ell|^2}
\bigg]\langle k^\sigma,\ell \rangle^2\\
&= \frac1{2^{n}}\sum_{\sigma\in\{-1,1\}^n}
\tilde\delta^2(k^\sigma,\ell)
\end{align*}
with 
\begin{align*}
\tilde\delta^2(k,\ell)&:= \frac14\bigg[\big|k-\ell\big|^2
-\frac{(|k|^2-|\ell|^2)^2}{4+|k+\ell|^2}
\bigg]\langle k,\ell \rangle^2\\
&=
\Big[|k-\ell|^2+|k|^2\,|\ell|^2-\langle k,\ell\rangle^2\Big]\frac{\langle k,\ell \rangle^2}{
4+|k+\ell|^2}.
\end{align*}
\end{corollary} 

\begin{remark}
$\frac1{2^{n}}\sum_{\sigma\in\{-1,1\}^n}\langle k^\sigma,\ell \rangle^2=\langle k^2,\ell^2 \rangle$.
\end{remark}

\begin{corollary} In the generic case,
\begin{align*}
\Sec_\mu^\nabla(e_k,e_\ell)
&=\frac{3}{2^{n+2}}\,\sum_{\sigma}
\bigg[| k^\sigma-\ell|^2
-
\frac{\big(|k|^2-|\ell|^2
\big)^2}{4+|k^\sigma+\ell|^2}
\bigg]
\, \frac{
 \langle k^\sigma,\ell\rangle^2}
 {(4+|k|^2)\cdot(4+|\ell|^2)}\\
&
= \frac{1}{2^n}\,\sum_{\sigma} S^\nabla(k^\sigma,\ell)
\end{align*}
with 
$$S^\nabla(k,\ell):=\frac34\,
\bigg[| k-\ell|^2
-
\frac{\big(|k|^2-|\ell|^2
\big)^2}{4+|k+\ell|^2}
\bigg]
\, \frac{
 \langle k,\ell\rangle^2}
 {(4+|k|^2)\cdot(4+|\ell|^2)}.
$$
\end{corollary}

\subsubsection{Summing up}

\begin{theorem} In the generic case,
\begin{align*}\Sec_\mu(e_k,e_\ell)&=\Sec_\mu^\uparrow(e_k,e_\ell)+
\Sec_\mu^\nabla(e_k,e_\ell)\\
&= \frac{1}{2^n}\,\sum_{\sigma} S(k^\sigma,\ell)
\end{align*}
with
\begin{align*}S(k,\ell)&:=S^\uparrow(k,\ell)+S^\nabla(k,\ell)\\
&=\frac{1}{(4+|k|^2)\,(4+|\ell|^2)}
\bigg(-|k|^2\, |\ell|^2+\langle k^2,\ell^2\rangle
+\frac34\,
\bigg[| k-\ell|^2
-
\frac{\big(|k|^2-|\ell|^2
\big)^2}{4+|k+\ell|^2}
\bigg]
\, 
 \langle k,\ell\rangle^2 \bigg).\end{align*}
\end{theorem}

\begin{theorem}\label{universal bounds} For arbitrary $k,\ell$ with $k\not=\ell$,
$$-\Big(\frac32\Big)^{n}\le \Sec_\mu\big(e_{k},e_{\ell} \big)\le \frac32\, 4^n\, \Big(|k|^2+|\ell|^2\Big).$$
\end{theorem}
\begin{proof} Since, $\Sec_\mu^\nabla\ge0$,
the lower bound follows from the lower bound for $\Sec_\mu^\uparrow$
as stated in Theorem \ref{lower-bound}. On the other hand,   Theorem \ref{torus-twisted} implies that
\begin{align*}
\Sec_\mu^\nabla(e_k,e_\ell)
&=\frac{3}{2^{2n+2}}\,\frac1
 {(4+|k|^2)\cdot(4+|\ell|^2)}\\
 &\cdot\sum_{\sigma,\tau,\sigma',\tau'\in\{-1,1\}^n}
\bigg[\big\langle k^\sigma-\ell^\tau, k^{\sigma'}-\ell^{\tau'}\big\rangle
-\frac{(|k|^2-|\ell|^2)^2}{4+|k^\sigma+\ell^\tau|^2}
\bigg]\langle k^\sigma,\ell^\tau \rangle\, \langle k^{\sigma'},\ell^{\tau'} \rangle\\
&\qquad\cdot
\prod_{p=1}^n
\bigg[\one_{\sigma'_p=\sigma_p,\tau'_p=\tau_p}+\Big(\one_{k_p=-\ell_p\not=0}\one_{\sigma_p\tau_p=1}-\one_{k_p=\ell_p\not=0}\one_{\sigma_p\tau_p=-1}\Big)\,\one_{\sigma'_p=-\sigma_p,\tau'_p=-\tau_p}\\
&\qquad\qquad-i\,\sigma_p\one_{k_p=0}\one_{\sigma'_p=-\sigma_p,\tau'_p=\tau_p}
-i\,\tau_p\one_{\ell_p=0}\one_{\sigma_p'=\sigma_p,\tau_p'=-\tau_p}
-\sigma_p\tau_p\one_{k_p=0,\ell_p=0}\, \one_{\sigma_p'=-\sigma_p,\tau_p'=-\tau_p}\bigg]\\
&\le\frac{3}{2^{2n+2}}\,\frac1
 {(4+|k|^2)\cdot(4+|\ell|^2)}
 \cdot\sum_{\sigma,\tau,\sigma',\tau'\in\{-1,1\}^n}
\bigg|\big\langle k^\sigma-\ell^\tau, k^{\sigma'}-\ell^{\tau'}\big\rangle
\cdot \langle k^\sigma,\ell^\tau \rangle\, \langle k^{\sigma'},\ell^{\tau'} \rangle\bigg|\\
&\le\frac{3}{2^{2n+2}}\,\frac1
 {|k|^2\cdot |\ell|^2}
 \cdot\sum_{\sigma,\tau,\sigma',\tau'\in\{-1,1\}^n}
\big( |k^\sigma|+|\ell^\tau|\big)\cdot
\big(|k^{\sigma'}| + |\ell^{\tau'}|\big)
\cdot |k^\sigma|\, |\ell^\tau|\, |k^{\sigma'}|\,|\ell^{\tau'}| \\
&=3 \cdot 2^{2n-2}\big(|k|+|\ell|\big)^2.
\end{align*}
Since $\Sec_\mu^\uparrow\le0$,
this implies
the upper bound.
\end{proof}
\subsection{Some particular cases}
\subsubsection{Vertical  case}
\begin{proposition}Assume $k=0$.
Then for all $\ell\not=0$,
$$\Sec_\mu(e_k,e_\ell)=0.$$
\end{proposition}

Hence, in particular,
$$\Ric(e_k,e_k):=\|e_k\|^2_{H^1(\mu)}\cdot \sum_{\ell\not=k}\Sec_\mu(e_k,e_\ell)=0.$$
\begin{proof} If $k=0$ then $\Sec^\uparrow_\mu(e_k,e_\ell)=0$ according to Theorem \ref{torus-lifted} and Proposition \ref{1d-hell},
$\Sec^\nabla_\mu(e_k,e_\ell)=0$ according to Theorem \ref{torus-twisted}. 
\end{proof}

\subsubsection{1-dimensional case}
\begin{proposition} Assume $n=1$.
Then for all $k,\ell\in\Z\setminus\{0\}$, $k\not=\pm\ell$,
$$\Sec_\mu(e_k,e_\ell)=3k^2\ell^2 \frac{ 4k^2+4\ell^2+k^4+6k^2\ell^2+\ell^4}{(4+k^2)(4+\ell^2)(4+(k-\ell)^2)(4+(k+\ell)^2)}>0.$$
Hence, for every $k\not=0$,
$$\lim_{\ell\to\pm\infty}\Sec_\mu(e_k,e_\ell)=\frac{3k^2}{4+k^2}\ \in \ \bigg[\,\frac35,\,3\,\bigg)$$
and, in particular,
$$\Ric(e_k,e_k):=\|e_k\|^2_{H^1(\mu)}\cdot \sum_{\ell\not=k}\Sec_\mu(e_k,e_\ell)=+\infty.$$
\end{proposition}
\begin{proof} Obviously $\Sec^\uparrow_\mu(e_k,e_\ell)=0$ and 
\begin{align*}
\frac43\Sec^\nabla_\mu(e_k,e_\ell)&=\frac{k^2\ell^2}{2(4+k^2)(4+\ell^2)}\bigg[
(k-\ell)^2-\frac{(k^2-\ell^2)^2}{4+(k+\ell)^2}+(k+\ell)^2-\frac{(k^2-\ell^2)^2}{4+(k-\ell)^2}\bigg]\\
&=\frac{k^2\ell^2}{2(4+k^2)(4+\ell^2)}\bigg[
\frac{4(k-\ell)^2}{4+(k+\ell)^2}+\frac{4(k+\ell)^2}{4+(k-\ell)^2}\bigg]\\
&=4k^2\ell^2 \frac{ 4k^2+4\ell^2+k^4+6k^2\ell^2+\ell^4}{(4+k^2)(4+\ell^2)(4+(k-\ell)^2)(4+(k+\ell)^2)}.
\end{align*}
\end{proof}
\subsubsection{Super-orthogonal case}
\begin{proposition}Assume $n\ge 2$, $k,\ell\not=0$, and $k_p\ell_p=0$ ($\forall p=1,\ldots,n$). Then
$$\Sec_\mu(e_k,e_\ell)=-\frac{|k|^2\,|\ell|^2}{(4+|k|^2)\,(4+|\ell|^2)}\ \in \ \bigg(-1,-\frac1{25}\bigg].$$
\end{proposition}
\begin{proof} Under the given assumption, $\langle k^\sigma,\ell^\tau\rangle=0$ for all $\sigma,\tau\in\{-1,1\}^n$. Thus according to Theorem \ref{torus-twisted},
$\Sec^\nabla_\mu(e_k,e_\ell)=0$, and according to Theorem \ref{torus-lifted},
\begin{align*}\Sec_\mu^\uparrow\big(e_{k},e_{\ell} \big)&=
-\frac{1}{4^n\, (4+|k|^2)\,(4+|\ell|^2)}
\sum_{\sigma,\tau,\sigma',\tau'\in\{-1,1\}^n} \langle k^\sigma,k^{\sigma'}\rangle\langle \ell^\tau,\ell^{\tau'}\rangle
    \\
  &\quad\cdot
\prod_{p=1}^n
\bigg[\one_{\sigma'_p=\sigma_p,\tau'_p=\tau_p}+\Big(\one_{k_p=-\ell_p\not=0}\one_{\sigma_p\tau_p=1}-\one_{k_p=\ell_p\not=0}\one_{\sigma_p\tau_p=-1}\Big)\,\one_{\sigma'_p=-\sigma_p,\tau'_p=-\tau_p}\\
&\qquad\quad-i\,\sigma_p\one_{k_p=0}\one_{\sigma'_p=-\sigma_p,\tau'_p=\tau_p}
-i\,\tau_p\one_{\ell_p=0}\one_{\sigma_p'=\sigma_p,\tau_p'=-\tau_p}
-\sigma_p\tau_p\one_{k_p=0,\ell_p=0}\, \one_{\sigma_p'=-\sigma_p,\tau_p'=-\tau_p}\bigg]\\
&=-\frac{1}{ (4+|k|^2)\,(4+|\ell|^2)}\, |k|^2\,|\ell|^2.
\end{align*}
\end{proof}

\subsection{Divergence}
For convenience, let us now restrict to the case $n=2$.
\begin{theorem}\label{tor-div} For all sufficiently large $C,L$ and  all $k=(i,j)\in\Z^2$ with $|i|,|j|\ge C$,
$$\sum_{\ell\in \Z^2, |\ell|\le L}\Sec_\mu(e_k,e_\ell)\ge \frac1{1000} \,C^2\,L^2.$$
\end{theorem}

\begin{proof} Recall that for generic $k=(i,j)$ and $\ell=(p,q)$,
\begin{align*}\Sec_\mu(e_k,e_\ell)&= \frac{1}{2^n}\,\sum_{\sigma} S(k^\sigma,\ell)
\end{align*}
with
\begin{align*}S(k,\ell)&:=\frac{1}{(4+|k|^2)\,(4+|\ell|^2)}
\bigg(-|k|^2\, |\ell|^2+\langle k^2,\ell^2\rangle
+\frac34\,
\bigg[| k-\ell|^2
-
\frac{\big(|k|^2-|\ell|^2
\big)^2}{4+|k+\ell|^2}
\bigg]
\, 
 \langle k,\ell\rangle^2 \bigg).\end{align*}
Then
\begin{align*}4S\big((i,j),(p,q)\big)&\cdot (4+i^2+j^2)(4+p^2+q^2)\\
&=
-4(iq-jp)^2+3\frac{(ip+jq)^2}{4+(i+p)^2+(j+q)^2}\big[4+(i-p)^2+(j-q)^2+(iq-jp)^2\big].
\end{align*}
Thus
with $S'\big((i,j),(p,q)\big):=\frac12 S\big((i,j),(p,q)\big)+\frac12 S\big((i,j),(-p,-q)\big)$, 
\begin{align*}4S'\big((i,j),(p,q)\big)&\cdot (4+i^2+j^2)(4+p^2+q^2)\\
&\ge
-4(iq-jp)^2+3(ip+jq)^2+ 3\frac{(ip+jq)^2(iq-jp)^2}{4+i^2+j^2+p^2+q^2}.
\end{align*}
Furthermore, with 
$S''\big((i,j),(p,q)\big):=\frac14 S'\big((i,j),(p,q)\big)+\frac14 S'\big((i,j),(p,-q)\big)+\frac14 S'\big((i,j),(q,p)\big)+\frac14 S'\big((i,j),(q,-p)\big)$
\begin{align*}4S''\big((i,j),(p,q)\big)&\cdot (4+i^2+j^2)(4+p^2+q^2)\\
&\ge- \frac12(i^2+j^2)(p^2+q^2)
+ 3\frac{(i^2-j^2)^2p^2q^2+i^2j^2(p^2-q^2)^2}{4+i^2+j^2+p^2+q^2}\\
&\ge -i^2j^2(p^2+q^2)+3\frac45\cdot i^2j^2\frac{(p^2-q^2)^2}{p^2+q^2} 
\end{align*}
for sufficiently large generic $p,q$,
and 
\begin{align*}4S''\big((i,j),(p,q)\big)&\cdot (4+i^2+j^2)(4+p^2+q^2)\\
&\ge -i^2j^2(p^2+q^2)+3\frac45\cdot  \frac49 i^2j^2\frac{(p+q)^4}{p^2+q^2} \ge\frac1{15}i^2j^2(p^2+q^2)
\end{align*}
if in addition
\begin{equation}\label{sector}
\Big| |p|-|q|\Big|\ge \frac23 \Big(|p|+|q|\Big).\end{equation}
The latter is equivalent to
$
\big|p/q\big|\ge 5$ or $\big|p/q\big|\le 1/5$,
and for sufficiently large $L\in\N$, 
$$\sharp\Big\{(p,q)\in \Z^2\cap[-L,L]^2: (p,q) \textrm{ satisfies }\eqref{sector}\Big\}\ge \frac16\cdot L^2.$$
Even more, $\sharp A_L\ge \frac16\cdot L^2$ for sufficiently large $L$
where
$$A_L:=\Big\{(p,q)\in \Z^2\cap[-L,L]^2: (p,q) \textrm{ satisfies \eqref{sector}, is generic,  and is sufficiently large}\Big\}.$$
Hence,
\begin{align*}
\sum_{\ell\in A_L}\Sec_\mu(e_k,e_\ell)&=
\sum_{(p,q)\in A_L}S''\big((i,j),(p,q)\big)\\
&\ge \sharp A_L\cdot 
\frac1{4(4+i^2+j^2)(4+p^2+q^2)}\frac1{15}i^2j^2(p^2+q^2)\\
&\ge\frac1{360}\, L^2\,\frac{p^2+q^2}{4+p^2+q^2}\cdot \frac{i^2j^2}{4+i^2+j^2}\\
&\ge\frac1{800} L^2\, C^2
\end{align*}
provided $|i|, |j|\ge C$. 
On the other hand, according to Theorem \ref{universal bounds} ,
\begin{align*}
\sum_{\ell\in [-L,L]^2\setminus A_L}\Sec_\mu(e_k,e_\ell)&\ge -(2L+1)^2\cdot \frac94 \ge -10\,L^2.
\end{align*}
Thus for sufficiently large $C$,
\begin{align*}
\sum_{\ell\in [-L,L]^2}\Sec_\mu(e_k,e_\ell)&\ge -(2L+1)^2\cdot \frac94 \ge \frac1{800} L^2\, C^2-10\,L^2\ge \frac1{1000} L^2\, C^2.
\end{align*}
\end{proof}

\begin{corollary}\label{NNC-torus} For $n\ge 2$, the geodesic space $(\Mea(\mathbb T^n), \HK)$ does not admit any upper or lower curvature bound in the sense of Alexandrov.

For $n=1$ it has nonnegative curvature in the sense of Alexandrov.
\end{corollary}

\begin{proof} 
Again according to Remark \ref{alex-rem} it only remains to verify that the curvature of $(\Mea(\mathbb T^n), \HK)$ in the sense of Alexandrov is neither nonnegative nor nonpositive. The former assertion is included in Theorem \ref{alex-lms}, the latter follows from Theorems \ref{alex-vs-sec} and  \ref{tor-div}.
\end{proof}

\begin{corollary}
For all $k=(i,j)\in\Z^2$ with sufficiently large $|i|,|j|$,
$$\mathsf{Ric}_\mu(e_{k}, e_{k}):=\|e_k\|_{H^1(\mu)}^2\cdot\lim_{L\to\infty}\sum_{\ell\in\in \Z^2, |\ell|\le L}\Sec_\mu(e_{k},e_{\ell})=+\infty.$$
\end{corollary}
 \subsection{Rescaling and Regularization}
 Already on the torus we see that ``typically'' the sectional curvature is not summable, that is,
 \begin{align*}\mathsf{Ric}_\mu(e_{k}, e_{k})&:=\|e_k\|_{H^1(\mu)}^2\cdot\sum_{\ell\in\in \Z^n\setminus\{k\}}\Sec_\mu(e_{k},e_{\ell})=+\infty\end{align*}
for most $k\in \Z^n$. To overcome this divergence, we can consider a rescaled version of the sectional curvature $\Sec_\mu(\varphi,\psi)$ where we replace the normalizing factor $\Lambda_\mu(\varphi,\psi)$ by its scaled version
$$\Lambda_{\mu,s}(\varphi,\psi):=\|\varphi\|^2_{H^{s}(\M,\mu)}\cdot  \|\psi\|^2_{H^{s}(\M,\mu)}-\langle \varphi,\psi\rangle^2_{H^{s}(\M,\mu)}$$
with $\|u\|_{H^{s}(\M,\mu)}:=\|(4-\Delta_\mu)^{s/2}u\|_{L^2(\mu)}$.
\begin{definition}
%
Let a Riemannian manifold $(\M,\g)$, a probability measure $\mu\in\Wass_2(\M)$,  an ON eigenbasis  $(\varphi_j)_{j\in\N_0}$ for the Laplacian on   $L^2(\M,\mu)$ with eigenvalues $(\lambda_j)_{j\in\N_0}$, and a number $s\ge0$
 be given.
 \begin{itemize}
 \item[(i)]
 Define the rescaled sectional curvature by 
 $$\Sec^{\Wass_2}_{\mu,s}(\varphi,\psi):=\frac{\Lambda_\mu(\varphi_{k},\varphi_{\ell})}{\Lambda_{\mu,s}(\varphi_{k},\varphi_{\ell})}\cdot
 \Sec^{\Wass_2}_\mu(\varphi,\psi)=
 \frac{1}{\lambda^{s-1}_k\,\lambda^{s-1}_\ell}\cdot
 \Sec^{\Wass_2}_\mu(\varphi,\psi)$$
and the rescaled Ricci curvature by 
$$\mathsf{Ric}^{\Wass_2}_{\mu,s}(\varphi_{j},\varphi_{j}):=\sum_{i=0, i\not=j}^\infty
\Sec^{\Wass_2}_{\mu,s}(\varphi_j,\varphi_i)$$
provided all the involved quantities exist and the sum converges. 

\item[(ii)]

Assume that the function
$$\Psi: z\mapsto 
\mathsf{Ric}_{\mu,z}(\varphi_{j},\varphi_{j}):=
\frac1{
(4+\lambda_j)^{z-1}} \sum_{k=0, k\not=j}^\infty \frac1{
(4+\lambda_k)^{z-1}}\,\Sec_{\mu}(\varphi_j,\varphi_k)
$$
is well-defined on $\{z\in{\mathbb C}: \Re(z)>L\}$ for some $L\in\R$ and admits a unique holomorphic extension to $z=1$. Then the $\zeta$-function regularization of the  Ricci curvature is defined as
$$\mathsf{Ric}_{\mu,\zeta}(\varphi_{j},\varphi_{j}):=\Psi(1).$$
\end{itemize}
\end{definition}

\begin{proposition} Let $\M$ be the $n$-dimensional torus, $\mu$ its normalized volume measure, and $(e_{k})_{k\in \Z^n}$ its eigenbasis for the Laplacian as considered above, and $\varphi_k:=(4+|k|^2)^{-1/2} e_k$. Then for every $k\in\Z^n$ and every $s>2+\frac n2$, the rescaled Ricci curvature
 $\mathsf{Ric}_{\mu,s}(\varphi_{k},\varphi_{k})$
 exists and is finite. 
 
 Furthermore, the defining series for $\Psi(z)= 
\mathsf{Ric}_{\mu,z}(\varphi_{j},\varphi_{j})$ converges absolutely and locally uniformly on   $\{z\in{\mathbb C}: \Re(z)>2+\frac n2\}$ and defines a holomorphic function there.
\end{proposition}

\begin{proof} By definition,
\begin{align*}
{\Sec_{\mu,s}(\varphi_{k},\varphi_{\ell})}&=
\frac1{(4+|k|^2)^{s-1}\cdot (4+|\ell|^2)^{s-1}}\cdot {\Sec_\mu(\varphi_{k},\varphi_{\ell})}.
\end{align*}
Together with the both-sided estimate
for $\Sec_\mu(\varphi_{k},\varphi_{\ell})=\Sec_\mu(e_{k},e_{\ell})$
 in Theorem \ref{universal bounds} this yields for each fixed $k\in\Z^n$ and any $s>1$,
\begin{align*}
\sum_{\ell\in\Z^n,\ell\not=k}\Big|\Sec_{\mu,s}(\varphi_{k},\varphi_{\ell})\Big|&\le
\sum_{\ell\in\Z^n,\ell\not=k}
\frac1{(4+|k|^2)^{s-1}\cdot (4+|\ell|^2)^{s-1}}\cdot
\Big|\Sec_\mu(\varphi_{k},\varphi_{\ell})\Big|\\
&\le  \frac32\, 4^n\,
\sum_{\ell\in\Z^n}
\frac1{(4+|k|^2)^{s-1}\cdot (4+|\ell|^2)^{s-1}}\cdot
 \Big(|k|^2+|\ell|^2\Big)\\
 &=C \, \sum_{\ell\in\Z^n}
\frac{|k|^2+|\ell|^2}{ (4+|\ell|^2)^{s-1}}.\end{align*}
The latter is finite if and only if
$s>2+\frac{n}2$.
\end{proof}

The lower bound on $s$ in the previous results is not sharp. 
We expect that convergence of the rescaled Ricci curvature 
$\mathsf{Ric}_{\mu,s}(\varphi_{k},\varphi_{k})$ holds for every $s>1+\frac n2$.

\subsection{Kantorovich-Wasserstein geometry for the torus}

\subsubsection{The setting} 

Now let us have a closer look on the Kantorovich-Wasserstein space $(\Wass_2(\mathbb T^n),\W_2)$ for the torus $\mathbb T^n$, $n\ge2$. As before, let $(e_k)_{k\in\Z^n}$ denote the Fourier basis of $L^2(\M,\mu)$ for $\M=\mathbb T^n$ and $\mu$ = normalized volume measure on $\M$.
Then the family $(e_k)_{k\in\Z^n\setminus\{0\}}$ is a complete orthogonal basis of 
$\T^{\sf ana}_\mu\Wass_2:=\overline{\big\{\nabla\varphi: \ \varphi\in \C^\infty_c(\M)\big\}}^{L^2(\mu)}$.
For $k\not=\ell$,
\begin{align*} \Lambda^\Wass_{k\ell}&:=|\nabla e_{k}|_{L^2(\mu)}^2\cdot|\nabla e_{\ell}|_{L^2(\mu)}^2-\langle \nabla e_{k},\nabla e_{\ell}\rangle_{L^2(\mu)}^2
=|k|^2\cdot|\ell|^2.
\end{align*}

\subsubsection{The sectional curvature}
\begin{theorem}\label{torus-Wass} For all $k,\ell \in \Z^n\setminus\{0\}$ with $k\not=\ell$,
\begin{align*}
\Sec_\mu^\Wass(e_k,e_\ell)
&=\frac{3}{2^{2n+2}}\,\frac1
 {|k|^2\cdot |\ell|^2}\\
 &\cdot\sum_{\sigma,\tau,\sigma',\tau'\in\{-1,1\}^n}
\bigg[\big\langle k^\sigma-\ell^\tau, k^{\sigma'}-\ell^{\tau'}\big\rangle
-\frac{(|k|^2-|\ell|^2)^2}{|k^\sigma+\ell^\tau|^2}
\bigg]\langle k^\sigma,\ell^\tau \rangle\, \langle k^{\sigma'},\ell^{\tau'} \rangle\\
&\qquad\cdot
\prod_{p=1}^n
\bigg[\one_{\sigma'_p=\sigma_p,\tau'_p=\tau_p}+\Big(\one_{k_p=-\ell_p\not=0}\one_{\sigma_p\tau_p=1}-\one_{k_p=\ell_p\not=0}\one_{\sigma_p\tau_p=-1}\Big)\,
\one_{\sigma'_p=-\sigma_p,\tau'_p=-\tau_p}\\
&\qquad\qquad-i\,\sigma_p\one_{k_p=0}\one_{\sigma'_p=-\sigma_p,\tau'_p=\tau_p}
-i\,\tau_p\one_{\ell_p=0}\one_{\sigma_p'=\sigma_p,\tau_p'=-\tau_p}
-\sigma_p\tau_p\one_{k_p=0,\ell_p=0}\, \one_{\sigma_p'=-\sigma_p,\tau_p'=-\tau_p}\bigg].
\end{align*}
\end{theorem}

\begin{corollary} In the generic case where $k_p\not=\pm\ell_p$  and $k_p\ell_p\not=0$ ($\forall p$),

\begin{align*}
\Sec_\mu^\Wass(e_k,e_\ell)
= \frac{1}{2^n}\,\sum_{\sigma} S^\Wass(k^\sigma,\ell)
\end{align*}
with
\begin{align*}S^\Wass(k,\ell)&:=\frac34\,
\bigg[| k-\ell|^2
-
\frac{\big(|k|^2-|\ell|^2
\big)^2}{|k+\ell|^2}
\bigg]
\, \frac{
 \langle k,\ell\rangle^2}
 {|k|^2\, |\ell|^2}\\
 &=3\,\frac{ \langle k,\ell\rangle^2}{|k+\ell|^2}\, 
 \bigg[1-
 \frac{
 \langle k,\ell\rangle^2}
 {|k|^2\, |\ell|^2}\bigg].
\end{align*}
\end{corollary}

\subsubsection{Some examples}

\begin{example} Assume $n=1$.
Then for all $k,\ell\in\Z\setminus\{0\}$, $k\not=\pm\ell$,
$$\Sec^\Wass_\mu(e_k,e_\ell)=0.$$
\end{example}

\begin{example} Assume $k_p\ell_p=0$ ($\forall p=1,\ldots,n$). Then
$$\Sec^\Wass_\mu(e_k,e_\ell)=0.$$
\end{example}

\subsubsection{Divergence}
\begin{theorem} Assume $n=2$. Then for all  $k=(i,j)\in\Z^2\setminus\{(0,0)\}$,
\begin{align*}\mathsf{Ric}^\Wass_\mu(e_{k}, e_{k})&:=\|\nabla\varphi_k\|_{L^2}^2\cdot\sum_{\ell\in\in \Z^2\setminus\{0,k\}}\Sec^\Wass_\mu(e_{k},e_{\ell})=+\infty\end{align*}
\end{theorem}

\begin{proof}
(i) Let us first consider the nondegenerate case where $ij\not=0$. 
We derive a lower estimate for 
$$\sum_{p,q\in\Z\setminus\{0\}, p\not=\pm i, q\not=\pm j}\Sec^\Wass_\mu(e_{i,j},e_{p,q})=
\sum_{p,q\in\Z\setminus\{0\}, p\not=\pm i, q\not=\pm j}S^\Wass\big((i,j),(p,q)\big)$$
where
\begin{align*}S^\Wass\big((i,j),(p,q)\big)
 &=3\,\frac{(ip+jq)^2}{(i+p)^2+(j+q)^2}\, 
 \bigg[1-
 \frac{
 (ip+jq)^2}
 {(i^2+j^2)(p^2+q^2)}\bigg]\\
 &=
 3\,\frac{(ip+jq)^2\, (iq-jp)^2}{[(i+p)^2+(j+q)^2]\,(i^2+j^2)(p^2+q^2)}.
\end{align*}
Consider
$$
S'\big((i,j),(p,q)\big):=\frac12S^\Wass\big((i,j),(-p,-q)\big)
+\frac12S^\Wass\big((i,j),(p,q)\big)
$$ and 
$$S''\big((i,j),(p,q)\big):=\frac12S'\big((i,j),(p,q)\big)
+\frac12S'\big((i,j),(p,-q)\big).$$
Then
$$S'\big((i,j),(p,q)\big)\ge 3\,\frac{(ip+jq)^2\, (iq-jp)^2}{[i^2+j^2+p^2+q^2]\,(i^2+j^2)(p^2+q^2)}$$ and
$$S''\big((i,j),(p,q)\big)\ge 3\,\frac{(i^2-j^2)^2p^2q^2+i^2j^2(p^2-q^2)^2
}{[i^2+j^2+p^2+q^2]\,(i^2+j^2)(p^2+q^2)}.$$
Thus
\begin{align*}
\sum_{\ell\in\in \Z^2\setminus\{0,k\}}\Sec^\Wass_\mu(e_{k},e_{\ell})&\ge
\sum_{p,q\in\Z\setminus\{0\}, p\not=\pm i, q\not=\pm j}S^\Wass\big((i,j),(p,q)\big)\\
&=\sum_{p,q\in\Z\setminus\{0\}, p\not=\pm i, q\not=\pm j}S''\big((i,j),(p,q)\big)\\
&\ge \sum_{p,q\in\Z\setminus\{0\}, p\not=\pm i, q\not=\pm j}3\,\frac{(i^2-j^2)^2p^2q^2+i^2j^2(p^2-q^2)^2
}{[i^2+j^2+p^2+q^2]\,(i^2+j^2)(p^2+q^2)}\\
&=+\infty.
\end{align*}

\medskip

(ii) Now let us consider the degenerate case where $ij=0$ but $(i,j)\not=(0,0)$.  Without restriction, assume $i=0, j\not=0$. 
According to Theorem \ref{torus-Wass}, for $k=(0,j)$ and $\ell=(p,q)$ with $p,q,j\not=0$ and $q\not=\pm j$,
\begin{align*}
\Sec_\mu^\Wass(e_k,e_\ell)
&=\frac{3}{2^{6}}\,\frac{1}
 {j^2(p^2+q^2)}\\
 &\cdot\sum_{\sigma,\tau,\sigma',\tau'\in\{-1,1\}^2}
\bigg[\tau_1\tau'_1p^2+\sigma_2\sigma_2'j^2+\tau_2\tau_2'q^2+(\sigma_2\tau_2'+\tau_2\sigma_2')jq\\
&\qquad\qquad \qquad\qquad
-\frac{(p^2+q^2-j^2)^2}{p^2+(\sigma_2j+\tau_2q)^2}
\bigg]
 \sigma_2\tau_2\sigma_2'\tau_2'\, j^2q^2\\
&\qquad\qquad\cdot
 \Big[\one_{\sigma'_1=\sigma_1,\tau'_1=\tau_1}-i\sigma_1 \one_{\sigma'_1=-\sigma_1,\tau'_1=-\tau_1}\Big]\,
 \one_{\sigma'_2=\sigma_2,\tau'_2=\tau_2}\\
 &=\frac{3}{4}\,\frac{q^2}
 {p^2+q^2}\,\bigg[p^2+q^2+j^2-\frac12\sum_{\sigma_2\in\{-1,1\}}\frac{(p^2+q^2-j^2)^2}{p^2+(\sigma_2j+q)^2}\Bigg]\\
&=3\frac{p^2q^2j^2}{p^2+q^2}\cdot\frac{p^2+q^2+j^2}{(p^2+q^2+j^2)^2-4q^2j^2}.
\end{align*}
Thus with $\ell_m:=(mj,mj)$,
\begin{align*}
\Sec_\mu^\Wass(e_k,e_\ell)
&\ge \frac38j^2\cdot \frac{4m^4+2m^2}{4m^4+1}\ge \frac38j^2
\end{align*}
and 
\begin{align*}
\sum_{\ell\in\in \Z^2\setminus\{0,k\}}\Sec^\Wass_\mu(e_{k},e_{\ell})\ge
\sum_{m=2}^\infty\Sec^\Wass_\mu(e_{k},e_{\ell_m})=+\infty.
\end{align*}
\end{proof}

\subsubsection{Rescaling and Regularization}

Similarly as for $(\Mea,\HK)$, the sectional curvature for $(\Wass_2,\W_2)$   is not summable, that is,
 \begin{align*}\mathsf{Ric}^{\Wass_2}_\mu(e_{k}, e_{k})&:=\|\nabla e_k\|_{L^2(\mu)}^2\cdot\sum_{\ell\in\in \Z^n\setminus\{k\}}\Sec^{\Wass_2}_\mu(e_{k},e_{\ell})=+\infty\end{align*}
for all $k\in \Z^n\setminus\{0\}$. To overcome this divergence, we again consider rescaled versions of  sectional and Ricci curvature. 
\begin{definition}
Let a Riemannian manifold $(\M,\g)$, a probability measure $\mu\in\Wass_2(\M)$,  an ON eigenbasis  $(\varphi_j)_{j\in\N_0}$ for the Laplacian on   $L^2(\M,\mu)$ with eigenvalues $(\lambda_j)_{j\in\N_0}$, and a number $s\ge0$
 be given.
 Define the rescaled sectional curvature by 
 $$\Sec^{\Wass_2}_{\mu,s}(\varphi,\psi):=\frac{\Lambda_\mu(\varphi_{k},\varphi_{\ell})}{\Lambda_{\mu,s}(\varphi_{k},\varphi_{\ell})}\cdot
 \Sec^{\Wass_2}_\mu(\varphi,\psi)=
 \frac{1}{\lambda^{s-1}_k\,\lambda^{s-1}_\ell}\cdot
 \Sec^{\Wass_2}_\mu(\varphi,\psi)$$
and the rescaled Ricci curvature by 
$$\mathsf{Ric}^{\Wass_2}_{\mu,s}(\varphi_{j},\varphi_{j}):=\sum_{i=1, i\not=j}^\infty
\Sec^{\Wass_2}_{\mu,s}(\varphi_j,\varphi_i)$$
provided all the involved quantities exist and the sum converges. 

Assume that the function
$$\Psi: z\mapsto 
\mathsf{Ric}^{\Wass_2}_{\mu,z}(\varphi_{j},\varphi_{j}):=
\frac1{
\lambda_j^{z-1}} \sum_{k=1, k\not=j}^\infty \frac1{
\lambda_k^{z-1}}\,\Sec_{\mu}(\varphi_j,\varphi_k)
$$
is well-defined on $\{z\in{\mathbb C}: \Re(z)>L\}$ for some $L\in\R$ and admits a unique holomorphic extension to $z=1$. Then the $\zeta$-regularized  Ricci curvature is defined as
$$\mathsf{Ric}^{\Wass_2}_{\mu,\zeta}(\varphi_{j},\varphi_{j}):=\Psi(1).$$
\end{definition}

\begin{proposition} Let $\M$ be the $n$-dimensional torus and $\mu$ its normalized volume measure. Then for every $k\in\Z^n\setminus\{0\}$ and every $s>2+\frac n2$, the rescaled Ricci curvature
 $\mathsf{Ric}_{\mu,s}(\varphi_{k},\varphi_{k})$
 exists and is finite. 
 
 More generally,  the sum defining 
 $\mathsf{Ric}_{\mu,z}(\varphi_{k},\varphi_{k})$ converges absolutely and locally uniformly in $\{z\in\mathbb C: \Re(z)>2+\frac n2\}$.

\end{proposition}

\begin{remark} Similar calculations have been carried out and analogous observations have been made  in an unpublished work \cite{Ricci-on-Wasserstein} for the sectional curvature 
of the space $(\Wass_2(\R^n),\W_2)$ at the point $\mu=\mathcal N(0,1)$, the standard Gaussian measure on $\R^n$.
\end{remark}

\def\cprime{$'$} \def\cprime{$'$} \providecommand{\MR}[1]{}
\providecommand{\bysame}{\leavevmode\hbox to3em{\hrulefill}\thinspace}
\providecommand{\MR}{\relax\ifhmode\unskip\space\fi MR }
\providecommand{\MRhref}[2]{%
  \href{http://www.ams.org/mathscinet-getitem?mr=#1}{#2}
}
\providecommand{\href}[2]{#2}

%

\begin{thebibliography}{CPSV18b}

\bibitem[BBI01]{BBI}
D.~Burago, I.U.D. Burago, and S.~Ivanov, \emph{A course in metric geometry},
  Graduate Studies in Mathematics, American Mathematical Society, 2001.

\bibitem[Bis23]{Bisson}
Angelina Bisson, \emph{An exposition of the curvature of warped product
  manifolds}, CSUSB Scholar works, https://scholarworks.lib.csusb.edu/etd/1810
  (2023).

\bibitem[CPSV18a]{CPSV1}
Lenaic Chizat, Gabriel Peyr{\'e}, Bernhard Schmitzer, and Fran{{c}}ois-Xavier
  Vialard, \emph{An interpolating distance between optimal transport and
  {F}isher--{R}ao metrics}, Foundations of Computational Mathematics
  \textbf{18} (2018), no.~1, 1--44.

\bibitem[CPSV18b]{CPSV2}
\bysame, \emph{Unbalanced optimal transport: Dynamic and {K}antorovich
  formulations}, Journal of Functional Analysis \textbf{274} (2018), no.~11,
  3090--3123.

\bibitem[DPST25]{DST-infimal}
Nicol{\`o} De~Ponti, Giacomo~Enrico Sodini, and Luca Tamanini, \emph{The
  infimal convolution structure of the {H}ellinger-{K}antorovich distance},
  arXiv preprint arXiv:2503.12939 (2025).

\bibitem[DS25]{schiavo2025hellinger}
Lorenzo {Dello Schiavo} and Giacomo~Enrico Sodini, \emph{The
  {H}ellinger-{K}antorovich metric measure geometry on spaces of measures},
  arXiv preprint arXiv:2503.07802 (2025).

\bibitem[GGV25]{gallouet2025regularity}
Thomas Gallou{\"e}t, Roberta Ghezzi, and Francois-Xavier Vialard,
  \emph{Regularity theory and geometry of unbalanced optimal transport},
  Journal of Functional Analysis \textbf{289} (2025), no.~7, 111042.

\bibitem[Gig11]{Gigli11}
Nicola Gigli, \emph{On the inverse implication of {B}renier-{M}c{C}ann theorems
  and the structure of {$(\mathscr P_2(M),W_2)$}}, Methods Appl. Anal.
  \textbf{18} (2011), no.~2, 127--158. \MR{2847481 (2012h:49090)}

\bibitem[Gig12]{Giglimemo}
\bysame, \emph{Second order analysis on {$(\mathcal P_2(M),W_2)$}}, Mem. Amer.
  Math. Soc. \textbf{216} (2012), no.~1018, xii+154. \MR{2920736}

\bibitem[GS18]{Ricci-on-Wasserstein}
Maria Gordina and Karl-Theodor Sturm, \emph{Sectional and {R}icci curvature on
  the {W}asserstein space}, Unpublished preprint (2018).

\bibitem[GV18]{gallouet2018camassa}
Thomas Gallou{\"e}t and Francois-Xavier Vialard, \emph{The {C}amassa--{H}olm
  equation as an incompressible {E}uler equation: A geometric point of view},
  Journal of Differential Equations \textbf{264} (2018), no.~7, 4199--4234.

\bibitem[KMV16]{kondratyev2016new}
Stanislav Kondratyev, L{\'e}onard Monsaingeon, and Dmitry Vorotnikov, \emph{A
  new optimal transport distance on the space of finite {R}adon measures},
  Advances in Differential Equations \textbf{21} (2016), no.~11-12, 1117--1164.

\bibitem[KV19]{kondratyev2019spherical}
Stanislav Kondratyev and Dmitry Vorotnikov, \emph{Spherical
  {Hellinger}--{K}antorovich gradient flows}, SIAM Journal on Mathematical
  Analysis \textbf{51} (2019), no.~3, 2053--2084.

\bibitem[LM26]{laschos2026evolutionary}
Vaios Laschos and Alexander Mielke, \emph{Evolutionary variational inequalities
  on the {H}ellinger-{K}antorovich and spherical {H}ellinger-{K}antorovich
  spaces}, Communications in Partial Differential Equations (2026), 1--43.

\bibitem[LMS16]{LMS-optimal}
Matthias Liero, Alexander Mielke, and Giuseppe Savar{\'e}, \emph{Optimal
  transport in competition with reaction: The {H}ellinger--{K}antorovich
  distance and geodesic curves}, SIAM Journal on Mathematical Analysis
  \textbf{48} (2016), no.~4, 2869--2911.

\bibitem[LMS18]{LMS-inventiones}
\bysame, \emph{Optimal entropy-transport problems and a new
  {H}ellinger--{K}antorovich distance between positive measures}, Inventiones
  mathematicae \textbf{211} (2018), no.~3, 969--1117.

\bibitem[LMS23]{LMS-fine}
\bysame, \emph{Fine properties of geodesics and geodesic $\lambda$-convexity
  for the {H}ellinger--{K}antorovich distance}, Archive for Rational Mechanics
  and Analysis \textbf{247} (2023), no.~6, 112.

\bibitem[Lot08]{Lott07}
John Lott, \emph{Some geometric calculations on {W}asserstein space}, Comm.
  Math. Phys. \textbf{277} (2008), no.~2, 423--437. \MR{2358290 (2009b:58014)}

\bibitem[Lot17]{lott2017tangent}
\bysame, \emph{On tangent cones in {W}asserstein space}, Proceedings of the
  American Mathematical Society \textbf{145} (2017), no.~7, 3127--3136.

\bibitem[O'N83]{ONeill}
Barrett O'Neill, \emph{Semi-riemannian geometry with applications to
  relativity}, vol. 103, Academic press, 1983.

\bibitem[Ott01]{Ot01}
F.~Otto, \emph{The geometry of dissipative evolution equations: the porous
  medium equation}, Comm. Partial Differential Equations \textbf{26} (2001),
  no.~1-2, 101--174. \MR{1842429 (2002j:35180)}

\end{thebibliography}
%
\end{document}